\documentclass[12pt]{article}
\usepackage{xcolor}
\usepackage{amsmath,amsthm,amsfonts,amssymb,amscd,caption,color,subcaption,cite}
\usepackage[title]{appendix}
\usepackage{graphicx}
\usepackage[mathscr]{eucal}
\allowdisplaybreaks[4]
\usepackage[figurename=Fig.]{caption}
\numberwithin{equation}{section}
\allowdisplaybreaks[4]

\newtheorem {Lemma}{Lemma}[section]
\newtheorem {Theorem} {Theorem}[section]
\newtheorem {Corollary}{Corollary}[section]

\newtheorem {Claim} {Claim}[section]
\newtheorem {Conjecture} {Conjecture}[section]
\usepackage{tikz}
\usetikzlibrary{positioning}
\usepackage{xcolor}
\usepackage{circuitikz}
\usetikzlibrary{patterns}
\usetikzlibrary{decorations.pathreplacing}
\usepackage[cm]{fullpage}
\begin{document}

\title{Tight spectral conditions for the Hamiltonicity of $K_{1,r}$-free split graphs}

\author{Yiting Cai $^{a}$\footnote{E-mail: yitingcai@m.scnu.edu.cn},
Haiyan Guo $^{b}$\footnote{E-mail: ghaiyan0705@m.scnu.edu.cn}, 
Hong-Jian Lai $^{c,d}$\footnote{E-mail: hjlai2015@hotmail.com}, 
Bo Zhou $^{a}$\footnote{%Corresponding author; 
E-mail: zhoubo@scnu.edu.cn}\\
$^{a}$ School of  Mathematical Sciences, South China Normal University,\\
Guangzhou 510631, P.R. China\\
 $^{b}$School of Computer Science, South China Normal University, \\ 
Guangzhou 510631, P.R. China\\
$^{c}$ School of Mathematics and Systems Science,
Guangdong Polytechnic Normal University,\\ Guangzhou 510665, P.R. China\\
$^{d}$ Department of Mathematics, West Virginia University, \\ Morgantown, WV 26506, USA}

\date{}
\maketitle

\begin{abstract}
The Hamiltonicity and related subjects of split graphs, and in particular $K_{1,r}$-free split graphs with $r\ge 3$ received much attention. Dai et al. [Discrete Math. 345 (2022) 112826]
conjectured that every $(r-1)$-connected $K_{1,r}$-free split graph is Hamiltonian. They proved the case when $r=4$, and earlier Renjith and Sadagopan [Int. J. Found. Comput. Sci. 33 (2022) 1--32] proved the case
when $r=3$. Recently, Liu, Song, Zhang and Lai [Discrete Math. 346 (2023) 113402]
proved that a split graph  is Hamiltonian if and
only if it is fully cycle extendable. So  for $r=3,4$ every $(r-1)$-connected $K_{1,r}$-free split graph is fully cycle extendable. We give tight spectral sufficient conditions for a  $K_{1,r}$-free split graph to be Hamiltonian for $r=3,4$.  
%
%Providing tight spectral conditions for specific graph classes (such as split graphs)  offers efficient sufficient criteria for determining Hamiltonicity in such graphs, and thus falls within the scope of algorithm design and analysis. At the same time, the Hamiltonian problem is a central problem in combinatorial optimization. The paper provides tight sufficient conditions via spectral radius, which are typical results in discrete optimization.
\end{abstract}

\section{Introduction}

In this paper, we consider only finite simple graphs.
A graph $G$ is  Hamiltonian if it contains a Hamiltonian cycle, i.e., a cycle contains all vertices of $G$. A graph is Hamilton-connected if for every pair of distinct vertices $u, v$, there exists a Hamiltonian path (a path containing all vertices) from $u$ to $v$.
A graph  is fully cycle extendable if every vertex  lies on a cycle of length $3$ and for every non-hamiltonian cycle $C$ there is a cycle $C'$  such that $|V(C')| =|V(C)|+1$ and $V(C)\subset V(C')$.

A graph $G$ is a split graph if the vertex set $V(G)$ can be partitioned into a clique $K$ and an independent set $I$, either of which
may be empty. The clique can be chosen to be a maximum clique, and in such case, we say that $G$ is of type $(K,I)$, and write $G=(K,I)$ if there is no confusion.
Split graphs were introduced by Foldes
and Hammer \cite{FH} in 1977, and were studied further in \cite{BH,FH1,KLM,M,RS,TH}.
A split graph $G=(K,I)$ is said to be complete if any vertex of $K$ is adjacent all vertices of $I$.
An obvious necessary condition that  a  split graph $G=(K,I)$ is Hamiltonian
is $|K|\ge |I|$.

Let $r$ be an integer with $r\ge 3$. A graph $G$  is $K_{1,r}$-free if $G$ does not have an induced subgraph isomorphic to $K_{1,r}$, that is,
 no vertex of $G$ has $r$ pairwise nonadjacent neighbors. A $K_{1,3}$-free graph is known as a
claw-free graph. The Hamiltonicity of claw-free graphs received much attention, see, e.g. \cite{Chen1,Chen2,Chen3,Chen4}.
 Renjith and Sadagopan \cite{RS} and Dai, Zhang, Broerama and Zhang \cite{DZ} showed the following results.

\begin{Theorem} \cite{RS} \label{Th1}
Let $G=(K,I)$ be a claw-free split graph. Then $G$ is Hamiltonian if and only if $G$ is $2$-connected.
\end{Theorem}

\begin{Theorem} \cite{DZ} \label{Th2}
Every $3$-connected $K_{1,4}$-free split graph is Hamiltonian.
\end{Theorem}

After proving Theorem \ref{Th2},  Dai, Zhang, Broerama and Zhang \cite{DZ} proposed the following conjecture, which is true for $r=3,4$ (by Theorems \ref{Th1} and \ref{Th2}).

\begin{Conjecture} \cite{DZ}
Every $(r-1)$-connected $K_{1,r}$-free split graph is Hamiltonian.
\end{Conjecture}

Liu, Song, Zhang and Lai \cite{LSZL} studied the Hamilton-connectedness and fully cycle extendability of
$K_{1,r}$-free split graphs. Let $r=3,4$.  They showed that
a $r$-connected $K_{1,r}$-free split graph is Hamilton-connected, and as an immediate consequence of the claim that a split graph is Hamiltonian if and
only if it is fully cycle extendable, every $(r-1)$-connected $K_{1,r}$-free split graph is fully cycle extendable.

The spectral radius $\rho(M)$ of a square (complex) matrix $M$  the largest modulus of the eigenvalues of $M$. 
Let $A(G)$ be the adjacency mattrix of $G$, where  $A(G)=(a_{uv})_{u,v\in V(G)}$ with   $a_{uv}=1$ if $u$ and $v$ are adjacent and $0$ otherwise.
The spectral radius of a graph $G$ is defined as $\rho(G)=\rho(A(G))$.

For integers $n$ and $t$ with $1\le t \le \frac{n}{2}$ and $n\ge 4$, let $G_{n,t}$  be the graph obtained from $K_{n-t}\cup tK_1$  $V(tK_1)=\{u_1,\dots, u_t\}$ and $V(K_{n-t})=\{v_1,\dots, v_{n-t}\}$ by adding edges $u_iv_i$ for all $i=1,\dots, \max\{t-1,1\}$ and $u_tv_j$ for all $j=t,\dots, n-t$ and $t\ge2$, see Fig. \ref{f1}.

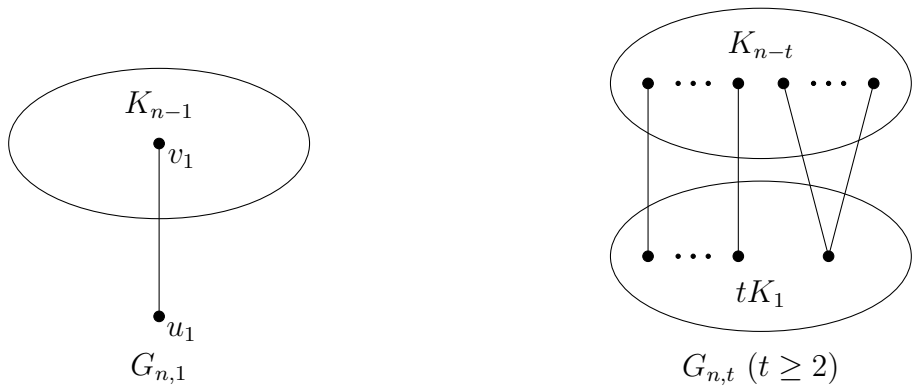
\begin{figure}[htbp]
\centering
\begin{tikzpicture}
\draw (2,-0.8) ellipse (2 and 1);
\filldraw [black] (2,-0.8) circle (2pt);
\draw  [black](2,-0.8)--(2,-3.1);
\filldraw [black] (2,-3.1) circle (2pt);
\node at (2, -0.3) {$K_{n-1}$};
\node at (2, -3.8) {$G_{n,1}$};
\node at (2.3, -3.3) {$u_1$};
\node at (2.3, -1.0) {$v_1$};
\draw (10,0) ellipse (2 and 1);
\draw (10,-2.3) ellipse (2 and 1);
\filldraw [black] (8.5,0) circle (2pt);
\filldraw [black] (8.9,0) circle (0.8pt);
\filldraw [black] (9.1,0) circle (0.8pt);
\filldraw [black] (9.3,0) circle (0.8pt);
\filldraw [black] (9.7,0) circle (2pt);
\filldraw [black] (10.3,0) circle (2pt);
\filldraw [black] (10.7,0) circle (0.8pt);
\filldraw [black] (10.9,0) circle (0.8pt);
\filldraw [black] (11.1,0) circle (0.8pt);
\filldraw [black] (11.5,0) circle (2pt);

\filldraw [black] (10.9,-2.3) circle (2pt);
\draw [black] (10.9,-2.3)--(11.5,0);
\draw [black] (10.9,-2.3)--(10.3,0);
\filldraw [black] (8.5,-2.3) circle (2pt);
\draw [black] (8.5,-2.3)--(8.5,0);
\filldraw [black] (8.9,-2.3) circle (0.8pt);
\filldraw [black] (9.1,-2.3) circle (0.8pt);
\filldraw [black] (9.3,-2.3) circle (0.8pt);
\filldraw [black] (9.7,-2.3) circle (2pt);
\draw [black] (9.7,-2.3)--(9.7,0);
%\draw[decorate,decoration={brace,raise=8pt,amplitude=0.2cm},black] (10.3,-0.1) -- (11.5,-0.1);
%\node at (10.9, 0.5) {$t-1$};
\node at (10, -2.8) {$tK_1$};
\node at (10, 0.5) {$K_{n-t}$};
\node at (10, -3.8) {$G_{n,t}~(t\ge2)$};
\end{tikzpicture}
\caption{Graphs $G_{n,t}$.}
\label{f1}
\end{figure}

\begin{Theorem}\label{N1}
Let $G=(K,I)$ be a connected claw-free split graph of order $n$, where $n\ge \max\{4, 2|I|\}$. If
\[
\rho(G)\ge \rho(G_{n,|I|}), 
\]
then $G$ is Hamiltonian, unless $G\cong G_{n,|I|}$.
\end{Theorem}

An immediate consequence of the previous theorems is as follows.

\begin{Corollary}
Let $G=(K,I)$ be a connected claw-free split graph of order $n$, where $n\ge \max\{4, 2|I|\}$. If
\[
\rho(G)\ge \rho(G_{n,|I|}), 
\]
then $G$ is $2$-connected, unless $G\cong G_{n,|I|}$.
\end{Corollary}

To state the result  for $K_{1,4}$-free split graphs, we need further notation.

First, we define the graphs $\Gamma_{n,t}$ for $1\le t\le \frac{n}{2}$ and $n\ge 5$.
Let $\Gamma_{n,1}=G_{n,1}$. For $t\ge 2$, we define  a split graph $\Gamma_{n,t}=(K,I)$ with $I=\{u_1,\dots, u_t\}$ and $K=\{v_1,\dots, v_{n-t}\}$ as follows.
In  $\Gamma_{n,2}$, the edges between $K$ and $I$ are $u_1v_i$ for all $i=1,\dots,n-3$ and $u_2v_{n-3}$.
In  $\Gamma_{n,3}$, the edges between $K$ and $I$ are $u_iv_j$ for all $i=1,2$ and $j=1,\dots, n-4$ and $u_3v_{n-3}$.
In  $\Gamma_{n,4}$, the edges between $K$ and $I$ are $u_iv_j$ for all $i=1,2$ and  $j=1,\dots, n-5$, $u_3v_k$ for all $k=2,\dots, n-4$, and $u_4v_{n-4}$.
%In $\Gamma_{n,4}'$, the edges between $K$ and $I$  are $u_1v_1$, $u_2v_i$ for all $i=1,\dots, n-5$, $u_3v_j$ for all $j=1,\dots, n-6, n-4$, and $u_4v_k$ for all $k=2,\dots,n-4$.
In $\Gamma_{n, 5}$, the edges between $K$ and $I$ are $u_1v_1$, $u_1v_{n-5}$, $u_iv_j$ for all $i=2,3$ and $j=1,\dots, n-6$, $u_4v_k$ for all $k=2,\dots, n-5$, and $u_5v_{n-5}$.
For  $6\le t\le \lfloor\frac{n}{2}\rfloor$,
in $\Gamma_{n,t}$, the edges between $K$ and $I$  are $u_iv_i$ for all $i=1,\dots,t-5$,
$u_{t-4}v_j$ for all $j=1, \dots, n-t-1$,
$u_{t-3}v_k$ for all $k=1, \dots, n-t-2, n-t$,
$u_{t-2}v_\ell$ for all $\ell=t-4,\dots, n-t$,
$u_{t-1}v_{n-t-1}$ and $u_tv_{n-t}$.
See Fig. \ref{f3} for $\Gamma_{n,t}$ with $t\ge 2$.

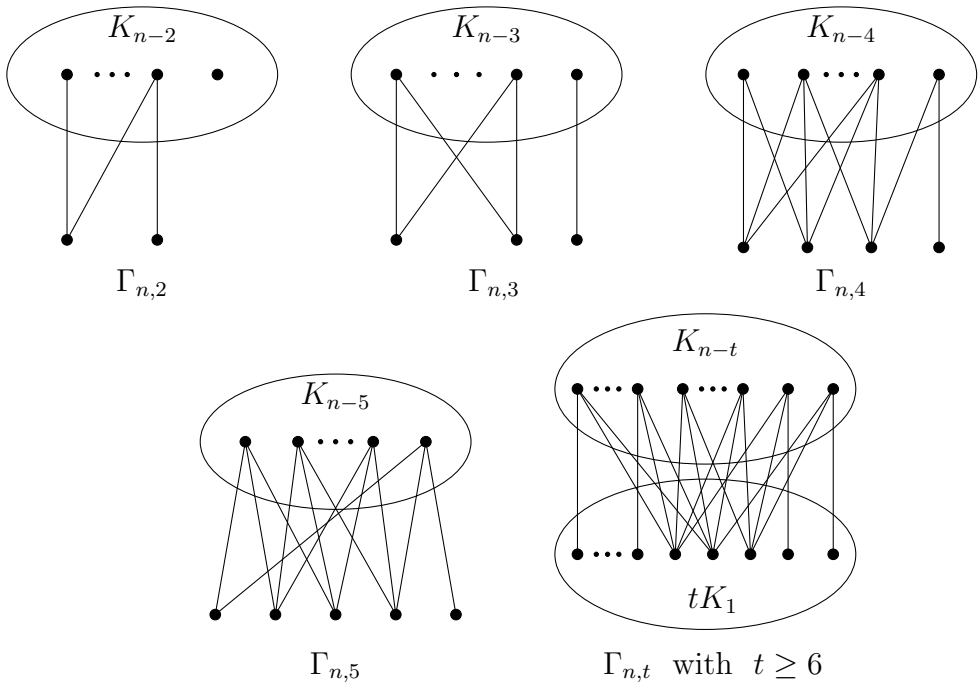
\begin{figure}[htbp]
\centering
\begin{tikzpicture}
\draw (4,-0.8) ellipse (1.8 and 0.9);
\filldraw [black] (3,-0.8) circle (2pt);
\filldraw [black] (3.4,-0.8) circle (0.8pt);
\filldraw [black] (3.6,-0.8) circle (0.8pt);
\filldraw [black] (3.8,-0.8) circle (0.8pt);
\filldraw [black] (4.2,-0.8) circle (2pt);
\filldraw [black] (5,-0.8) circle (2pt);
\draw  [black](3,-3)--(3,-0.8);
\draw  [black](3,-3)--(4.2,-0.8);
\filldraw [black] (3,-3) circle (2pt);
\filldraw [black] (4.2,-3) circle (2pt);
\draw  [black](4.2,-3)--(4.2,-0.8);
\node at (4, -0.2) {$K_{n-2}$};
\node at (4, -3.6) {$\Gamma_{n, 2}$};
\end{tikzpicture} \ \ \ \ \ \
\begin{tikzpicture}
\draw (12,-0.8) ellipse (1.8 and 0.9);
\filldraw [black] (10.8,-0.8) circle (2pt);
\filldraw [black] (11.3,-0.8) circle (0.8pt);
\filldraw [black] (11.6,-0.8) circle (0.8pt);
\filldraw [black] (11.9,-0.8) circle (0.8pt);
\filldraw [black] (12.4,-0.8) circle (2pt);
\filldraw [black] (13.2,-0.8) circle (2pt);
\filldraw [black] (13.2,-3) circle (2pt);
\draw  [black](13.2,-3)--(13.2,-0.8);
\filldraw [black] (12.4,-3) circle (2pt);
\draw  [black](12.4,-3)--(12.4,-0.8);
\draw  [black](12.4,-3)--(10.8,-0.8);
\filldraw [black] (10.8,-3) circle (2pt);
\draw  [black](10.8,-3)--(12.4,-0.8);
\draw  [black](10.8,-3)--(10.8,-0.8);
\node at (12.1, -3.6) {$\Gamma_{n,3}$};
\node at (12, -0.2) {$K_{n-3}$};
\end{tikzpicture}\ \ \ \ \ \ \ \
\begin{tikzpicture}
\draw (7,-0.8) ellipse (1.8 and 0.9);
\filldraw [black] (5.7,-0.8) circle (2pt);
\filldraw [black] (6.5,-0.8) circle (2pt);
\filldraw [black] (6.8,-0.8) circle (0.8pt);
\filldraw [black] (7.0,-0.8) circle (0.8pt);
\filldraw [black] (7.2,-0.8) circle (0.8pt);
\filldraw [black] (7.5,-0.8) circle (2pt);
\filldraw [black] (8.3,-0.8) circle (2pt);
\filldraw [black] (5.7,-3.1) circle (2pt);
\draw  [black](5.7,-3.1)--(5.7,-0.8);
\draw  [black](5.7,-3.1)--(6.5,-0.8);
\draw  [black](5.7,-3.1)--(7.5,-0.8);
\filldraw [black] (6.55,-3.1) circle (2pt);
\draw  [black](6.55,-3.1)--(5.7,-0.8);
\draw  [black](6.55,-3.1)--(6.5,-0.8);
\draw  [black](6.55,-3.1)--(7.5,-0.8);
\filldraw [black] (7.4,-3.1) circle (2pt);
\draw  [black](7.4,-3.1)--(6.5,-0.8);
\draw  [black](7.4,-3.1)--(7.5,-0.8);
\draw  [black](7.4,-3.1)--(8.3,-0.8);
\filldraw [black] (8.3,-3.1) circle (2pt);
\draw  [black](8.3,-3.1)--(8.3,-0.8);
\node at (7, -0.2) {$K_{n-4}$};
\node at (7, -3.6) {$\Gamma_{n,4}$};
\end{tikzpicture}\ \ \ \ \ \ \ \

\begin{tikzpicture}
\draw (8,-0.8) ellipse (1.8 and 0.9);
\filldraw [black] (6.8,-0.8) circle (2pt);
\filldraw [black] (7.5,-0.8) circle (2pt);
\filldraw [black] (7.8,-0.8) circle (0.8pt);
\filldraw [black] (8.0,-0.8) circle (0.8pt);
\filldraw [black] (8.2,-0.8) circle (0.8pt);
\filldraw [black] (8.5,-0.8) circle (2pt);
\filldraw [black] (9.2,-0.8) circle (2pt);
%\draw (8,-0.8) ellipse (1.8 and 0.9);
%\filldraw [black] (6.5,-0.8) circle (2pt);
%\filldraw [black] (7.1,-0.8) circle (2pt);
%\filldraw [black] (7.5,-0.8) circle (0.8pt);
%\filldraw [black] (7.7,-0.8) circle (0.8pt);
%\filldraw [black] (7.9,-0.8) circle (0.8pt);
%\filldraw [black] (8.1,-0.8) circle (0.8pt);
%\filldraw [black] (8.3,-0.8) circle (0.8pt);
%\filldraw [black] (8.5,-0.8) circle (0.8pt);
%\filldraw [black] (8.9,-0.8) circle (2pt);
%\filldraw [black] (9.5,-0.8) circle (2pt);
\filldraw [black] (6.4,-3.1) circle (2pt);
\draw  [black](6.4,-3.1)--(6.8,-0.8);
\draw  [black](6.4,-3.1)--(9.2,-0.8);
\filldraw [black] (7.2,-3.1) circle (2pt);
\draw  [black](7.2,-3.1)--(6.8,-0.8);
\draw  [black](7.2,-3.1)--(7.5,-0.8);
\draw  [black](7.2,-3.1)--(8.5,-0.8);
\filldraw [black] (8.0,-3.1) circle (2pt);
\draw  [black](8.0,-3.1)--(6.8,-0.8);
\draw  [black](8.0,-3.1)--(7.5,-0.8);
\draw  [black](8.0,-3.1)--(8.5,-0.8);
\filldraw [black] (8.8,-3.1) circle (2pt);
\draw  [black](8.8,-3.1)--(7.5,-0.8);
\draw  [black](8.8,-3.1)--(8.5,-0.8);
\draw  [black](8.8,-3.1)--(9.2,-0.8);
\filldraw [black] (9.6,-3.1) circle (2pt);
\draw  [black](9.6,-3.1)--(9.2,-0.8);
\node at (8, -0.2) {$K_{n-5}$};
\node at (8, -3.8) {$\Gamma_{n, 5}$};
\end{tikzpicture}\ \ \ \ \ \ \ \
\begin{tikzpicture}
\draw (12,-0.8) ellipse (2 and 1);
\filldraw [black] (10.3,-0.8) circle (2pt);
\filldraw [black] (10.55,-0.8) circle (0.8pt);
\filldraw [black] (10.7,-0.8) circle (0.8pt);
\filldraw [black] (10.85,-0.8) circle (0.8pt);
\filldraw [black] (11.1,-0.8) circle (2pt);
\filldraw [black] (11.7,-0.8) circle (2pt);
\filldraw [black] (11.95,-0.8) circle (0.8pt);
\filldraw [black] (12.1,-0.8) circle (0.8pt);
\filldraw [black] (12.25,-0.8) circle (0.8pt);
\filldraw [black] (12.5,-0.8) circle (2pt);
\filldraw [black] (13.1,-0.8) circle (2pt);
\filldraw [black] (13.7,-0.8) circle (2pt);

\filldraw [black] (10.3,-3) circle (2pt);
\draw  [black](10.3,-3)--(10.3,-0.8);
\filldraw [black] (10.55,-3) circle (0.8pt);
\filldraw [black] (10.7,-3) circle (0.8pt);
\filldraw [black] (10.85,-3) circle (0.8pt);
\filldraw [black] (11.1,-3) circle (2pt);
\draw  [black](11.1,-3)--(11.1,-0.8);
\filldraw [black] (11.6,-3) circle (2pt);
\draw  [black](11.6,-3)--(10.3,-0.8);
\draw  [black](11.6,-3)--(11.1,-0.8);
\draw  [black](11.6,-3)--(11.7,-0.8);
\draw  [black](11.6,-3)--(12.5,-0.8);
\draw  [black](11.6,-3)--(13.1,-0.8);
\filldraw [black] (12.1,-3) circle (2pt);
\draw  [black](12.1,-3)--(10.3,-0.8);
\draw  [black](12.1,-3)--(11.1,-0.8);
\draw  [black](12.1,-3)--(11.7,-0.8);
\draw  [black](12.1,-3)--(12.5,-0.8);
\draw  [black](12.1,-3)--(13.7,-0.8);
\filldraw [black] (12.6,-3) circle (2pt);
\draw  [black](12.6,-3)--(11.7,-0.8);
\draw  [black](12.6,-3)--(12.5,-0.8);
\draw  [black](12.6,-3)--(13.1,-0.8);
\draw  [black](12.6,-3)--(13.7,-0.8);
\filldraw [black] (13.1,-3) circle (2pt);
\draw  [black](13.1,-3)--(13.1,-0.8);
\filldraw [black] (13.7,-3) circle (2pt);
\draw  [black](13.7,-3)--(13.7,-0.8);
\draw (12,-3) ellipse (2 and 1);
\node at (12.1, -3.6) {$tK_1$};
\node at (12, -0.2) {$K_{n-t}$};
\node at (12.1, -4.5) {$\Gamma_{n,t}$ \text{ with } $t\ge 6$};
\end{tikzpicture}
\caption{Graphs $\Gamma_{n,t}$ \text{ with } $t\ge 2$.}
\label{f3}
\end{figure}

Next, let $\Gamma_{n,2}'$ and $\Gamma_{n, 3}'$ be the $n$-vertex graphs $(K,I)$ with $|I|=2,3\le \frac{n}{2}$, given in Fig. \ref{fth}, where $\Gamma_{n, 2}'$ is the graph obtained from $K_{n-2}\cup 2K_1$ with
$V(2K_1)=\{u_1,u_2\}$ and $V(K_{n-2})=\{v_1,\dots, v_{n-2}\}$
by adding edges $u_iv_j$ for all $i,j=1,2$,
and $\Gamma_{n, 3}'$ is the graph obtained from $K_{n-3}\cup 3K_1$ with
$V(3K_1)=\{u_1,u_2,u_3\}$ and $V(K_{n-3})=\{v_1,\dots, v_{n-3}\}$
by adding edges $u_iv_j$ for all $i,j=1,2$ and $u_3v_k$ for all $k=2,\dots,n-3$.
%and  $\Gamma_{n, 5}'$ is  the graph obtained from $K_{n-5}\cup 5K_1$
%with $V(5K_1)=\{u_1,\dots,u_5\}$ and $V(K_{n-5})=\{v_1,\dots, v_{n-5}\}$
%by adding edges $u_1v_1$, $u_1v_{n-5}$, $u_iv_j$ for all $i=2,3$ and $j=1,\dots, n-6$, $u_4v_k$ for all $k=2,\dots, n-5$, and $u_5v_{n-5}$.

\begin{figure}[htbp]
\centering
\begin{tikzpicture}
\draw (2,-0.8) ellipse (2 and 1);
\filldraw [black] (0.8,-0.8) circle (2pt);
\filldraw [black] (1.6,-0.8) circle (2pt);
\filldraw [black] (2.2,-0.8) circle (2pt);
\filldraw [black] (2.5,-0.8) circle (0.8pt);
\filldraw [black] (2.7,-0.8) circle (0.8pt);
\filldraw [black] (2.9,-0.8) circle (0.8pt);
\filldraw [black] (3.2,-0.8) circle (2pt);
\draw  [black](0.8,-3)--(0.8,-0.8);
\draw  [black](0.8,-3)--(1.6,-0.8);
\filldraw [black] (0.8,-3) circle (2pt);
\filldraw [black] (1.6,-3) circle (2pt);
\draw  [black](1.6,-3)--(0.8,-0.8);
\draw  [black](1.6,-3)--(1.6,-0.8);
\node at (2, -0.2) {$K_{n-2}$};
\node at (2, -3.8) {$\Gamma_{n,2}'$};
\end{tikzpicture} \ \ \ \ \ \ \ \ \
\begin{tikzpicture}\draw (2,-0.8) ellipse (2 and 1);
\filldraw [black] (0.5,-0.8) circle (2pt);
\filldraw [black] (1.5,-0.8) circle (2pt);
\filldraw [black] (2.5,-0.8) circle (2pt);
\filldraw [black] (2.8,-0.8) circle (0.8pt);
\filldraw [black] (3.0,-0.8) circle (0.8pt);
\filldraw [black] (3.2,-0.8) circle (0.8pt);
\filldraw [black] (3.5,-0.8) circle (2pt);
\filldraw [black] (0.5,-3.1) circle (2pt);
\draw  [black](0.5,-3.1)--(0.5,-0.8);
\draw  [black](0.5,-3.1)--(1.5,-0.8);
\filldraw [black] (1.5,-3.1) circle (2pt);
\draw  [black](1.5,-3.1)--(0.5,-0.8);
\draw  [black](1.5,-3.1)--(1.5,-0.8);
\filldraw [black] (2.5,-3.1) circle (2pt);
\draw  [black](2.5,-3.1)--(1.5,-0.8);
\draw  [black](2.5,-3.1)--(2.5,-0.8);
\draw  [black](2.5,-3.1)--(3.5,-0.8);
\node at (2, -0.2) {$K_{n-3}$};
\node at (2, -3.8) {$\Gamma_{n,3}'$};
\end{tikzpicture}\ \ \ \  \ \
\caption{Graphs $\Gamma_{n,2}', \Gamma_{n, 3}'$.}
\label{fth}
\end{figure}
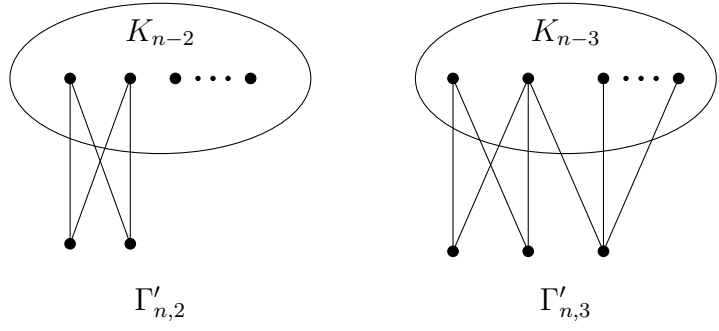

\begin{Theorem}\label{N2}
Let $G=(K,I)$ be a connected $K_{1,4}$-free split graph of order $n\ge \max\{5,2|I|\}$. If
\[
\rho(G)\ge\begin{cases}
\rho(\Gamma_{n,2}') & \mbox {if~} |I|=2, n=5\\
%\rho(\Gamma_{n,2}) & \mbox {if~} |I|=2, n\ge6\\
\rho(\Gamma_{n,3}') & \mbox {if~} |I|=3, n=6\\
%\rho(\Gamma_{n,5}) & \mbox {if~} |I|=5 \\
\rho(\Gamma_{n,|I|}) & \mbox {otherwise}
\end{cases}
\]
then $G$ is Hamiltonian, unless
\[
G\cong \begin{cases}
\Gamma_{n,2}' & \mbox {if~} |I|=2, n=5,\\
%\Gamma_{n,2} & \mbox {if~} |I|=2, n\ge6,\\
\Gamma_{n,3}' & \mbox {if~} |I|=3, n=6,\\
%\Gamma_{n,5} & \mbox {if~} |I|=5, \\
\Gamma_{n,|I|} & \mbox {otherwise}.
\end{cases}
\]
\end{Theorem}

Very recently, Zhu et al. \cite{ZFL}  gave a spectral  condition to imply
a connected split graph  $(K,I)$ with certain constraints on
$|K|$ and $|I|$ is  Hamiltonian. For positive integers $p$ and $q$, let  $S(1, p,1,q)$ be the split graph $G=(K,I)$ with $|K|=1+p$ and $|I|=1+q$ in which the degree of one vertex, say $v$ of $I$ is one,  and $G-v$ is a complete split graph.

\begin{Theorem}
Let $G =(K,I)$ be a connected split graph with $|K|\ge \frac{3|I|}{2}+\frac{17}{2}$ and $|I|\ge 3$. If
$\rho(G)\ge \rho(S(1, |K|-1,1,|I|-1))$,  then $G$ is Hamiltonian, unless $G\cong S(1, |K|-1,1,|I|-1)$.
\end{Theorem}

Note that $S(1, |K|-1,1,|I|-1)$ contains  an induced $K_{1, |I|}$. So the result is different from ours.

\section{Preliminaries}

For a graph $G$ with $u\in V(G)$, we denote by $N_G(u)$ and $d_G(u)$ the neighborhood and the degree of $u$ in $G$, respectively. For $S\subset V(G)$,
let $G-S$ be the graph obtained from $G$ be removing the vertices in $S$ (and every edge incident to some vertex of $S$).
Let $E_1\subset E(K_n)\setminus E(G)$. Then denoted by $G+E_1$ the graph obtained from $G$ by adding all edges in $E_1$, and if $E_1=\{uv\}$, then we write $G+uv$ for $G+\{uv\}$.
And for $E_2\subset E(G)$, let $G-E_2$ be the spanning subgraph from $G$ by deleting all edges in $E_2$, and if $E_2=\{uv\}$, then we write $G-uv$ for $G-\{uv\}$.

%A non-complete graph $G$ is $t$-tough if $|S|\ge tc(G-S)$ for every cut set $S\subset V(G)$, where $c(G-S)$ is the number of components of $G-S$.
%The toughness $\tau(G)$ of $G$ is the maximum $t$ for which $G$ is $t$-tough. Let $\tau(K_n)=\infty$.
%
%\begin{Lemma}\cite{KLM}\label{TH}
%Every $\frac{3}{2}$-tough split graph is Hamiltonian.
%\end{Lemma}

%Let $G$ be a connected $K_{1,4}$-free split graph of type $(K,I)$. Let
%$I'=\{u\in I: d_G(u)=2\}$ and $K'=\cup_{u\in I'}N_G(u)$. If $I'\ne \emptyset$, then
%let $H$ be the bipartite subgraph of $G$ with $V(H)=I'\cup K'$ and $E(H)=\{uv\in E(G) : u\in I', v\in K'\}$. Let $C$ be an induced cycle in $H$ such that $V(K)\setminus V(C)\ne\emptyset$. We refer to $C$ in $H$ as a short cycle in $G$, if such a cycle exists.
%
%
%\begin{Lemma} \cite{RS} \label{Hsnd}
%Let $G$ be a $2$-connected $K_{1,4}$-free split graph with $|K|\ge|I|\ge8$. Then $G$ is Hamiltonian
%if and only if there are no induced short cycles in $G$.
%\end{Lemma}

The following lemma is an immediate consequence of the Perron-Frobenius theorem.

\begin{Lemma}\label{addedges}
Let $G$ be a  graph and $u$ and $v$  two nonadjacent vertices of $G$. If $G+uv$ is connected,  then $\rho(G+uv)> \rho(G)$.
\end{Lemma}

If $G$ is a connected graph, then $A(G)$ is irreducible, so by Perron-Frobenius theorem, there exists a unique unit positive eigenvector $\mathbf{x}$ of $A(G)$ corresponding to $\rho(G)$, which we call the Perron vector of $G$. For $u\in V(G)$, denote by $x_u$ the entry of $\mathbf{x}$ at $u$.

\begin{Lemma}\label{perron} \cite{BS}
Let $G$ be a connected graph with $\{u,v,w\}\subseteq V(G)$ such that $uw\notin E(G)$ and $vw\in E(G)$. If $\mathbf{x}$ is the Perron vector of $G$ with $x_u\ge x_v$, then $\rho(G-vw+uw)>\rho(G)$.
\end{Lemma}

Note  that if $\mathbf{y}$ is the Perron vector of $G-vw+uw$ in Lemma \ref{perron}, then  $y_u>y_v$ by Lemma \ref{perron}. So Lemma \ref{perron} has the equivalent form:
Let $G$ be a connected graph with $\{u,v\}\subset V(G)$ and  $N_G(v)\setminus N_G[u]\ne \emptyset$.
If $\mathbf{x}$ is the Perron vector of $G$ with $x_u\ge x_v$, then for any nonempty $N\subseteq  N_G(v)\setminus N_G[u]$, $\rho(G-\{vw: w\in N\}+\{uw: w\in N\})>\rho(G)$.

For an $n\times n$ matrix $A$, whose rows and columns are indexed by elements in $X=\{1, \dots, n\}$,
let $\pi=\{X_1, \dots, X_s\}$ be a partition of $X$, and let $A_{ij}$ be the submatrix (block) of $A$ whose rows and columns are indexed by elements of $X_i$ and $X_j$ for $1\le i,j\le s$.
The quotient matrix of $A$ is the $s\times s$ matrix, whose $(i,j)$-entry is the average row sums of $A_{ij}$, where $1\le i,j\le s$.
Moreover, if the row sum of each  $A_{ij}$ is a constant, then we call the partition $\pi$ is equitable.

\begin{Lemma}\label{QM}\cite{BH}
Let $A$ be a real symmetric matrix. Then the spectrum of the quotient matrix $B$ of $A$
corresponding to an equitable partition is contained in the spectrum of $A$.
Furthermore, if $A$ is nonnegative and irreducible, then $\rho(A)=\rho(B)$.
\end{Lemma}

\section{Proof of Theorem \ref{N1}}

\begin{proof}[Proof of Theorem \ref{N1}]
Let $G=(K,I)$ be a connected non-Hamiltonian claw-free split graph of order $n$ that maximizes the spectral radius. Assume that $I=\{u_1,\dots, u_{|I|}\}$.
If $|I|=1$, then $d_G(u_1)=1$, so $G\cong G_{n,1}$,  as otherwise $G$ contains a Hamiltonian cycle, which is a contradiction.
Suppose in the following that $|I|\ge 2$.

\begin{Claim}\label{C1}
For any $\{i,j\}\subseteq \{1, \dots, |I|\}$, $N_G(u_i)\cap N_G(u_j)=\emptyset$.
\end{Claim}

\begin{proof}
Suppose that this is not true. We may assume that  $v\in N_G(u_1)\cap N_G(u_2)=:N$. Then  $K=N_G(u_1)\cup N_G(u_2)$, as otherwise, there would be a vertex $v'\in K$ such that $v'\notin N_G(u_1)\cup N_G(u_2)$, so
$v$ has three pairwise nonadjacent neighbors $u_1,u_2,v'$, a contradiction. Moreover,
as $K$ is a maximum clique of $G$, we have $N_G(u_i)\subset K$ and so $N_G(u_i)\setminus N\ne \emptyset$ for $i=1,2$. Then there is a vertex $v_i\in N_G(u_i)\setminus N$ for $i=1,2$. Let $P$ be a Hamiltonian path in $G[K\setminus \{v\}]$ from $v_2$ to $v_1$. If $|I|=2$, then  $v_1u_1vu_2v_2P$ is a Hamiltonian cycle of $G$, a contradiction.

Suppose that $|I|\ge 3$.

As $G$ is claw-free, $N_G(u_3)\subseteq K\setminus N$.
Suppose that  $N_G(u_3)\ne K\setminus N$. Assume that $u_3$ is not adjacent to some vertex  $w\in N_G(u_1)\setminus N$. If $u_3$ has a neighbor $z$ in $N_G(u_2)\setminus N$, then $z$ has three pairwise nonadjacent neighbors $w,u_2,u_3$, a contradiction. So $u_3$ has no neighbor in $N_G(u_2)\setminus N$. As $G$ is connected, $u_3$ has a neighbor $y$ in $N_G(u_1)\setminus N$, so
 $y$ has three pairwise nonadjacent neighbors $u_1, u_3, v_2$, also a contradiction. Then
$N_G(u_3)=K\setminus N$. Recall that $K=N_G(u_1)\cup N_G(u_2)$.
Thus, any vertex in $K$ has at least two neighbors in $I$.
As $G$ is claw-free, it follows that $|I|=3$.
So $G$ has structure as displayed in Fig. \ref{f2}, so
$G$ contains a Hamiltonian cycle $u_3Pu_1Qu_2Ru_3$, where $P$, $Q$, $R$ are  Hamiltonian paths in
$G[N_G(u_1)\cap N_G(u_3)]$, $G[N]$, $G[N_G(u_2)\cap N_G(u_3)]$, respectively. This is
a contradiction.
\end{proof}

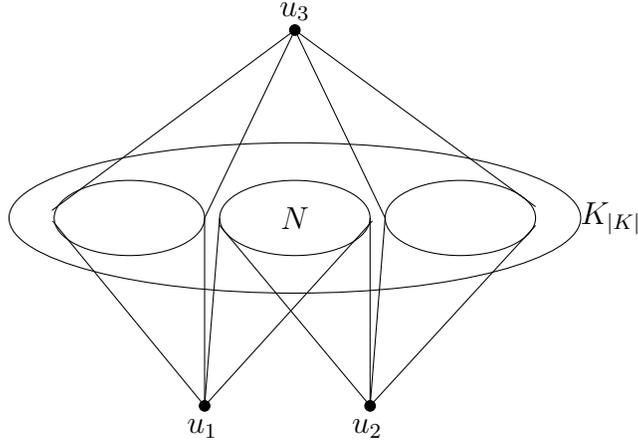
\begin{figure}[htbp]
\centering
%\begin{tikzpicture}
%\draw (2,-0.2) ellipse (1 and 3.8);
%\draw (2,2) ellipse (0.5 and 1);
%\draw (2,-0.2) ellipse (0.5 and 1);
%\draw (2,-2.4) ellipse (0.5 and 1);
%\filldraw [black] (5,0.8) circle (2pt);
%\draw  [black](5,0.8)--(2.2,2.92);
%\draw  [black](5,0.8)--(2,1);
%\draw  [black](5,0.8)--(2,0.8);
%\draw  [black](5,0.8)--(2.1,-1.2);
%\filldraw [black] (5,-1.4) circle (2pt);
%\draw  [black](5,-1.4)--(2.1,0.8);
%\draw  [black](5,-1.4)--(2,-1.2);
%\draw  [black](5,-1.4)--(2,-1.4);
%\draw  [black](5,-1.4)--(2.2,-3.32);
%\node at (5.36, 0.76) {$u_1$};
%\node at (5.36, -1.43) {$u_2$};
%\node at (2, 4) {$K_{|K|}$};
%\node at (2, -0.2) {$N$};
%\filldraw [black] (-1,-0.2) circle (2pt);
%\draw  [black](-1,-0.2)--(1.8,2.92);
%\draw  [black](-1,-0.2)--(2,1);
%\draw  [black](-1,-0.2)--(2,-1.4);
%\draw  [black](-1,-0.2)--(1.8,-3.32);
%\node at (-1.35, -0.2) {$u_3$};
%\end{tikzpicture}
\begin{tikzpicture}
\draw (-0.2,2) ellipse (3.8 and 1);
\draw (2,2) ellipse (1 and 0.5);
\draw (-0.2,2) ellipse (1 and 0.5);
\draw (-2.4,2) ellipse (1 and 0.5);
\filldraw [black] (0.8,-0.5) circle (2pt);
\draw  [black](0.8,-0.5)--(3,1.9);
\draw  [black](0.8,-0.5)--(1,2);
\draw  [black](0.8,-0.5)--(0.8,2);
\draw  [black](0.8,-0.5)--(-1.2,1.9);
\filldraw [black] (-1.4,-0.5) circle (2pt);
\draw  [black](-1.4,-0.5)--(0.83,1.96);
\draw  [black](-1.4,-0.5)--(-1.2,2);
\draw  [black](-1.4,-0.5)--(-1.4,2);
\draw  [black](-1.4,-0.5)--(-3.42,1.96);
\node at (0.76,-0.8) {$u_2$};
\node at (-1.43,-0.8) {$u_1$};
\node at (4,2) {$K_{|K|}$};
\node at (-0.2,2) {$N$};
\filldraw [black] (-0.2,4.5) circle (2pt);
\draw  [black](-0.2,4.5)--(3,2.15);
\draw  [black](-0.2,4.5)--(1,2);
\draw  [black](-0.2,4.5)--(-1.4,2);
\draw  [black](-0.2,4.5)--(-3.43,2.1);
\node at (-0.2,4.75) {$u_3$};
\end{tikzpicture}
\caption{The structure of $G$ in the proof of Theorem \ref{N1}.}
\label{f2}
\end{figure}

Assume that $d_G(u_i)\ge d_G(u_j)$ for any $1\le i<j\le |I|$.
By Claim \ref{C1}, for any $1\le i<j\le |I|$, $N_G(u_i)\cap N_G(u_j)=\emptyset$.
Let $\mathbf{x}$ be the Perron vector of $G$. For $j=1,\dots, |I|$, let $y_j=x_{u_j}$,  and by symmetry, all  the entry of $\mathbf{x}$ at each vertex  from
$N_G(u_i)$ is the same, which we denote by $x_i$, and the entry of $\mathbf{x}$ at each vertex  from
$K\setminus \bigcup_{\ell=1}^{|I|}N_G(u_\ell)$ (if any exists) is the same, which we denote by $x_{|I|+1}$.
Let $1\le i<j\le |I|$. On one hand, from $\rho(G)\mathbf{x}=A(G)\mathbf{x}$ at any vertex in  $N_G(u_i)$ and $N_G(u_j)$, we have
\[
\rho(G)x_i=(d_G(u_i)-1)x_i+\sum_{k=1\atop k\ne i}^{|I|}d_G(u_k)x_k+\left|K\setminus \bigcup_{\ell=1}^{|I|}N_G(u_\ell)\right|x_{|I|+1}+y_i
\]
and
\[
\rho(G)x_j=(d_G(u_j)-1)x_j+\sum_{k=1\atop k\ne j}^{|I|}d_G(u_k)x_k+\left|K\setminus\bigcup_{\ell=1}^{|I|}N_G(u_\ell)\right|x_{|I|+1}+y_j,
\]
so
\begin{equation} \label{CAI}
(\rho(G)+1)(x_i-x_j)=y_i-y_j.
\end{equation}
On the other hand,  from $\rho(G)\mathbf{x}=A(G)\mathbf{x}$ at $u_i$ and $u_j$, we have $\rho(G)y_i=d_G(u_i)x_i$ and $\rho(G)y_j=d_G(u_j)x_j$, so
\[
\rho(G)y_i-\rho(G)y_j=d_G(u_i)x_i-d_G(u_j)x_j\ge d_G(u_i)(x_i-x_j)
\]
in view of $d_G(u_i)\ge d_G(u_j)$.
Now we have
\[
\rho(G)(\rho(G)+1)(x_i-x_j)=\rho(G)y_i-\rho(G)y_j\ge d_G(u_i)(x_i-x_j),
\]
i.e.,
\[
(\rho(G)^2+\rho(G)-d_G(u_i))(x_i-x_j)\ge0.
\]
As $G$ contains $K$, we have  $\rho(G)\ge\rho(K)=|K|-1$, so $\rho^2(G)+\rho(G)-d_G(u_i)\ge\rho^2(G)-1>0$,
and then  $x_i-x_j\ge 0$.  From \eqref{CAI}, we have  $y_i\ge y_j$.
It follows that
\[
 y_1=\max\{y_i:1\le i\le|I|\}.
\]

Suppose that $d_G(u_2)\ge 2$.  Let $E'$ be a set of edges of $G$  incident to $u_2$ except one edge. Let $K'=\{w: wu_2\in E'\}$. Let $G'=G-E'+\{u_1w: w\in K'\}$. Note that $G'$ is claw-free and as $d_{G'}(u_2)=1$, it is not Hamiltonian.
By Lemma
\ref{perron},  we have $\rho(G')>\rho(G)$, which is a contradiction. Thus  $d_G(u_2)=1$ and so
$d_G(u_2)=\dots=d_G(u_{|I|})=1$.

We claim that $K\setminus \bigcup_{i=1}^{|I|}N_G(u_i)=\emptyset$, as otherwise, for a vertex $w\in K\setminus \bigcup_{i=1}^{|I|}N_G(u_i)$, we have by Lemma \ref{addedges} that $\rho(G+u_1w)>\rho(G)$, which is a contradiction, as $G+u_1w$ is claw-free and non-Hamiltonian.

Therefore,
$G\cong G_{n,|I|}$.
\end{proof}

\section{Proof of Theorem \ref{N2}}

We need a few lemmas, where the proof of the first five ones is given in Appendix.

\begin{Lemma} \label{bbb} For $n\ge 6$, let $\Gamma_{n,2}''$ be the graph displayed in Fig. \ref{f1n}.
Then $\rho(\Gamma_{n,2})>\max\{\rho(\Gamma_{n,2}'), \rho(\Gamma_{n,2}'')\}$.
\end{Lemma}

\begin{figure}[htbp]
\centering
\begin{tikzpicture}
\draw (2,-0.8) ellipse (2 and 1);
\filldraw [black] (0.8,-0.8) circle (2pt);
\filldraw [black] (1.2,-0.8) circle (0.8pt);
\filldraw [black] (1.5,-0.8) circle (0.8pt);
\filldraw [black] (1.8,-0.8) circle (0.8pt);
\filldraw [black] (2.3,-0.8) circle (2pt);
\filldraw [black] (3.2,-0.8) circle (2pt);
\draw  [black](1.5,-3)--(0.8,-0.8);
\draw  [black](1.5,-3)--(2.3,-0.8);
\filldraw [black] (1.5,-3) circle (2pt);
\filldraw [black] (3.2,-3) circle (2pt);
\draw  [black](3.2,-3)--(3.2,-0.8);
\node at (2, -0.2) {$K_{n-2}$};
\node at (2, -3.8) {$\Gamma_{n,2}''$};

\end{tikzpicture}
\caption{The structure of $\Gamma_{n,2}''$.}
\label{f1n}
\end{figure}
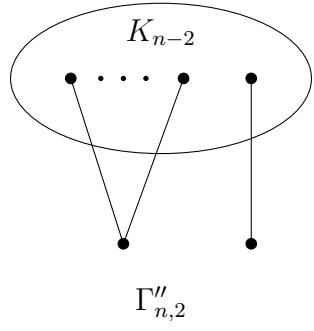

\begin{Lemma}\label{I3}
%Let $\Gamma_{n,3}$, $\Gamma_{n,3}'$ and $\Gamma_{n,3}''$ be three graphs with $n\ge 6$ displayed in Figs. \ref{f3}, \ref{fth} and \ref{flem}, respectively.
Let $\Gamma_{n,3}''$ be the graph displayed in Fig. \ref{flem} with $n\ge 6$.
If $n=6$, then  $\rho(\Gamma_{n,3}')>\rho(\Gamma_{n,3})>\rho(\Gamma_{n,3}'')$, and if  $n\ge7$, then  $\rho(\Gamma_{n,3})>\max\{\rho(\Gamma_{n,3}'), \rho(\Gamma_{n,3}'')\}$.
\end{Lemma}

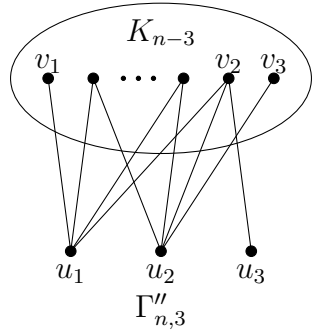
\begin{figure}[htbp]
\centering
\begin{tikzpicture}
\draw (2,-0.8) ellipse (2 and 1);
\filldraw [black] (0.5,-0.8) circle (2pt);
\filldraw [black] (1.1,-0.8) circle (2pt);
\filldraw [black] (1.5,-0.8) circle (0.8pt);
\filldraw [black] (1.7,-0.8) circle (0.8pt);
\filldraw [black] (1.9,-0.8) circle (0.8pt);
\filldraw [black] (2.3,-0.8) circle (2pt);
\filldraw [black] (2.9,-0.8) circle (2pt);
\filldraw [black] (3.5,-0.8) circle (2pt);
\filldraw [black] (0.8,-3.1) circle (2pt);
\draw  [black](0.8,-3.1)--(0.5,-0.8);
\draw  [black](0.8,-3.1)--(1.1,-0.8);
\draw  [black](0.8,-3.1)--(2.3,-0.8);
\draw  [black](0.8,-3.1)--(2.9,-0.8);
\filldraw [black] (2.0,-3.1) circle (2pt);
\draw  [black](2,-3.1)--(3.5,-0.8);
\draw  [black](2,-3.1)--(1.1,-0.8);
\draw  [black](2,-3.1)--(2.3,-0.8);
\draw  [black](2,-3.1)--(2.9,-0.8);
\filldraw [black] (3.2,-3.1) circle (2pt);
\draw  [black](3.2,-3.1)--(2.9,-0.8);
\node at (2.9, -0.6) {$v_2$};
\node at (0.5, -0.6) {$v_1$};
\node at (3.5, -0.6) {$v_3$};
\node at (0.8, -3.4) {$u_1$};
\node at (2, -3.4) {$u_2$};
\node at (3.2, -3.4) {$u_3$};
\node at (2, -0.2) {$K_{n-3}$};
\node at (2, -3.9) {$\Gamma_{n,3}''$};
\end{tikzpicture}
\caption{The structure of $\Gamma_{n,3}''$.}
\label{flem}
\end{figure}

\begin{Lemma}\label{I44}
Let  $\Gamma_{n,4}'$ be the graph displayed in Fig. \ref{abc} with $n\ge 8$. Then  $\rho(\Gamma_{n,4})>\rho(\Gamma_{n,4}')$.
\end{Lemma}

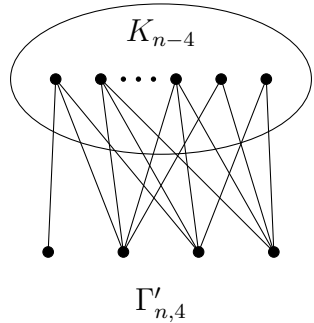
\begin{figure}[htbp]
\centering
\begin{tikzpicture}
\draw (2,-0.8) ellipse (2 and 1);
\filldraw [black] (0.6,-0.8) circle (2pt);
\filldraw [black] (1.2,-0.8) circle (2pt);
\filldraw [black] (1.5,-0.8) circle (0.8pt);
\filldraw [black] (1.7,-0.8) circle (0.8pt);
\filldraw [black] (1.9,-0.8) circle (0.8pt);
\filldraw [black] (2.2,-0.8) circle (2pt);
\filldraw [black] (2.8,-0.8) circle (2pt);
\filldraw [black] (3.4,-0.8) circle (2pt);
\filldraw [black] (0.5,-3.1) circle (2pt);
\draw  [black](0.5,-3.1)--(0.6,-0.8);
\filldraw [black] (1.5,-3.1) circle (2pt);
\draw  [black](1.5,-3.1)--(0.6,-0.8);
\draw  [black](1.5,-3.1)--(1.2,-0.8);
\draw  [black](1.5,-3.1)--(2.2,-0.8);
\draw  [black](1.5,-3.1)--(2.8,-0.8);
\filldraw [black] (2.5,-3.1) circle (2pt);
\draw  [black](2.5,-3.1)--(0.6,-0.8);
\draw  [black](2.5,-3.1)--(1.2,-0.8);
\draw  [black](2.5,-3.1)--(2.2,-0.8);
\draw  [black](2.5,-3.1)--(3.4,-0.8);
\filldraw [black] (3.5,-3.1) circle (2pt);
\draw  [black](3.5,-3.1)--(1.2,-0.8);
\draw  [black](3.5,-3.1)--(2.2,-0.8);
\draw  [black](3.5,-3.1)--(2.8,-0.8);
\draw  [black](3.5,-3.1)--(3.4,-0.8);
\node at (2, -0.2) {$K_{n-4}$};
\node at (2, -3.8) {$\Gamma_{n,4}'$};
\end{tikzpicture}
\caption{The structure of $\Gamma_{n,4}'$.}
\label{abc}
\end{figure}

\begin{Lemma}\label{I5}
Let $\Gamma_{n,5}'$ and $\Gamma_{n,5}''$ be two graphs displayed in Fig. \ref{f22} with $n\ge10$.
Then $\rho(\Gamma_{n,5})>\rho(\Gamma_{n,5}'')>\rho(\Gamma_{n,5}')$.
\end{Lemma}

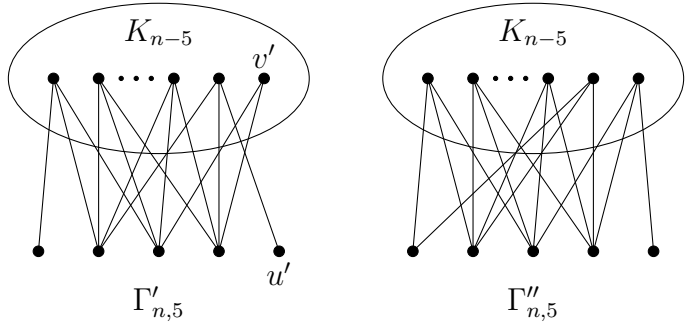
\begin{figure}[htbp]
\centering

\begin{tikzpicture}
\draw (7,-0.8) ellipse (2 and 1);
\filldraw [black] (5.6,-0.8) circle (2pt);
\filldraw [black] (6.2,-0.8) circle (2pt);
\filldraw [black] (6.5,-0.8) circle (0.8pt);
\filldraw [black] (6.7,-0.8) circle (0.8pt);
\filldraw [black] (6.9,-0.8) circle (0.8pt);
\filldraw [black] (7.2,-0.8) circle (2pt);
\filldraw [black] (7.8,-0.8) circle (2pt);
\filldraw [black] (8.4,-0.8) circle (2pt);

\filldraw [black] (5.4,-3.1) circle (2pt);
\draw  [black](5.4,-3.1)--(5.6,-0.8);
\filldraw [black] (6.2,-3.1) circle (2pt);
\draw  [black](6.2,-3.1)--(5.6,-0.8);
\draw  [black](6.2,-3.1)--(6.2,-0.8);
\draw  [black](6.2,-3.1)--(7.2,-0.8);
\draw  [black](6.2,-3.1)--(7.8,-0.8);
\filldraw [black] (7,-3.1) circle (2pt);
\draw  [black](7,-3.1)--(5.6,-0.8);
\draw  [black](7,-3.1)--(6.2,-0.8);
\draw  [black](7,-3.1)--(7.2,-0.8);
\draw  [black](7,-3.1)--(8.4,-0.8);
\filldraw [black] (7.8,-3.1) circle (2pt);
\draw  [black](7.8,-3.1)--(6.2,-0.8);
\draw  [black](7.8,-3.1)--(7.2,-0.8);
\draw  [black](7.8,-3.1)--(7.8,-0.8);
\draw  [black](7.8,-3.1)--(8.4,-0.8);
\filldraw [black] (8.6,-3.1) circle (2pt);
\draw  [black](8.6,-3.1)--(7.8,-0.8);
\node at (8.6, -3.4) {$u'$};
\node at (8.4, -0.5) {$v'$};
\node at (7, -0.2) {$K_{n-5}$};
\node at (7, -3.8) {$\Gamma_{n,5}'$};
\end{tikzpicture}\ \ \ \ \ \ \
\begin{tikzpicture}
\draw (7,-0.8) ellipse (2 and 1);
\filldraw [black] (5.6,-0.8) circle (2pt);
\filldraw [black] (6.2,-0.8) circle (2pt);
\filldraw [black] (6.5,-0.8) circle (0.8pt);
\filldraw [black] (6.7,-0.8) circle (0.8pt);
\filldraw [black] (6.9,-0.8) circle (0.8pt);
\filldraw [black] (7.2,-0.8) circle (2pt);
\filldraw [black] (7.8,-0.8) circle (2pt);
\filldraw [black] (8.4,-0.8) circle (2pt);

\filldraw [black] (5.4,-3.1) circle (2pt);
\draw  [black](5.4,-3.1)--(5.6,-0.8);
\draw  [black](5.4,-3.1)--(7.8,-0.8);
\filldraw [black] (6.2,-3.1) circle (2pt);
\draw  [black](6.2,-3.1)--(5.6,-0.8);
\draw  [black](6.2,-3.1)--(6.2,-0.8);
\draw  [black](6.2,-3.1)--(7.2,-0.8);
\draw  [black](6.2,-3.1)--(7.8,-0.8);
\filldraw [black] (7,-3.1) circle (2pt);
\draw  [black](7,-3.1)--(5.6,-0.8);
\draw  [black](7,-3.1)--(6.2,-0.8);
\draw  [black](7,-3.1)--(7.2,-0.8);
\draw  [black](7,-3.1)--(8.4,-0.8);
\filldraw [black] (7.8,-3.1) circle (2pt);
\draw  [black](7.8,-3.1)--(6.2,-0.8);
\draw  [black](7.8,-3.1)--(7.2,-0.8);
\draw  [black](7.8,-3.1)--(7.8,-0.8);
\draw  [black](7.8,-3.1)--(8.4,-0.8);
\filldraw [black] (8.6,-3.1) circle (2pt);
\draw  [black](8.6,-3.1)--(8.4,-0.8);
\node at (7, -0.2) {$K_{n-5}$};
\node at (7, -3.8) {$\Gamma_{n,5}''$};

\end{tikzpicture}
\caption{The structures of $\Gamma_{n,5}'$ and $\Gamma_{n,5}''$.}
\label{f22}
\end{figure}

\begin{Lemma}\label{I4}
Let $\Gamma_{n, |I|}^*$ and $\Gamma_{n, |I|}^{**}$ be two graphs displayed in Fig. \ref{flem11}. Denote by $V_i$ the vertex subset of $V(\Gamma_{n, |I|}^*)$ for $i=1,\dots,9$ and $V_j'$ the vertex subset of $V(\Gamma_{n, |I|}^{**})$ for $j=1,\dots,8$ that are displayed in Fig. \ref{flem11}, where $|V_4|=|V_1|$, $|V_4'|=|V_1'|$, $|V_2|=|V_2'|$, and  $V_3,V_5,\dots, V_9$, $V_6'$ and $V_7'$ are all vertex subset of a single vertex.
Then for $n\ge2|I|$ and $|I|\ge6$, $\rho(\Gamma_{n,|I|})>\max\{\rho(\Gamma_{n, |I|}^*), \rho(\Gamma_{n, |I|}^{**})\}$.
%Then for $n\ge2|I|, \frac{5}{2}|I|-6\}$ and $|I|\ge6$, $\rho(\Gamma_{n, |I|}^*)\le\rho(\Gamma_{n, |I|}^{**})<\rho(\Gamma_{n,|I|})$.
\end{Lemma}

\begin{figure}[htp]
\centering
\begin{tikzpicture}

\draw (6,-0.8) ellipse (0.65 and 0.17);
\node at (6, -0.4) {$V_1$};
\draw (7.5,-0.8) ellipse (0.65 and 0.17);
\node at (7.5, -0.4) {$V_2$};
\node at (8.5, -0.4) {$V_3$};
\draw (5.5,-3.1) ellipse (0.65 and 0.17);
\node at (5.5, -3.5) {$V_4$};
\draw (7,-0.8) ellipse (2 and 1);
\filldraw [black] (5.5,-0.8) circle (2pt);
\filldraw [black] (5.8,-0.8) circle (0.8pt);
\filldraw [black] (6.0,-0.8) circle (0.8pt);
\filldraw [black] (6.2,-0.8) circle (0.8pt);
\filldraw [black] (6.5,-0.8) circle (2pt);
\filldraw [black] (7.0,-0.8) circle (2pt);
\filldraw [black] (7.3,-0.8) circle (0.8pt);
\filldraw [black] (7.5,-0.8) circle (0.8pt);
\filldraw [black] (7.7,-0.8) circle (0.8pt);
\filldraw [black] (8.0,-0.8) circle (2pt);
\filldraw [black] (8.5,-0.8) circle (2pt);

\filldraw [black] (5.0,-3.1) circle (2pt);
\draw  [black](5.0,-3.1)--(5.5,-0.8);
\filldraw [black] (5.3,-3.1) circle (0.8pt);
\filldraw [black] (5.5,-3.1) circle (0.8pt);
\filldraw [black] (5.7,-3.1) circle (0.8pt);
\filldraw [black] (6.0,-3.1) circle (2pt);
\draw  [black](6.0,-3.1)--(6.5,-0.8);
\filldraw [black] (6.6,-3.1) circle (2pt);
\node at (6.6, -3.4) {$V_5$};
\draw  [black](6.6,-3.1)--(5.5,-0.8);
\draw  [black](6.6,-3.1)--(6.5,-0.8);
\draw  [black](6.6,-3.1)--(8.5,-0.8);
\filldraw [black] (7.2,-3.1) circle (2pt);
\node at (7.2, -3.4) {$V_6$};
\draw  [black](7.2,-3.1)--(5.5,-0.8);
\draw  [black](7.2,-3.1)--(6.5,-0.8);
\draw  [black](7.2,-3.1)--(7.0,-0.8);
\draw  [black](7.2,-3.1)--(8.0,-0.8);
\filldraw [black] (7.8,-3.1) circle (2pt);
\node at (7.8, -3.4) {$V_7$};
\draw  [black](7.8,-3.1)--(7.0,-0.8);
\draw  [black](7.8,-3.1)--(8.0,-0.8);
\draw  [black](7.8,-3.1)--(8.5,-0.8);
\filldraw [black] (8.4,-3.1) circle (2pt);
\node at (8.4, -3.4) {$V_8$};
\draw  [black](8.4,-3.1)--(8.0,-0.8);
\draw  [black](8.4,-3.1)--(7.0,-0.8);
\filldraw [black] (9.0,-3.1) circle (2pt);
\node at (9, -3.4) {$V_9$};
\draw  [black](9.0,-3.1)--(8.5,-0.8);
\node at (7, 0.4) {$K_{n-|I|}$};
\node at (7, -3.96) {$\Gamma_{n, |I|}^* (|I|\ge6)$};

\draw (12,-0.8) ellipse (2 and 1);
\draw (10.75,-0.8) ellipse (0.65 and 0.17);
\node at (10.75,-0.4) {$V_1'$};
\draw (12.25,-0.8) ellipse (0.65 and 0.17);
\node at (12.25,-0.4) {$V_2'$};
\draw (13.5,-0.8) ellipse (0.4 and 0.13);
\node at (13.4,-0.4) {$V_3'$};
\draw (10.5,-3.1) ellipse (0.65 and 0.17);
\node at (10.5,-3.5) {$V_4'$};
\draw (11.75,-3.1) ellipse (0.4 and 0.13);
\node at (11.75,-3.5) {$V_5'$};
\draw (13.75,-3.1) ellipse (0.4 and 0.13);
\node at (13.75,-3.5) {$V_8'$};
\filldraw [black] (10.25,-0.8) circle (2pt);
\filldraw [black] (10.55,-0.8) circle (0.8pt);
\filldraw [black] (10.75,-0.8) circle (0.8pt);
\filldraw [black] (10.95,-0.8) circle (0.8pt);
\filldraw [black] (11.25,-0.8) circle (2pt);
\filldraw [black] (11.75,-0.8) circle (2pt);
\filldraw [black] (12.05,-0.8) circle (0.8pt);
\filldraw [black] (12.25,-0.8) circle (0.8pt);
\filldraw [black] (12.45,-0.8) circle (0.8pt);
\filldraw [black] (12.75,-0.8) circle (2pt);
\filldraw [black] (13.25,-0.8) circle (2pt);
\filldraw [black] (13.75,-0.8) circle (2pt);

\filldraw [black] (10.00,-3.1) circle (2pt);
\draw  [black](10.00,-3.1)--(10.25,-0.8);
\filldraw [black] (10.3,-3.1) circle (0.8pt);
\filldraw [black] (10.5,-3.1) circle (0.8pt);
\filldraw [black] (10.7,-3.1) circle (0.8pt);
\filldraw [black] (11.0,-3.1) circle (2pt);
\draw  [black](11.0,-3.1)--(11.25,-0.8);
\filldraw [black] (11.5,-3.1) circle (2pt);
\draw  [black](11.5,-3.1)--(10.25,-0.8);
\draw  [black](11.5,-3.1)--(11.25,-0.8);
\draw  [black](11.5,-3.1)--(11.75,-0.8);
\draw  [black](11.5,-3.1)--(12.75,-0.8);
\draw  [black](11.5,-3.1)--(13.25,-0.8);
\filldraw [black] (12,-3.1) circle (2pt);
\draw  [black](12,-3.1)--(10.25,-0.8);
\draw  [black](12,-3.1)--(11.25,-0.8);
\draw  [black](12,-3.1)--(11.75,-0.8);
\draw  [black](12,-3.1)--(12.75,-0.8);
\draw  [black](12,-3.1)--(13.75,-0.8);
\filldraw [black] (12.5,-3.1) circle (2pt);
\node at (12.5,-3.5) {$V_6'$};
\draw  [black](12.5,-3.1)--(11.75,-0.8);
\draw  [black](12.5,-3.1)--(12.75,-0.8);
\filldraw [black] (13,-3.1) circle (2pt);
\node at (13,-3.5) {$V_7'$};
\draw  [black](13,-3.1)--(13.25,-0.8);
\draw  [black](13,-3.1)--(13.75,-0.8);
\filldraw [black] (13.5,-3.1) circle (2pt);
\draw  [black](13.5,-3.1)--(13.25,-0.8);
\filldraw [black] (14,-3.1) circle (2pt);
\draw  [black](14,-3.1)--(13.75,-0.8);
\node at (12, 0.4) {$K_{n-|I|}$};
\node at (12, -3.96) {$\Gamma_{n, |I|}^{**} (|I|\ge6)$};
\end{tikzpicture}
\caption{The structures of $\Gamma_{n, |I|}^*$ and $\Gamma_{n, |I|}^{**}$.}
\label{flem11}
\end{figure}
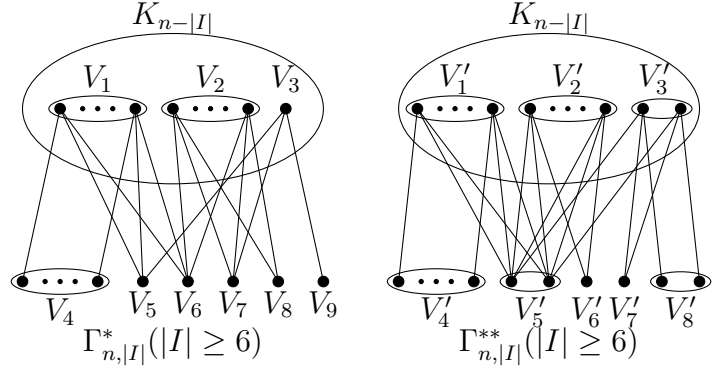

Let $\mathbb{G}_n$ be the class of $n$-vertex split graphs of type $(K,I)$ each of which is  connected, $K_{1,4}$-free, and  not Hamiltonian.

\begin{Lemma}\label{DI1}
Let $G=(K,I)$ be a connected $K_{1,4}$-free split graph of order $n\ge 2|I|$ that is not Hamiltonian but with maximum spectral radius, where $|I|\ge4$.
If there exists  a vertex of $K$  such that its degree in $I$ is equal to $3$, then there is a vertex of $I$ with degree $1$.
\end{Lemma}

\begin{proof}
Let $I=\{u_1,\dots,u_{|I|}\}$. Denote by $\mathbf{x}$ the Perron vector of $G$.
Let $x_{v_0}=\max\{x_v: v\in K, d_I(v)=3\}$, and $N_I(v_0)=\{u_1,u_2,u_3\}$ with $x_{u_1}\ge x_{u_2}\ge x_{u_3}$. As $G$ is $K_{1,4}$-free, we have $\bigcup_{i=1}^3 N_G(u_i)=K$.

Suppose that $d_G(u)\ge2$ for any $u\in I$.
Then $G$ has a vertex of degree $2$ in $I$, as otherwise removing any two vertices results in a connected graph, so $G$ is $3$-connected, and by Theorem \ref{Th2}, $G$ is Hamiltonian, a contradiction. Choose $u_i\in I$ such that $x_{u_i}$
is minimum among all vertices of degree two.

For convenience, we write $\rho$ for $\rho(G)$.

\noindent
{\bf Case 1.}  $i\in\{1,2,3\}$.

Let $N_G(u_i)=\{v_0,v_1\}$. Let $\{u_1,u_2,u_3\}\setminus\{u_i\}=\{u_j,u_k\}$, where $1\le j<k\le3$ and $j,k\ne i$.

\noindent
{\bf Case 1.1.} $d_I(v_1)\ge2$ and $N_I(v_1)\setminus\{u_1,u_2,u_3\}\ne\emptyset$.

Let $u'\in N_I(v_1)\setminus\{u_1,u_2,u_3\}$ such that $x_{u'}=\min\{x_u: u\in N_I(v_1)\setminus\{u_1,u_2,u_3\}\}$.
As $\rho(x_{v_0}-x_{u_j})>x_{u_j}-x_{v_0}$ and $\rho>1$, we have $x_{v_0}>x_{u_j}$.
If $d_G(u')\ge3$, then
\begin{align*}
(1+\rho)\rho(x_{u'}-x_{u_i}) & =(1+\rho)\left(\sum_{v\in N_G(u')\setminus\{v_1\}}x_v-x_{v_0}\right)\\
&> 2x_{u'}+x_{u_k}+x_{v_0}-x_{u_i}-x_{u_j}>x_{u'}-x_{u_i},
\end{align*}
so $x_{u'}> x_{u_i}$. It follows that
 $x_{u'}\ge x_{u_i}$ with equality  only if $d_G(u')=2$.

\begin{Claim} \label{Fis}
$d_G(u')\ge3$, $N_I(v_1)\setminus\{u_1,u_2,u_3\}=\{u'\}$ and $N_G(u_j)\cap N_G(u_k)=\{v_0\}$.
%
%If $d_G(u')=2$, or
%$d_G(u')\ge3$ and $|N_I(v_1)\setminus\{u_1,u_2,u_3\}|=2$, or
%$d_G(u')\ge3$, $|N_I(v_1)\setminus\{u_1,u_2,u_3\}|=1$ and $|N_G(u_j)\cap N_G(u_k)|\ge2$, then $u'$ is not adjacent to all vertices of $K\setminus\{v_0\}$.
\end{Claim}

\begin{proof} Suppose that this is not true. Then
$d_G(u')=2$, or
$d_G(u')\ge3$ and $|N_I(v_1)\setminus\{u_1,u_2,u_3\}|=2$, or
$d_G(u')\ge3$, $|N_I(v_1)\setminus\{u_1,u_2,u_3\}|=1$ and $|N_G(u_j)\cap N_G(u_k)|\ge2$.
We show that in each such case, $u'$ is not adjacent to all vertices of $K\setminus\{v_0\}$.
This is obvious in the first case.
In the second case, it follows  as
otherwise, the vertex of $N_I(v_1)\setminus\{u_i,u'\}$ is adjacent to all vertices of $K\setminus\{v_0\}$ by the choice of $u'$, so $\rho(\rho+1)(x_{v_1}-x_{v_0})=\rho(2x_{u'}-x_{u_j}-x_{u_k})>2(x_{v_1}-x_{v_0})$, a contradiction.
Now we consider  the third case. If $v'$ is a common neighbor of  $u_j,u_k$ and $u'$,
then  $(1+\rho)(x_{v'}-x_{v_0})=x_{u'}-x_{u_i}>0$, so $x_{v'}>x_{v_0}$, a contradiction. This shows that  $u_j,u_k$ and $u'$ have no common neighbor in $G$, so
 $u'$ is not adjacent to all vertices of $K\setminus\{v_0\}$.

Note that $N_G(u_j)\cup N_G(u_k)\cup N_G(u')=K$.  Then  $G-u_iv_0+u'v_0\in\mathbb{G}_n$, and by Lemma \ref{perron}, we have $\rho(G-u_iv_0+u'v_0)>\rho$, a contradiction.
It now follows that  $d_G(u')\ge3$, $N_I(v_1)\setminus\{u_1,u_2,u_3\}=\{u'\}$ and $N_G(u_j)\cap N_G(u_k)=\{v_0\}$.
\end{proof}

\begin{Claim} \label{jk} 
$N_G(u_j)= K\setminus\{v_1\}$, $N_G(u_k)=\{v_0, v_1\}$ and $N_G(u')=K\setminus\{v_0\}$. 
\end{Claim}

\begin{proof}
Take $u_\ell\in\{u_1,u_2,u_3\}$ such that $u_\ell\notin N_G(v_1)$. 
We claim that $x_{u_\ell}\ge x_{u_i}$. It is obvious if $d_G(u_\ell)=2$.
If $d_I(v_1)=3$, then, by Claim \ref{Fis} and the choice of $v_0$, we have
$x_{u_\ell}\ge x_{u'}> x_{u_i}$. 
If $d_I(v_1)=2$ and $d_G(u_\ell)=3$, then, as $\rho(x_{v_1}-x_{u'})>x_{u'}-x_{v_1}$, we have $x_{v_1}>x_{u'}$, so
\[
(\rho+1)\rho(x_{u_\ell}-x_{u_i})=(\rho+1)\left(\sum_{v\in N_G(u_\ell)\setminus\{v_0\}}x_v-x_{v_1}\right)> 2x_{u_\ell}+x_{v_1}-x_{u_i}-x_{u'}>x_{u_\ell}-x_{u_i},
\] 
implying that $x_{u_\ell}>x_{u_i}$.
It follows that  $x_{u_\ell}\ge x_{u_i}$, as claimed. 

Suppose that $N_G(u_\ell)\ne K\setminus\{v_1\}$. Then $G-u_iv_1+u_\ell v_1\in\mathbb{G}_n$, and we have by Lemma \ref{perron} that $\rho(G-u_iv_1+u_\ell v_1)>\rho$, a contradiction.
So $N_G(u_\ell)= K\setminus\{v_1\}$. 
By Claim \ref{Fis}, we have $N_G(u_j)\cap N_G(u_k)=\{v_0\}$. 
Since $d_G(u_k)\ge2$ and $x_{u_j}\ge x_{u_k}$, we have $\ell=j$, and $N_G(u_k)=\{v_0, v_1\}$. 
As $G$ is $K_{1,4}$-free, we have $N_G(u')=K\setminus\{v_0\}$. 
\end{proof}

From Claim \ref{jk} and $A(G)\mathbf{x}=\rho\mathbf{x}$, we have $x_{u_i}=x_{u_k}$ and  $x_{v_0}=x_{v_1}$, then $\rho(x_{u_i}+x_{u_k})=2(x_{v_0}+x_{v_1})=4x_{v_0}$.

Now we show that $d_I(v)=2$ for any $v\in K\setminus\{v_0,v_1\}$. Otherwise,
assume that $v''\in K\setminus\{v_0,v_1\}$ with $d_I(v'')=3$. Let $u''\in N_I(v'')\setminus\{u_j,u'\}$. 
%such that $x_{v''}=\min\{x_v: v\in K\setminus\{v_0,v_1\}, d_I(v)=3\}$. Let $N_I(v'')=\{u_j,u',u_4\}$.
%As $d_G(u_4)\ge2$, there are at least two vertices of $K\setminus\{v_0,v_1\}$ whose degree in $I$ is $3$.
Note that $d_G(u'')\ge2$. Then $\rho x_{u''}\ge 2x_{v''}$ by Claim \ref{jk}. So
\[
\rho(1+\rho)(x_{v''}-x_{v_0})=\rho(x_{u''}+x_{u'}-x_{u_k}-x_{u_i})> 2x_{v''}+2x_{v''}+x_{v_1}-4x_{v_0}>4(x_{v''}-x_{v_0}),
\]
implying that $x_{v''}>x_{v_0}$, a contradiction.
% So we have
%\[
%0\ge(\rho^2(G)+\rho-4)(x_{v''}-x_{v_0})\ge x_{v_1}>0,
%\]
%a contradiction.

Thus,  $N_I(v)=\{u_j,u'\}$ for any $v\in K\setminus\{v_0,v_1\}$, implying that $|I|=4$. 
Let $G''=G-u_iv_1-u_kv_0+\{u_kv: v\in K\setminus\{v_0,v_1\}\}$. Note that $G''\cong\Gamma_{n,4}\in\mathbb{G}_n$.
By a direct calculation, we have $\Gamma_{8,4}=4.5722>4.46=\rho$, a contradiction. So $n\ge 9$, implying that $|K\setminus\{v_0,v_1\}|\ge3$. 
Fix $v''\in K\setminus\{v_0,v_1\}$. Then we have by Claim \ref{jk} that $x_v=x_{v''}$ for any $v\in K\setminus\{v_0,v_1\}$.
%Note that $\rho(x_{u'}-x_{u_i}-x_{u_k})\ge x_{v_1}+3x_{v''}-2x_{v_0}-2x_{v_1}= 3(x_{v''}-x_{v_0})$.
So $\rho(\rho+1)(x_{v''}-x_{v_0})=\rho(x_{u'}-x_{u_i}-x_{u_k})\ge 3(x_{v''}-x_{v_0})$, 
implying that $x_{v''}\ge x_{v_0}$.
Then
\begin{align*}
\rho(G'')-\rho& \ge\mathbf{x}^\top(A(G')-A(G))\mathbf{x}
 =2\sum_{v\in K\setminus\{v_0,v_1\}}x_{u_k}x_v-2x_{u_i}x_{v_1}-2x_{u_k}x_{v_0}\\
& =2x_{u_k}\left(\sum_{v\in K\setminus\{v_0,v_1\}}x_v-2x_{v_0}\right)
 \ge2x_{u_k}(3x_{v''}-2x_{v_0})
 >0,
\end{align*}
a contradiction.

\noindent
{\bf Case 1.2.} $d_I(v_1)=1$, or $d_I(v_1)\ge2$ and $N_I(v_1)\subseteq\{u_1,u_2,u_3\}$.

Suppose first that $d_I(v_1)=1$. Let $v_2$ be a neighbor of $u_k$ different from $v_0$. Then
\[
\rho (\rho+1)(x_{v_2}-x_{v_1})\ge \rho(x_{u_k}-x_{u_i})\ge x_{v_2}-x_{v_1},\]
%so
%\[
%\rho (\rho+1)(x_{v_2}-x_{v_1})\ge x_{v_2}-x_{v_1},
%\]
implying that $x_{v_2}\ge x_{v_1}$, from which, we have  $x_{u_k}\ge x_{u_i}$. So $i=3$, $k=2$ and $j=1$.
As $G$ is $K_{1,4}$-free, $d_I(v_1)=1$ and $d_G(u_3)=2$, we have $d_I(v)\le2$ for any $v\in K\setminus\{v_0,v_1\}$.
Together with $x_{u_1}\ge x_{u_2}$ and $|I|\ge4$, we have $N_G(u_2)\ne K\setminus \{v_1\}$.
So  $G-u_3v_1+u_2v_1\in\mathbb{G}_n$, and by Lemma \ref{perron}, we have $\rho(G-u_3v_1+u_2v_1)>\rho$, a contradiction.

Suppose next that $d_I(v_1)\ge2$ and $N_I(v_1)\subseteq\{u_1,u_2,u_3\}$.
We claim that $x_{u}\ge x_{u_i}$ for any $u\in I\setminus\{u_1,u_2,u_3\}$, where the equality holds only if $d_G(u)=2$. It is obvious if $d_G(u)=2$.
Note that if $N_I(v_1)\setminus\{u_i,u_k\}\ne\emptyset$, then $\rho(x_{v_1}-\sum_{u\in N_I(v_1)\setminus\{u_i,u_k\}}x_u)=\rho(x_{v_1}-x_{u_j})>x_{u_j}-x_{v_1}$, i.e., $x_{v_1}>x_{u_j}$. Similarly, we have $x_{v_0}>x_{u_j}$.
If $d_G(u')\ge3$ for some $u'\in I\setminus\{u_1,u_2,u_3\}$, then
\begin{align*}
(1+\rho)\rho(x_{u'}-x_{u_i})& =(1+\rho)\left(\sum_{v\in N_G(u')}x_v-x_{v_0}-x_{v_1}\right)\\
& > 2x_{u_k}+3x_{u'}+x_{v_0}+x_{v_1}-2x_{u_i}-x_{u_j}-x_{u_k}-\sum_{u\in N_I(v_1)\setminus\{u_i\}}x_u\\
& >3x_{u'}-2x_{u_i},
\end{align*}
so $(\rho^2(G)+\rho-2)(x_{u'}-x_{u_i})>x_{u'}>0$, as claimed.
As $|I|\ge 4$, we can take $u^*\in I\setminus\{u_1,u_2,u_3\}$. Note that $G-u_iv_1+u^*v_1\in \mathbb{G}_n$. By Lemma \ref{perron}, we have  $\rho(G-u_iv_1+u^*v_1)>\rho$, a contradiction.

\noindent
{\bf Case 2.} $i\ne 1,2,3$.

Assume that $i=4$.
Let $N_G(u_4)=\{v_1,v_2\}$. Then  $|N_I(v_a)\cap\{u_1,u_2,u_3\}|=1$ for $a=1,2$. Otherwise,
assume that  $|N_I(v_1)\cap\{u_1,u_2,u_3\}|=2$.  Let $u_j=\{u_1,u_2,u_3\}\setminus N_G(v_1)$. Then  $x_{u_j}\ge x_{u_4}$ by the choice of $v_0$. If $N_G(u_j)\ne K\setminus\{v_1\}$, then we can check that
$G-u_4v_1+u_jv_1\in\mathbb{G}_n$, and by Lemma \ref{perron}, we have $\rho(G-u_4v_1+u_jv_1)>\rho$, a contradiction. So $N_G(u_j)=K\setminus\{v_1\}$. Then we can choose $u_k\in N_I(v_1)\setminus\{u_4\}$ and $u_k\notin N_G(v_2)$.
Let $u_r=\{u_1,u_2,u_3\}\setminus\{u_j,u_k\}$.
Suppose that $u_r\in N_I(v_2)$. Then $x_{u_k}\ge x_{u_4}$ by the choice of $v_0$.
If $N_G(u_k)\ne K\setminus\{v_2\}$, then $G-u_4v_2+u_kv_2\in\mathbb{G}_n$ and we have by Lemma \ref{perron} that $\rho(G-u_4v_2+u_kv_2)>\rho$, a contradiction. So $N_G(u_k)= K\setminus\{v_2\}$. Then $|I|\ge5$, as otherwise $G$ is Hamiltonian, which is a contradiction.
Take $u'\in I\setminus\{u_1,\dots,u_4\}$.
We claim that $x_{u'}\ge x_{u_4}$. It is obvious if $d_G(u')=2$.
Note that $x_{v_1}>x_{u_r}$.
If $d_G(u')\ge3$, then $(\rho+1)\rho(x_{u'}-x_{u_4})> x_{v_0}+x_{v_1}+2x_{u_j}+2x_{u_k}+3x_{u'}-2x_{u_r}-2x_{u_4}>2(x_{u'}-x_{u_4})$, so $x_{u'}>x_{u_4}$ as desired.
So $G-u_4v_2+u'v_2\in \mathbb{G}_n$ and we have by Lemma \ref{perron} that $\rho(G-u_4v_2+u'v_2)>\rho$, a contradiction.
Thus $u_r\notin N_I(v_2)$. Note that $G+u_rv_2\in \mathbb{G}_n$. By Lemma \ref{addedges}, we have $\rho(G+u_rv_2)>\rho$, a contradiction.

%As $\rho(x_{u_k}-x_{u_4})\ge x_{v_0}-x_{v_2}$, $(\rho+1)(x_{v_0}-x_{v_2})>x_{u_k}-x_{u_4}$ if $d_I(v_2)=2$ and $x_{v_0}\ge x_{v_2}$ if $d_I(v_2)=3$, we have
%\[
%(\rho^2(G)+\rho-1)(x_{u_k}-x_{u_4})\ge0,
%\]
%implying that $x_{u_k}\ge x_{u_4}$.
%Note that $G-u_4v_2+u_kv_2\in \mathbb{G}_n$. So we have by Lemma \ref{perron} that $\rho(G-u_4v_2+u_kv_2)>\rho$, a contradiction.

\noindent
{\bf Case 2.1.}
 $N_I(v_1)\cap\{u_1,u_2,u_3\}=N_I(v_2)\cap\{u_1,u_2,u_3\}$.

Assume that $N_I(v_1)\cap\{u_1,u_2,u_3\}=\{u_r\}$.
Let $1\le s<t\le3$ with $s,t\ne r$.
Then we claim that $x_{u_4}\le x_{u_s}$. It is obvious that $d_G(u_s)=2$.
Suppose that $d_G(u_s)\ge3$. We show that $x_{v_1}$, $x_{v_2}\le x_{v_0}$. 
Assume that $d_I(v_1)\le d_I(v_2)$. 
It is obvious that $d_I(v_1)=3$ by the choice of $v_0$. 
If $d_I(v_1)=2$, then $x_{v_1}+x_{v_2}-2x_{v_0}\le x_{v_1}-x_{v_0}$ if $d_I(v_2)=3$, and $x_{v_1}+x_{v_2}-2x_{v_0}=2(x_{v_1}-x_{v_0})$ otherwise. Note that 
$(\rho+1)\rho(x_{v_1}-x_{v_0})=\rho(x_{u_4}-x_{u_s}-x_{u_t})<x_{v_1}+x_{v_2}-2x_{v_0}$.
So we have $x_{v_1}\le x_{v_0}$, and $x_{v_2}\le x_{v_0}$. 
Thus 
\begin{align*}
(\rho+1)\rho(x_{u_4}-x_{u_s})& =(\rho+1)\left(x_{v_1}+x_{v_2}-x_{v_0}-\sum_{v\in N_G(u_s)\setminus\{v_0\}}x_v\right) \le (\rho+1)\left(x_{v_0}-\sum_{v\in N_G(u_s)\setminus\{v_0\}}x_v\right)\\
& \le x_{u_t}+x_{u_r}-x_{u_s}-x_{v_0} <0
\end{align*}
as $x_{u_t}\le x_{u_s}$ and $x_{u_r}<x_{v_0}$.
So $x_{u_4}<x_{u_s}$, as claimed.
Then $G-u_4v_1+u_sv_1\in\mathbb{G}_n$ and we have by Lemma \ref{perron} that $\rho(G-u_4v_1+u_sv_1)>\rho$, a contradiction.

%Suppose first that $d_I(v_1)=d_I(v_2)=2$.
%If $N_I(v_k)\cap\{u_1,u_2,u_3\}=\{u_1\}$ for $k=1,2$, then $(1+\rho)\rho(x_{u_4}-x_{u_2})=(1+\rho)(x_{v_1}+x_{v_2}-x_{v_0}-\sum_{v\in N_G(u_2)\setminus\{v_0\}}x_v)\le2x_{u_4}-2x_{u_2}-x_{u_3}$, i.e.,
%\[
%(\rho^2(G)+\rho-2)(x_{u_4}-x_{u_2})<0,
%\]
%so $x_{u_4}<x_{u_2}$.
%Similarly, if $N_I(v_k)\cap\{u_1,u_2,u_3\}=\{u_2\}$ for $k=1,2$, then $x_{u_4}<x_{u_3}\le x_{u_1}$, and if $N_I(v_k)\cap\{u_1,u_2,u_3\}=\{u_3\}$, then
%$x_{u_4}<x_{u_1}$
%%Similarly, if $\{v_1,v_2\}\subseteq B_{jj}$, we can obtain that $x_{u_4}<x_{u_1}$ (we can check that $x_{u_4}<x_{u_3}$ if $j=2$).
%Let $G'=G-u_4v_1+u_2v_1$ if $N_I(v_k)\cap\{u_1,u_2,u_3\}=\{u_1\}$, and $G'=G-u_4v_1+u_1v_1$ otherwise.
%Then it is easy to see that $G'\in\mathbb{G}_n$, and by Lemma \ref{perron}, we have $\rho(G')>\rho$, a contradiction.
%
%
%Suppose next that  $d_I(v_k)=3$ for some $k=1,2$.
%Assume that $d_I(v_1)=3$. Let $N_I(v_1)=\{u_r,u_4,u_5\}$ with $r\in\{1,2,3\}$. Similarly, we have $x_{u_4}\le x_{u_5}$.
%As $G$ is $K_{1,4}$-free and Lemma \ref{perron}, $K\setminus N_G(u_r)\subset N_G(u_5)$.
%Note that $2x_{u_4}\le x_{u_4}+x_{u_5}\le x_{u_s}+x_{u_t}\le 2x_{u_s}$, where $1\le s<t\le3$ and $s,t\ne r$, so $x_{u_4}\le x_{u_s}$.Note that $G-u_4v_1+u_sv_1\in\mathbb{G}_n$, and by Lemma \ref{perron}, we have $\rho(G-u_4v_1+u_sv_1)>\rho$, a contradiction.

\noindent
{\bf Case 2.2.} $N_I(v_1)\cap\{u_1,u_2,u_3\}\ne N_I(v_2)\cap\{u_1,u_2,u_3\}$.

Assume that $N_I(v_1)\cap\{u_1,u_2,u_3\}=\{u_s\}$ and $N_I(v_2)\cap\{u_1,u_2,u_3\}=\{u_t\}$ with  $1\le s<t\le3$.

If $t=2$ and $N_G(u_2)=K\setminus\{v_1\}$, then $s=1$, and $N_G(u_1)=K\setminus\{v_2\}$ as $x_{u_1}\ge x_{u_2}$.
We claim that $x_{u_3}\ge x_{u_4}$.
Note that 
\[
\rho(x_{u_3}-x_{u_4})\ge 2x_{v_0}-x_{v_1}-x_{v_2}.
\] 
If $d_I(v_1)=3$ and $d_I(v_2)=3$, then $\rho(x_{u_3}-x_{u_4})\ge 0$ by the choice of $v_0$, so $x_{u_3}\ge x_{u_4}$.
Note that $(\rho+1)(2x_{v_0}-x_{v_1}-x_{v_2})>2(x_{u_3}-x_{u_4})$ if $d_I(v_1)=d_I(v_2)=2$, and $(\rho+1)(2x_{v_0}-x_{v_1}-x_{v_2})>x_{u_3}-x_{u_4}$ otherwise. So $(\rho^2+\rho-2)(x_{u_3}-x_{u_4})>0$ or $(\rho^2+\rho-1)(x_{u_3}-x_{u_4})>0$, implying that $x_{u_3}>x_{u_4}$, as claimed.
Note that $G-u_4v_1+u_3v_1\in \mathbb{G}_n$. Then by Lemma \ref{perron}, we have $\rho(G-u_4v_1+u_3v_1)>\rho$, a contradiction. 

Thus $t\ne2$ or $t=2$ and $N_G(u_2)\ne K\setminus\{v_1\}$.
We claim that $x_{u_4}<x_{u_2}$.
If $s=2$ or $t=2$, then $(1+\rho)\rho(x_{u_4}-x_{u_2})\le(1+\rho)(x_{v_\ell}-x_{v_0})$, where $\ell=1$ if $t=2$, and $\ell=2$ otherwise. Also, we have  $x_{v_l}-x_{v_0}\le0$ if $d_I(v_\ell)=3$, and  $(1+\rho)(x_{v_l}-x_{v_0})< x_{u_4}-x_{u_2}$ otherwise, so  $x_{u_4}<x_{u_2}$.
If $s\ne 2$ and $t\ne 2$, then $s=1$ and $t=3$, similarly, we have $x_{u_4}<x_{u_3}\le x_{u_2}$ as desired.
If $s=2$, then $N_G(u_2)\ne K\setminus\{v_2\}$ as otherwise $\rho(x_{u_2}-x_{u_1})>x_{v_1}>0$ which is a contradiction.
Let $G'=G-u_4v_1+u_2v_1$ if $s\ne2$, and  $G'=G-u_4v_2+u_2v_2$ otherwise. Then $G'\in\mathbb{G}_n$, and we have by Lemma \ref{perron} that $\rho(G')>\rho$, a contradiction.
\end{proof}

\begin{proof}[Proof of Theorem~\ref{N2}]
Suppose that $G=(K,I)$ is a graph in $\mathbb{G}_n$ with  maximum spectral radius.
Denote by $\mathbf{x}$ the  Perron vector of $G$.

Let  $I=\{u_1,\dots, u_{|I|}\}$. For any $v\in V(G)$, $d_G(v)\ge 1$.

If $|I|=1$, then $d_G(u_1)=1$, as otherwise $G$ contains a Hamiltonian cycle, a contradiction, so $G\cong \Gamma_{n,1}$.

Suppose that $|I|=2$. Assume that $d_G(u_1)\ge d_G(u_2)$.
If $d_G(u_2)\ge 2$, it is easy to see that $G$ contains a Hamiltonian cycle unless $d_G(u_1)=2$ and $N_G(u_1)=N_G(u_2)$, so $G\cong\Gamma_{n,2}'$, see Fig. \ref{fth}.
If $d_G(u_2)=1$, then  $d_G(u_1)\ge n-3$, as otherwise, $d_G(u_1)\le n-4$, so we can add an edge
$u_1v_1$ to $G$ for some vertex $v_1\in K\setminus N_G(u_1)$. Evidently,  $G+u_1v_1$ is a connected $K_{1,4}$-free split graph of type $(K, I)$, and as  $d_{G+u_1v_1}(u_2)=d_G(u_2)=1$,  it is not Hamiltonian. However, by Lemma \ref{addedges}, we have $\rho(G+u_1v_1)>\rho(G)$, a contradiction.
As $K$ is a maximum clique of $G$, $d_G(u_1)\le n-3$, so $d_G(u_1)=n-3$.
If $u_2$ is adjacent to one  vertex of $N_G(u_1)$, then $G\cong\Gamma_{n,2}$, and otherwise, $G\cong \Gamma_{n,2}''$, see Fig. \ref{f1n}. Thus, $G\cong\Gamma_{n,2}, \Gamma_{n,2}', \Gamma_{n,2}''$. 
By a direct calculation,  $\rho(\Gamma_{5,2}')=3>\rho(\Gamma_{5,2})=2.6855>\rho(\Gamma_{5,2}'')=2.6412$. 
It follows that $G\cong\Gamma_{n,2}'$ if $n=5$. By Lemma \ref{bbb}, we have $G\cong\Gamma_{n,2}$ if $n\ge6$.

Suppose in the following that $|I|\ge 3$.
Since $G$ is $K_{1,4}$-free and $I$ is an independent set of $G$,
 $d_I(v)\le 3$ for any vertex $v\in K$.

%
%\begin{Claim}\label{K14}
%If for any $v\in K$, $d_I(v)=2$, then $G$ is $K_{1,4}$-free.
%\end{Claim}
%\begin{proof}
%Suppose that $G$ contains $K_{1,4}$ as an induced subgraph.
%Then the center of $K_{1,4}$ must be in $K$, so there have at least two vertices with degree $1$ in $V(K_{1,4})$ that are in $K$, which are adjacent, a contradiction.
%\end{proof}

%For any vertex $v\in K$, $d_I(v)\le 3$
%By Claim \ref{Cd}, we will consider two cases separately.

\noindent{\bf Case 1.} $\forall v\in K$, $d_I(v)\le2$.

Let $u_1$ be a vertex in $I$ with $x_{u_1}=\max\{x_u: u\in I\}$. Let $K'=N_G(u_1)$ and $K''=K\setminus K'$.

\begin{Claim}\label{KK}
$d_G(u_1)\ge 2$, i.e., $|K'|\ge2$.
\end{Claim}

\begin{proof}
Suppose to the contrary that $|K'|=1$.

First, we show that $d_G(u)=1$ for any $u\in I$. Suppose that this is not true. Then $d_G(u_0)\ge 2$ for some $u_0\in I\setminus\{u_1\}$. Assume that $K'=\{v_1\}$. Then $v_1\not\in N_G(u_0)$, as otherwise we have from $A(G)x=\rho(G)x$ that $\rho(G)x_{u_1}=x_{v_1}<\rho(G)x_{u_0}$, contradicting the choice of $u_1$.
Fix a vertex $v_0\in N_G(u_0)$. Let
\[
G'=G-\{u_0v: v\in N_G(u_0)\setminus\{v_0\}\}+\{u_1v: v\in N_G(u_0)\setminus\{v_0\}\}.
\]
Evidently, $G'$ is a connected $K_{1,4}$-free split graph of the type $(K,I)$. As  $d_{G'}(u_0)=1$, $G'$ is not a Hamiltonian. However, we have by Lemma \ref{perron} that  $\rho(G')>\rho(G)$, a contradiction. Indeed, we have $d_G(u)=1$ for any $u\in I$.

As $|K|\ge|I|$, there is a vertex $v'\in K''$ such that $d_I(v')\le 1$. It is obvious that
 $G+u_1v'\in\mathbb{G}_n$. By Lemma \ref{addedges}, we have $\rho(G+u_1v')>\rho(G)$, a contradiction.
%Thus, $|K'|\ge2$.
\end{proof}

Let $I'=\{u\in I\setminus\{u_1\}: d_G(u)\ge2\}$.

\begin{Claim}\label{Ione}
$|I'|=1$.
\end{Claim}

\begin{proof}
Suppose first that $|I'|=0$. As $|K|\ge|I|$, we can find $v'\in K$ such that $d_I(v')\le 1$. Take $u^*\in I\setminus N_I(v')$ with $d_G(u^*)=1$. Then $G+u^*v'$ is a connected $K_{1,4}$-free split graph of type $(K, I)$, and is not Hamiltonian as $|I'|=0$ and $|I|\ge3$, and we have by Lemma \ref{addedges} that $\rho(G+u^*v')>\rho(G)$, a contradiction.
It follows that $|I'|\ge1$.

Suppose next that $|I'|\ge2$.
Let $u'\in I'$ with $x_{u'}=\max\{x_u: u\in I'\}$.
Suppose that for some $u''\in I'\setminus\{u'\}$, $N_G(u'')\subseteq N_G(u')$.
As $d_I(v)\le 2$ for any $v\in K$, we have  $N_G(u'')\subseteq K''$, so $|K''|\ge 2$.
Fix a vertex $v_1\in N_G(u'')$.
Let
\[
G'=G-\{u''v: v\in N_G(u'')\setminus \{v_1\}\}+\{u_1v: v\in N_G(u'')\setminus \{v_1\}\}.
 \]
It is a connected $K_{1,4}$-free split graph of type $(K, I)$, and it is not Hamiltonian as $d_{G'}(u'')=1$.
As $x_{u_1}\ge x_{u''}$, we have by Lemma \ref{perron} that $\rho(G')>\rho(G)$, a contradiction.
Thus, for any $u\in I'\setminus\{u'\}$,  $N_G(u)\nsubseteq N_G(u')$.
Take $u_0\in I'\setminus\{u'\}$.
%Let $v_0\in N_G(u_0)\setminus N_G(u')$.
Let $v_0\in N_G(u_0)\cap K''$ if $N_G(u_0)\cap K''\ne\emptyset$, and $v_0\in N_G(u_0)\cap K'$ otherwise.
As $x_{u_1}\ge x_{u_0}$, $K'\nsubseteq N_G(u_0)$ if $N_G(u_0)\cap K''\ne\emptyset$.
Then  $K'\cup (N_G(u_0)\cap K''\setminus\{v_0\})\ne K$ and $N_G(u')\cup (N_G(u_0)\cap K'\setminus\{v_0\})\ne K$.
Let
\[
G''=G-\{u_0v:v\in N_G(u_0)\setminus\{v_0\}\}+\{u_1v: v\in N_G(u_0)\cap K''\setminus\{v_0\}\}+\{u'v: v\in N_G(u_0)\cap K'\setminus\{v_0\}\}.
\]
It is a connected $K_{1,4}$-free split graph of type $(K, I)$, and  is not Hamiltonian as $d_{G'}(u_0)=1$.
As
\[
\rho(G'')-\rho(G)\ge \sum_{\substack{v\in N_G(u_0)\cap K''\\ v\ne v_0}}2(x_{u_1}-x_{u_0})x_v+\sum_{\substack{v\in N_G(u_0)\cap K'\\v\ne v_0}}2(x_{u'}-x_{u_0})x_v\ge0,
\]
we have $\rho(G'')=\rho(G)$, and hence $\mathbf{x}$ is also the Perron vector of $G''$.  Now
\[
(\rho(G'')-\rho(G))x_{u_1}=\sum_{v\in N_G(u_0)\cap K''\setminus\{v_0\}}x_v=0
 \]
\[
(\rho(G'')-\rho(G))x_{u'}=\sum_{v\in N_G(u_0)\cap K'\setminus\{v_0\}}x_v=0,
 \]
so $N_G(u_0)\cap K''\setminus \{v_0\}=\emptyset$ and $N_G(u_0)\cap K'\setminus\{v_0\}=\emptyset$, i.e., $N_G(u_0)=\{v_0\}$, a contradiction.
\end{proof}

%\noindent{\bf Case 1.1.} $|K''|=1$, say $K''=\{v_0\}$.
Let $I'=\{u_2\}$.

\begin{Claim}\label{11dk}
For any vertex $v\in K'$, $d_I(v)=2$.
\end{Claim}
\begin{proof}
Evidently, $d_I(v)\ge 1$ for any vertex $v\in K'$.
Suppose that for some $v_1\in K'$, $d_I(v_1)=1$. Then $N_I(v_1)=\{u_1\}$.
Let $I_1=\{u\in I: d_G(u)=1\}$.
Then $|I_1|\ge1$ by Claim \ref{Ione} and $|I|\ge3$.
Let $u'\in I_1$.
If $|I_1|\ge2$, then  $G+u'v_1$ is a connected $K_{1,4}$-free split graph of type $(K,I)$  and is not  Hamiltonian graph as $|I_1|\ge2$, but we have by Lemma \ref{addedges} that $\rho(G+u'v_1)>\rho(G)$, a contradiction.
So $I_1=\{u'\}$, implying that $|I|=3$.
%Let $I'=\{u_2\}$.

Suppose that $K'\setminus \{v_1\}\subseteq N_G(u_2)$.
As $d_I(v)\le2$ for any $v\in K$, we have $N_I(v)=\{u_1, u_2\}$ for any $v\in K'\setminus \{v_1\}$.
As $G$ is connected, we have $N_G(u')\subseteq K''$, say $N_G(u')=\{v_0\}$.
If $v_0\in N_G(u_2)$, then  $(\rho(G)+1)(x_{v_1}-x_{v_0})=x_{u_1}-x_{u_2}-x_{u'}$ and $\rho(G)(x_{u_1}-x_{u_2})=x_{v_1}-x_{v_0}$, so $(\rho^2(G)+\rho(G)-1)(x_{u_1}-x_{u_2})=-x_{u'}<0$, implying that $x_{u_1}<x_{u_2}$, a contradiction.
So $v_0\notin N_G(u_2)$.
Now it is easy to see that $G+u_2v_1\in\mathbb{G}_n$. By Lemma \ref{addedges}, we have $\rho(G+u_2v_1)>\rho(G)$, a contradiction.
So there are two vertices in $K'$, say $v_1,v_2$,  neither is adjacent to $u_2$.
Clearly, $G+u_2v_1\in\mathbb{G}_n$. By Lemma \ref{addedges}, we have $\rho(G+u_2v_1)>\rho(G)$, a contradiction.
\end{proof}

\begin{Claim}\label{u1u2}
$N_G(u_2)\subseteq K'$.
\end{Claim}
\begin{proof}
Suppose that $v_1\in N_G(u_2)\cap K''$.
Then there is $v_2\in K'$ and $v_2\notin N_G(u_2)$, as otherwise $\rho(G)(x_{u_1}-x_{u_2})\le -x_{v_1}<0$, a contradiction.
By Claim \ref{11dk}, there is $u'\in I$ such that $u'\in N_I(v_2)$.
Then we claim that $x_{u_2}>x_{u'}$.
If $N_G(u_2)\cap K'\ne\emptyset$, say $v'\in N_G(u_2)\cap K'$, then $(\rho(G)+1)\rho(G)(x_{u_2}-x_{u'})\ge(\rho(G)+1)(x_{v_1}+x_{v'}-x_{v_2})>x_{u_2}-x_{u'}$, implying that $x_{u_2}-x_{u'}>0$.
And if $N_G(u_2)\cap K'=\emptyset$, then $(\rho(G)+1)\rho(G)(x_{u_2}-x_{u'})>x_{v_2}+2x_{u_2}-x_{u_1}-x_{u'}$, and as $\rho(G)(x_{v_2}-x_{u_1})>x_{u_1}-x_{v_2}$, we have $x_{v_2}-x_{u_1}>0$, so $(\rho^2(G)+\rho(G)-1)(x_{u_2}-x_{u'})>0$, i.e., $x_{u_2}-x_{u'}>0$ as desired.

Note that
%$\rho(G)\left(x_{u'}-\sum_{u\in N_I(v_1)\setminus\{u_2\}}x_u\right)\ge x_{v_2}-x_{v_1}$.
\[
\rho(G)(\rho(G)+1)(x_{v_2}-x_{v_1})=\rho(G)\left(x_{u_1}-x_{u_2}+x_{u'}-\sum_{u\in N_I(v_1)\setminus\{u_2\}}x_u\right).
\]
If $d_I(v_1)=1$, then $\rho(G)(\rho(G)+1)(x_{v_2}-x_{v_1})>0$.
And if $d_I(v_1)=2$, then $\rho(G)\left(x_{u'}-\sum_{u\in N_I(v_1)\setminus\{u_2\}}x_u\right)=x_{v_2}-x_{v_1}$,
and so
\[
\rho(G)(\rho(G)+1)(x_{v_2}-x_{v_1})\ge x_{v_2}-x_{v_1}.
\]
Thus in either case, we have $x_{v_2}\ge x_{v_1}$.
Let $G'=G-u_2v_1-u'v_2+u_2v_2+u'v_1$. Then $G'\in\mathbb{G}_n$.
So
\[
\rho(G')-\rho(G)\ge \mathbf{x}^\top(A(G')-A(G))\mathbf{x}=2(x_{u_2}-x_{u'})(x_{v_2}-x_{v_1})\ge0,
\]
i.e., $\rho(G')=\rho(G)$ and $\mathbf{x}$ is also the Perron vector of $G'$,
however, we have $0=(\rho(G')-\rho(G))x_{v_2}=x_{u_2}-x_{u'}>0$, a contradiction.
\end{proof}

\noindent{\bf Case 1.1.} $|I|=3$.

If there is $v'\in K''$ such that $d_I(v')=0$, then $N_G(u_2)\ne K\setminus\{v'\}$ by Claim \ref{u1u2}, and $G+u_2v'\in\mathbb{G}_n$, and we have by Lemma \ref{addedges} that $\rho(G+u_2v')>\rho(G)$, a contradiction.
So $d_I(v)\ge1$ for any $v\in K''$. By Claims \ref{Ione} and \ref{u1u2}, we have $|K''|=1$ and $d_I(v)=1$ for $v\in K''$. By Lemma \ref{addedges} and Claims \ref{11dk} and \ref{u1u2}, as well as Lemma \ref{I3}, we have $G\cong\Gamma_{n,3}$ with $n\ge7$.

\noindent{\bf Case 1.2.} $|I|\ge4$.

If there exists $v'\in K''$ such that $d_I(v')\le1$, then $N_G(u_2)\ne K\setminus\{v'\}$ by Claim \ref{u1u2}. Then $G+u_2v'\in\mathbb{G}_n$, and we have by Lemma \ref{addedges} that $\rho(G+u_2v')>\rho(G)$, a contradiction.
So $d_I(v)=2$ for any $v\in K''$.
Thus $G\cong\Gamma$ as displayed in Fig. \ref{f400}.

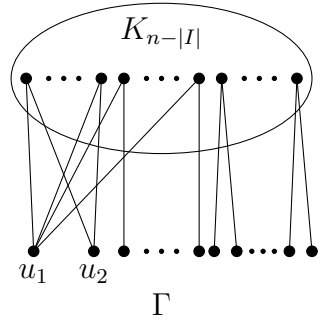
\begin{figure}[ht]
\centering
\begin{tikzpicture}

\draw (2,-0.8) ellipse (2 and 1);
\filldraw [black] (0.2,-0.8) circle (2pt);
\filldraw [black] (1.2,-0.8) circle (2pt);
\filldraw [black] (0.5,-0.8) circle (0.8pt);
\filldraw [black] (0.7,-0.8) circle (0.8pt);
\filldraw [black] (0.9,-0.8) circle (0.8pt);
\filldraw [black] (1.5,-0.8) circle (2pt);
\filldraw [black] (1.8,-0.8) circle (0.8pt);
\filldraw [black] (2.0,-0.8) circle (0.8pt);
\filldraw [black] (2.2,-0.8) circle (0.8pt);
\filldraw [black] (2.5,-0.8) circle (2pt);
\filldraw [black] (2.8,-0.8) circle (2pt);
\filldraw [black] (3.1,-0.8) circle (0.8pt);
\filldraw [black] (3.3,-0.8) circle (0.8pt);
\filldraw [black] (3.5,-0.8) circle (0.8pt);
\filldraw [black] (3.8,-0.8) circle (2pt);
\draw  [black](0.3,-3.1)--(0.2,-0.8);
\draw  [black](0.3,-3.1)--(1.2,-0.8);
\draw  [black](0.3,-3.1)--(1.5,-0.8);
\draw  [black](0.3,-3.1)--(2.5,-0.8);
\filldraw [black] (0.3,-3.1) circle (2pt);
\filldraw [black] (1.1,-3.1) circle (2pt);
\draw  [black](1.1,-3.1)--(0.2,-0.8);
\draw  [black](1.1,-3.1)--(1.2,-0.8);
\filldraw [black] (1.5,-3.1) circle (2pt);
\filldraw [black] (1.8,-3.1) circle (0.8pt);
\filldraw [black] (2.0,-3.1) circle (0.8pt);
\filldraw [black] (2.2,-3.1) circle (0.8pt);
\filldraw [black] (2.5,-3.1) circle (2pt);
\draw  [black](1.5,-3.1)--(1.5,-0.8);
\draw  [black](2.5,-3.1)--(2.5,-0.8);
\filldraw [black] (2.7,-3.1) circle (2pt);
\filldraw [black] (3.0,-3.1) circle (2pt);
\draw  [black](2.7,-3.1)--(2.8,-0.8);
\draw  [black](3.0,-3.1)--(2.8,-0.8);
\filldraw [black] (3.2,-3.1) circle (0.8pt);
\filldraw [black] (3.35,-3.1) circle (0.8pt);
\filldraw [black] (3.5,-3.1) circle (0.8pt);
\filldraw [black] (3.7,-3.1) circle (2pt);
\filldraw [black] (4.0,-3.1) circle (2pt);
\draw  [black](3.7,-3.1)--(3.8,-0.8);
\draw  [black](4.0,-3.1)--(3.8,-0.8);
\node at (0.3, -3.4) {$u_1$};
\node at (1.1, -3.4) {$u_2$};
\node at (2, -0.2) {$K_{n-|I|}$};
%\node at (3.3, -0.5) {$\ge2$};
\node at (2, -3.8) {$\Gamma$};
\end{tikzpicture}
\caption{The structure of $\Gamma$.}
\label{f400}
\end{figure}

Let $V_1=N_G(u_2)\subseteq K'$, $V_2=K'\setminus V_1$, $V_3=K''$, $V_4=\{u: u\in I\setminus \{u_1, u_2\}, N_G(u)\in V_2\}$ and $V_5=I\setminus(\{u_1, u_2\}\cup V_4)$, where $|V_5|=2|K''|$, $|V_4|=|I|-|V_5|-2$, and $|V_4|=0$ if $|I|=2|K''|+2$.
By symmetry, denote by $x_i$ the entry of $\mathbf{x}$ at each vertex from $V_i$ for $i=1,\dots,5$.
Note that $(\rho(\Gamma)+1)(x_1-x_3)=x_{u_1}+x_{u_2}-2x_5$, $\rho(\Gamma)(x_{u_1}-x_5)=|V_1|x_1+|V_2|x_2-x_3$ and $\rho(\Gamma)(x_{u_2}-x_5)=|V_1|x_1-x_3$.  %$\rho(\Gamma)(x_4-x_6)=|V_1|x_1+|V_2|x_2-x_3$.
Then
\[
\rho(\Gamma)(\rho(\Gamma)+1)(x_1-x_3)=2|V_1|x_1+|V_2|x_2-2x_3,
\]
i.e.,
\[
(\rho^2(\Gamma)+\rho(\Gamma)-2)(x_1-x_3)=2(|V_1|-1)x_1+|V_2|x_2.
\]
As $|V_1|+2|V_2|+3|V_3|+2=n\ge 2|I|=2(2+|V_2|+2|V_3|)$, we have $|V_1|\ge |V_3|+2\ge 2$, so  $2(|V_1|-1)x_1+2|V_2|x_2>0$. As  $\rho^2(\Gamma)+\rho(\Gamma)-2>0$,  we have $x_1>x_3$.
Let $V_3=\{z_1,\dots,z_{|V_3|}\}$, and $\{w_1,\dots,w_{|V_3|}\}\subset V_5$ such that $N_G(w_i)\cap N_G(w_j)=\emptyset$ and $N_G(w_i)=\{z_i\}$ for any $1\le i< j\le|V_3|$.
Fix $\{v_1,\dots,v_{|V_3|}\}\subset V_1$.
Let $\Gamma'=\Gamma+v_1w_1$ if $|V_3|=1$ and
\[
\Gamma'=\Gamma-\{z_iw_i: 1\le i\le|V_3|\}+\{v_iw_i: 1\le i\le|V_3|\}+u_2z_1+\{u_1z_i: 2\le i\le|V_3|\}
%+\{u_2v: v\in V_3\setminus\cup_{i=1}^{|V_3|}N_G(w_i)\}.
\]
otherwise.
It can be checked that $\Gamma'\in\mathbb{G}_n$. As $x_{v_i}=x_1>x_3=x_{z_i}$ for any $1\le i\le |V_3|$, we have by Lemma \ref{addedges} that $\rho(\Gamma')>\rho(\Gamma)$ in the former case, and in the latter case, we have
\[
\rho(\Gamma')-\rho(\Gamma)\ge2|V_3|x_5(x_1-x_3)+2x_{u_2}x_3+2(|V_3|-1)x_{u_1}x_3>0,
\]
a contradiction.

\noindent{\bf Case 2.} $\exists v_0\in K$ such that $d_I(v_0)=3$.

\begin{Claim}\label{adj}
	If $d_I(v)=3$ for some $v\in K$, then any vertex in $K\setminus\{v\}$ is adjacent to at least one vertex of $N_I(v)$.
\end{Claim}
\begin{proof}
	If $v_1\in K\setminus\{v\}$ such that $N_I(v_1)\cap N_I(v)=\emptyset$, then $v$ has four pairwise nonadjacent neighbors, $v_1$ and those in $N_I(v)$, a contradiction.
\end{proof}

Assume that  $N_I(v_0)=\{u_1, u_2, u_3\}$ such that\[
x_{u_1}+x_{u_2}+x_{u_3}=\max\{x_{u_r}+x_{u_s}+x_{u_t}: v\in K, d_I(v)=3,N_I(v)=\{u_r, u_s, u_t\}\},
 \]
where $x_{u_1}\ge x_{u_2}\ge x_{u_3}$.
By Claim \ref{adj}, we have $N_G(u_1)\cup N_G(u_2)\cup N_G(u_3)=K$.

Let $K'=N_G(u_1)\cap N_G(u_2)\cap N_G(u_3)$. Let $A_{11}=N_G(u_1)\setminus (N_G(u_2)\cup N_G(u_3))$, and for $j=2,3$, $A_{1j}=(N_G(u_1)\cap N_G(u_j))\setminus K'$.
Let $A=N_G(u_1)\setminus K'=A_{11}\cup A_{12}\cup A_{13}$.  Let $B=K\setminus N_G(u_1)=B_{22}\cup B_{33}\cup B_{23}$, where $B_{23}=N_G(u_2)\cap N_G(u_3)$, $B_{22}$ is the set of neighbors of $u_2$ in $K\setminus(N_G(u_1)\cup N_G(u_3))$ and $B_{33}$ is the set of neighbors of $u_3$ in $K\setminus(N_G(u_1)\cup N_G(u_2))$.
Then $K=K'\cup A\cup B$.

As $K$ is the maximum clique of $G$,
%\begin{Claim}\label{BB}
$B=K\setminus N_G(u_1)\ne\emptyset$, $A_{11}\cup A_{13}\cup B_{33}=K\setminus N_G(u_2)\ne\emptyset$ and $A_{11}\cup A_{12}\cup B_{22}=K\setminus N_G(u_3)\ne\emptyset$.
%\end{Claim}

Suppose  that $|I|=3$. We claim that $n=6$ and $G\cong\Gamma_{n,3}'$.

Firstly, we show that $d_G(u_i)\ge 2$ for $i=1,2,3$.
Suppose that there exists a vertex of $I$ with degree $1$.
Then $d_G(u_3)=1$ as $x_{u_1}\ge x_{u_2}\ge x_{u_3}$, so $N_G(u_3)=K'$.
Note that $A_{13}\cup B_{33}\cup B_{23}=\emptyset$. Then  $A_{11}\ne\emptyset$ and $B=B_{22}\ne\emptyset$.
If $|A_{11}|\ge 2$, then for $v'\in A_{11}$, $G+u_2v'$ is not Hamiltonian as $d_G(u_3)=1$, so $G+u_2v'\in\mathbb{G}_n$. By Lemma \ref{addedges}, we have $\rho(G+u_2v')>\rho(G)$, a contradiction. So $|A_{11}|=1$.
Similarly, we have $|B_{22}|=1$. Then $A=A_{11}\cup A_{12}$, so $G\cong \Gamma_{n,3}''$ as  displayed in Fig. \ref{flem}. By Lemma \ref{I3},  $\rho(G)<\rho(\Gamma_{n,3})$, a contradiction. Thus we have $d_G(u_i)\ge 2$ for $i=1,2,3$, as desired.

As $G$ is not Hamiltonian and a connected $K_{1,4}$-free split graph of type $(K,I)$, it is easy to see that $|K'|=1$. Recall that $d_G(u_i)\ge 2$ for $i=1,2,3$. We claim that there are at least two vertices of $I$ with degree $2$. Otherwise, let $u_i\in I$ such that $d_G(u_i)=2$ for some $i=1,2,3$, then letting $N_G(u_i)=\{v_0,v_1\}$ and denoting by $u_j$, $u_k$ the two other vertices of $I$, we  find $v_2\ne v_3\in K\setminus\{v_0, v_1\}$ such that $v_2\in N_G(u_j)$ and $v_3\in N_G(u_k)$.
If $v_2\in N_G(u_k)$, then let $P$ be a Hamiltonian path in $G[K\setminus\{v_0,v_2\}]$ from $v_3$ to $v_1$, and otherwise, let $v'\in N_G(u_k)\setminus\{v_0,v_3\}$, which may be equal to $v_1$, and let $Q$ be a Hamiltonian path in $G[K\setminus\{v_0,v_1,v'\}]$ from $v_2$ to $v_3$.
So correspondingly,
$v_1u_iv_0u_jv_2u_kv_3P$ or  $v_1u_iv_0u_jv_2QQ'$ (with $Q'=u_kv'v_1$ or  $Q'=u_kv_1$) is a Hamiltonian cycle of $G$, a contradiction.

%there is at least one vertex of $I$ with degree $2$. Otherwise, the minimum degree of $I$ is at least $3$. Then the removal of any two vertices from $G$ results in a connected graph, so $G$ is $3$-connected, implying that $G$ is Hamiltonian by Theorem 1.3???, a contradiction.

If the degree of each vertex of $I$ are $2$, then  $G\cong\Gamma_{n,3}'$ with $n=6$, or $G\cong \Gamma'$  with $n=7$ (as displayed in Fig. \ref{ff--1}).
Suppose that $G\cong\Gamma'$ when $n=7$. Note that $\Gamma'-u_2v_2+u_2v_1\cong\Gamma_{n,3}'$ and $x_{v_1}=x_{v_2}$. By Lemma \ref{perron}, we have $\rho(\Gamma_{n,3}')>\rho(\Gamma')$, so $G\cong\Gamma_{n,3}'$ with $n=6$.

\begin{figure}[ht]
\centering
\begin{tikzpicture}

\draw (7,-0.8) ellipse (2 and 1);
\filldraw [black] (5.5,-0.8) circle (2pt);
\filldraw [black] (6.5,-0.8) circle (2pt);
\filldraw [black] (7.5,-0.8) circle (2pt);
\filldraw [black] (8.5,-0.8) circle (2pt);
\filldraw [black] (6,-3.1) circle (2pt);
\draw  [black](6,-3.1)--(5.5,-0.8);
\draw  [black](6,-3.1)--(6.5,-0.8);
\filldraw [black] (7,-3.1) circle (2pt);
\draw  [black](7,-3.1)--(6.5,-0.8);
\draw  [black](7,-3.1)--(7.5,-0.8);
\filldraw [black] (8,-3.1) circle (2pt);
\draw  [black](8,-3.1)--(6.5,-0.8);
\draw  [black](8,-3.1)--(8.5,-0.8);
\node at (7, -0.3) {$K_{4}$};
\node at (7, -3.8) {$\Gamma'$};
\end{tikzpicture}
\caption{The structure of $\Gamma'$.}
\label{ff--1}
\end{figure}
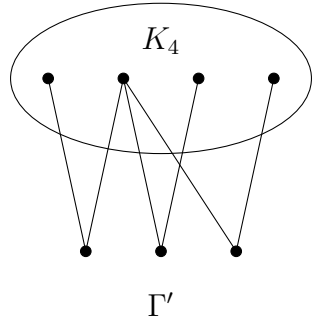

If there are exactly  two vertices of $I$ with degree $2$, then the neighborhood of them are the same as $G$ is not Hamiltonian, so  $G\cong\Gamma_{n,3}'$ with $n\ge7$.
By Lemma \ref{I3}, we have $\rho(G)<\rho(\Gamma_{n,3})$, a contradiction.
%If there is only one vertex, say $u_i$ of $I$ with degree $2$,  where $i=1,2,3$, then letting $N_G(u_i)=\{v_0,v_1\}$ and denoting by $u_j$, $u_k$ the two other vertices of $I$, we  find $v_2, v_3\in K\setminus\{v_0, v_1\}$ such that $v_2\in N_G(u_j)$ and $v_3\in N_G(u_k)$,  so $v_1u_iv_0u_jv_2u_kv_1$ if $v_3=v_2$, $v_1u_iv_0u_jv_2\dots v_3u_kv_1$ or $v_1u_iv_0u_jv_2\dots v_3u_kv_{|K|-1}v_1$ otherwise are  Hamiltonian cycle of $G$, a contradiction.

Suppose in the following that $|I|\ge 4$ (in Case 2).
By Lemma \ref{DI1}, there is a vertex of $I$ with degree $1$.
As $x_{u_1}\ge x_{u_2}\ge x_{u_3}$, we have  $d_G(u_2)=d_G(u_3)=1$ if $d_G(u_1)=1$, contradicting the fact that  $\cup_{i=1}^3N_G(u_i)=K$. So $d_G(u_1)\ge2$. Similarly, we have $d_G(u_2)\ge2$.

\begin{Claim}\label{NOTE}
For $v', v''$ in the same vertex set $A_{11}$ or $B_{ii}$ with $i=2,3$ such that $d_I(v')=d_I(v'')=3$,  we have   $N_I(v')\setminus\{u_1,u_2,u_3\}\ne N_I(v'')\setminus\{u_1,u_2,u_3\}$.
\end{Claim}
\begin{proof}
Suppose that  $v'$, $v''\in A_{11}$ (or $B_{ii}$), $d_I(v')=d_I(v'')=3$ and $N_I(v')\setminus\{u_1,u_2,u_3\}= N_I(v'')\setminus\{u_1,u_2,u_3\}$.
Let $u'\in N_I(v')\setminus\{u_1,u_2,u_3\}$ such that $x_{u'}=\min\{x_u: u\in N_I(v')\setminus\{u_1,u_2,u_3\}\}$.
Let $G'=G-\{v''u: u\in N_I(v'')\}+\{v''u_j: j=1,2,3\}$ if $d_G(u_3)\ne1$, and $G'=G-v''u'+v''u_k$ otherwise, where $k=2$ if $v',v''\in A_{11}$ ($k=1$ if $v',v''\in B_{22})$.
Then  $G'\in\mathbb{G}_n$.
In the latter case, we have by Lemma \ref{perron} in view of $x_{u_k}\ge x_{u'}$ that $\rho(G')>\rho(G)$, a contradiction.
In the former case, we have
\begin{align*}
\rho(G')-\rho(G)\ge\mathbf{x}^\top(A(G')-A(G))\mathbf{x}=2x_{v''}\left(\sum_{i=1}^3x_{u_i}-\sum_{u\in N_I(v'')}x_u\right)\ge0,
\end{align*}
so $\rho(G')=\rho(G)$ and thus  $\mathbf{x}$ is also the Perron vector of $G'$.
But by a direct calculation, $\rho(G')x_{u_2}-\rho(G)x_{u_2}=x_{v''}>0$, a contradiction.
\end{proof}

Let $t=\max\{|A_{12}|, |A_{13}|, |B_{23}|\}$. We consider  $t=0$, $t=1$ and $t\ge2$ separately.

\noindent{\bf Case 2.1.} $t=0$.

None of $B_{22}\cup B_{33}$, $A_{11}\cup B_{33}$ and $A_{11}\cup B_{22}$ is empty. We claim that $A_{11}\ne\emptyset$ and $B_{22}\ne\emptyset$.
Otherwise, we have $A_{11}=\emptyset$ or $B_{22}=\emptyset$.
Then $B_{33}\ne\emptyset$.
But then  $\rho(G)(x_{u_i}-x_{u_3})=-\sum_{v\in B_{33}}x_v<0$ for $i=1$ if $A_{11}=\emptyset$, and for $i=2$ if $B_{22}=\emptyset$, a contradiction.

Let $A_{11}=\{v_{11},\dots,v_{1|A_{11}|}\}$ and $B_{ii}=\{v_{i1},\dots,v_{i|B_{ii}|}\}$ for $i=2,3$.
Let $\Delta(B)=\max\{d_I(v): v\in B\}$. Then $1\le\Delta(B)\le 3$.

\noindent{\bf Case 2.1.1.} $\Delta(B)=1$.

We first show that  $B_{33}=\emptyset$.
Otherwise,  $G+u_1v_{31}$ is a connected $K_{1,4}$-free split graph of type $(K,I)$, and it is not Hamiltonian by Lemma \ref{DI1} and $d_G(u_1)\ge 2$.  By Lemma \ref{addedges}, we have $\rho(G+u_1v_{31})>\rho(G)$, a contradiction.
Next we show that  $|A_{11}|=|B_{22}|=1$.
Suppose that $|B_{22}|\ge2$ or $|A_{11}|\ge2$. Let $G'=G+u_1v_{21}$ in the former case and $G'=G+u_2v_{11}$ in the latter case. Then $G'\in\mathbb{G}_n$, and by Lemma \ref{addedges}, we have $\rho(G')>\rho(G)$, a contradiction.

\begin{figure}[ht]
\centering
\begin{tikzpicture}
\draw (2,-0.8) ellipse (2 and 1);
\filldraw [black] (0.5,-0.8) circle (2pt);
\filldraw [black] (1.1,-0.8) circle (2pt);
\filldraw [black] (1.6,-0.8) circle (0.8pt);
%\filldraw [black] (1.7,-0.8) circle (0.8pt);
\filldraw [black] (2.0,-0.8) circle (0.8pt);
%\filldraw [black] (2.1,-0.8) circle (0.8pt);
\filldraw [black] (2.4,-0.8) circle (0.8pt);
%\filldraw [black] (2.5,-0.8) circle (0.8pt);
\filldraw [black] (2.9,-0.8) circle (2pt);
\filldraw [black] (3.5,-0.8) circle (2pt);
\draw  [black](0.5,-3.1)--(0.5,-0.8);
\draw  [black](1.3,-3.1)--(0.5,-0.8);
\draw  [black](1.3,-3.1)--(1.1,-0.8);
\draw  [black](1.3,-3.1)--(2.9,-0.8);
\filldraw [black] (0.5,-3.1) circle (2pt);
\filldraw [black] (1.3,-3.1) circle (2pt);
\filldraw [black] (2.4,-3.1) circle (2pt);
\draw  [black](2.4,-3.1)--(3.5,-0.8);
\draw  [black](2.4,-3.1)--(1.1,-0.8);
\draw  [black](2.4,-3.1)--(2.9,-0.8);
\filldraw [black] (3.5,-3.1) circle (2pt);
\draw  [black](3.5,-3.1)--(1.1,-0.8);
\draw  [black](3.5,-3.1)--(2.9,-0.8);
\node at (2, -0.2) {$K_{n-4}$};
%\node at (2, -3.8) {$\Gamma_7$};
\end{tikzpicture}
\caption{The structure of $G$.}
\label{f5}
\end{figure}
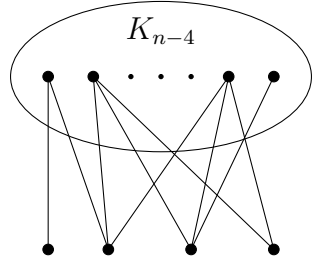

Then we claim that $d_I(v_{11})\le2$, as otherwise, to avoid $K_{1,4}$, one vertex of $N_I(v_{11})\setminus\{u_1\}$ is adjacent to $v_{21}$, contradicting the fact that $\Delta(B)=1$.
So $|I|=4$ and $G$ has the structure as displayed in Fig. \ref{f5}.
However, we can see that $G+u_3v_{21}\in \mathbb{G}_n$, and by Lemma \ref{addedges}, we have $\rho(G+u_3v_{21})>\rho(G)$, a contradiction.

\noindent{\bf Case 2.1.2.} $\Delta(B)=2$.

Assume that $v_{i1}\in B_{ii}$ with $d_I(v_{i1})=2$ with $i=2$ or $3$. Let $N_I(v_{i1})=\{u_i,u_4\}$. %, and $\{j\}=\{2,3\}\setminus\{i\}$.

We first show that $|B|=1$.
Otherwise, $G-u_iv_{i1}+u_1v_{i1}$ is a connected $K_{1,4}$-free split graph of type $(K,I)$, and it is not Hamiltonian by Lemma \ref{DI1} and $d_G(u_1)\ge2$.
By Lemma \ref{perron}, we have $\rho(G-u_iv_{i1}+u_1v_{i1})>\rho(G)$, a contradiction.
Together with $B_{22}\ne\emptyset$, we have  $|B|=|B_{22}|=1$, and $i=2$.

Next we show that $d_I(v)\le 2$ for some $v\in A_{11}$.
Suppose that $d_I(v)=3$ for any $v\in A_{11}$.
Since $G$ is $K_{1,4}$-free and $\Delta(B)=2$, we have $A_{11}\subset N_G(u_4)$.
For any $v\in A_{11}$, we have by Claim \ref{NOTE} that $d_G(u)=1$ for any $u\in N_I(v)\setminus\{u_1,u_4\}$.
%Then $|K'|\ge2$ as $x_{v_0}\ge x_{v}$ for any $v\in A_{11}$.
So $G+u_3v_{21}\in\mathbb{G}_n$. By Lemma \ref{addedges}, we have $\rho(G+u_3v_{21})>\rho(G)$, a contradiction.

Assume that $d_I(v_{11})\le2$. Then we claim $|A_{11}|=1$.  Otherwise,  $G+u_2v_{11}\in \mathbb{G}_n$, and we have by Lemma \ref{addedges} that  $\rho(G+u_2v_{11})>\rho(G)$, a contradiction.
Clearly, $|K'|\ge2$.
By Lemma \ref{DI1} and $d_I(v_{11})\le2$, $v_{11}$ is not adjacent to $u_4$.
Let $G'=G+u_3v_{11}$ if $d_I(v_{11})=1$, and $G'=G+u_4v_{11}$ otherwise.
Then   $G'\in \mathbb{G}_n$, and by Lemma \ref{addedges}, we have $\rho(G')>\rho(G)$, a contradiction.

\noindent{\bf Case 2.1.3.} $\Delta(B)=3$.

In this case, $|I|\ge5$. Assume that  $v_{i1}\in B_{ii}$ with $d_I(v_{i1})=3$ and $N_I(v_{i1})=\{u_i, u_4, u_5\}$, where $i=2$ or $3$. Assume that $x_{u_4}\ge x_{u_5}$.
By Claim \ref{adj}, $N_G(u_i)\cup N_G(u_4)\cup N_G(u_5)=K$.
So $d_G(u_4)\ge2$, as otherwise, $A_{11}\subset N_G(u_5)$, contradicting  $x_{u_4}\ge x_{u_5}$.
Let $j=\{2,3\}\setminus\{i\}$.

First, we show that $A_{11}\cup B_{jj}\subseteq N_G(u_4)$. It suffices to show that
$N_G(u_5)\subseteq N_G(u_4)$. Otherwise,  $G':=G-\{vu_5: v\in N_G(u_5)\setminus N_G(u_4)\}+\{vu_4: v\in N_G(u_5)\setminus N_G(u_4)\}\in \mathbb{G}_n$.
By Lemma \ref{perron}, we have $\rho(G')>\rho(G)$, a contradiction.

Next, we show that $d_G(u_3)\ge2$. Otherwise, $d_G(u_3)=1$. Then $i=2$, $B_{33}=\emptyset$ and $|K'|=1$.
If $|A_{11}|=1$, then $A=\{v_{11}\}$, $|B|=|B_{22}|\ge3$, and $(1+\rho(G))\rho(G)(x_{u_1}-x_{u_2})<(1+\rho(G))(x_{v_{11}}-x_{v_{21}}-x_{v_{22}})< x_{u_1}-x_{u_2}+\sum_{u\in N_I(v_{11})\setminus\{u_1,u_4\}}x_u-x_{v_{11}}<x_{u_1}-x_{u_2}$, where the last inequality follows as $\sum_{u\in N_I(v_{11})\setminus\{u_1,u_4\}}x_v<x_{v_{11}}$ which is obvious if $d_I(v_{11})=2$ and otherwise, let $u'\in N_I(v_{11})\setminus\{u_1,u_4\}$, then $\rho(G)(x_{u'}-x_{v_{11}})<x_{v_{11}}-x_{u'}$, as desired, so $x_{u_1}<x_{u_2}$, a contradiction. So $|A_{11}|\ge 2$.
If there is a vertex in $A_{11}$, say $v_{11}$, such that $d_I(v_{11})=2$, then  $G+u_2v_{11}\in \mathbb{G}_n$, and we have by Lemma \ref{addedges} that  $\rho(G+u_2v_{11})>\rho(G)$, a contradiction. So  $d_I(v)=3$ for any $v\in A_{11}$.
Let $N_I(v_{11})=\{u_1,u_4,u'\}$. By Claim \ref{adj}, we have $B_{22}\subset N_G(u_4)\cup N_G(u')$. As $A_{11}\cup B_{jj}\subseteq N_G(u_4)$, we have $\rho(G)(x_{u_4}+x_{u'})\ge \sum_{v\in A_{11}\cup B_{22}}x_v+x_{v_{11}}$.
Then
\begin{align*}
\rho(G)(1+\rho(G))(x_{v_{11}}-x_{v_0})& =\rho(G)(x_{u_4}+x_{u'}-x_{u_2}-x_{u_3})\\
& \ge\sum_{v\in A_{11}\cup B_{22}}x_v+x_{v_{11}}-\sum_{v\in B_{22}}x_v-2x_{v_0}\\
& >2(x_{v_{11}}-x_{v_0}),
\end{align*}
so $x_{v_{11}}>x_{v_0}$, a contradiction.
It follows that  $d_G(u_3)\ge2$, so $d_G(u_i)\ge 2$.

Now, we show that $|B|=1$.
Suppose that $|B|\ge2$.

Suppose that $B_{jj}\ne\emptyset$. Then for any $v\in B_{jj}$, $d_I(v)=3$. Otherwise, let $d_I(v_{j1})\le 2$. Note that $u_4\in N_I(v_{j1})$ and $N_G(u_4)\cup N_G(u_i)=K$. So $G+u_iv_{j1}=(K,I)$ that  is not Hamiltonian by Lemma \ref{DI1} as $d_G(u_i)\ge2$. Thus $G+u_iv_{j1}\in \mathbb{G}_n$. By Lemma \ref{addedges}, we have $\rho(G+u_iv_{j1})>\rho(G)$, a contradiction.
Let $N_I(v_{j1})=\{u_j, u_4, u_6\}$.
By Claim \ref{adj}, $A_{11}\cup B_{ii}\subset N_G(u_4)\cup N_G(u_6)$.
Recall that $B_{jj}\subset N_G(u_4)$.
Then $G-u_jv_{j1}+u_1v_{j1}\in \mathbb{G}_n$. By Lemma \ref{perron}, we have $\rho(G-u_jv_{j1}+u_1v_{j1})>\rho(G)$, a contradiction.
So $B_{jj}=\emptyset$.

Recall that $B_{22}\ne\emptyset$. Then $|B|=|B_{22}|\ge 2$.
Also, we have $|A|\ge 2$, as otherwise,
\[
(1+\rho(G))\rho(G)(x_{u_1}-x_{u_2})\le(1+\rho(G))(x_{v_{11}}-x_{v_{21}}-x_{v_{22}})<x_{u_1}-x_{u_2},
\]
implying that $x_{u_1}<x_{u_2}$, a contradiction.

Next we want to show that $B_{22}\subset N_G(u_4)$.
Suppose that this is not true. Then there is a vertex $v_{22}\in B_{22}$ such that $u_4\notin N_I(v_{22})$.
If $d_I(v_{22})\le2$, then $G+u_1v_{22}\in \mathbb{G}_n$, and by Lemma \ref{addedges}, we have $\rho(G+u_1v_{22})>\rho(G)$, a contradiction.
So $d_I(v_{22})=3$. Let $N_I(v_{22})=\{u_2, u', u''\}$ with $x_{u'}\le x_{u''}$.
Then $2x_{u'}\le x_{u'}+x_{u''}\le x_{u_1}+x_{u_3}\le2x_{u_1}$, i.e., $x_{u'}\le x_{u_1}$.
As $G$ is $K_{1,4}$-free and $|A|\ge 2$, from Claim \ref{NOTE}, we have $|A|=2$, and we can assume that $u'\in N_I(v_{11})$ and $u''\in N_I(v_{12})$.
Then  $G-u'v_{22}+u_1v_{22}\in \mathbb{G}_n$, and by Lemma \ref{perron}, we have $\rho(G-u'v_{22}+u_1v_{22})>\rho(G)$, a contradiction.

Note that $G-u_2v_{21}+u_1v_{21}$ is $K_{1,4}$-free as $B_{22}\subset N_G(u_4)$. So $G-u_2v_{21}+u_1v_{21}\in \mathbb{G}_n$, and by Lemma \ref{perron}, we have $\rho(G-u_2v_{21}+u_1v_{21})>\rho(G)$,  a contradiction.
So $|B|=1$, as desired.

Recall that $B_{22}\ne\emptyset$.
Then $i=2$.

We claim that $d_I(v)=3$ for any $v\in A_{11}$.
Otherwise, assume that $d_I(v_{11})=2$, then $G+u_3v_{11}\in \mathbb{G}_n$, and by Lemma \ref{addedges}, we have $\rho(G+u_3v_{11})>\rho(G)$, a contradiction.

Now we show that $u_5\notin N_I(v)$ for any $v\in A_{11}$. Otherwise,
there exists $v\in A_{11}$ such that $u_5\in N_I(v)$, say $v_{11}$. Then by Lemma \ref{DI1}, we have $|A|\ge2$. As $2x_{u_5}\le x_{u_4}+x_{u_5}\le x_{u_2}+x_{u_3}\le 2x_{u_2}$, we have $x_{u_5}\le x_{u_2}$. Note that $G-u_5v_{11}+u_2v_{11}\in \mathbb{G}_n$, and by Lemma \ref{perron}, we have $\rho(G-u_5v_{11}+u_2v_{11})>\rho(G)$, a contradiction.

By Claim \ref{NOTE}, we have $G\cong\Gamma_{n, |I|}^*$ as displayed in Fig. \ref{flem11}, which is impossible as we have by Lemma \ref{I4} that $\rho(\Gamma_{n, |I|}^*)<\rho(\Gamma_{n,|I|})$.

%Now, we claim that for any vertex $v\in A_{11}\cup B_{jj}$, $d_I(v)=3$.
%Otherwise, for some $v'\in A_{11}\cup B_{jj}$, $d_I(v')=2$.
%Let $r=j$ if $v'\in A_{11}$, and $r=1$ otherwise.
%Note that $G+v'u_r$ is a connected $K_{1,4}$-free split graph of type $(K,I)$ and is not Hamiltonian.
%By Lemma \ref{addedges}, we have $\rho(G+v_2u_r)>\rho(G)$, a contradiction.
%
%
%Finally, we claim that $|B|=1$. Otherwise, let $v''\in B\setminus\{v_1\}$ with $u_s\in N_I(v'')$, where $s=2, 3$. Then $G-v''u_s+v''u_1$ is a connected $K_{1,4}$-free split graph of type $(K,I)$ and is not Hamiltonian. By Lemma \ref{perron}, in view of  $x_{u_1}\ge x_{u_s}$, we have $\rho(G-v''u_s+v''u_1)>\rho(G)$, a contradiction.
%
%
%Note that $x_{u_2}\ge x_{u_3}$. By Lemma \ref{perron},  we have $i=2$, so $B_{jj}=\emptyset$.
%By Lemma \ref{DI1}, we have $|I|\ge6$.
%Thus, from the above argument and Claim \ref{NOTE}, we have $G\cong\Gamma_{n, |I|}^*$ as displayed in Fig. \ref{flem}, which is impossible as we have by Lemma \ref{I4} that $\rho(\Gamma_{n, |I|}^*)<\rho(\Gamma_{n,|I|})$.

\noindent{\bf Case 2.2.} $t=1$.

\noindent{\bf Case 2.2.1.} $\exists v'\in A_{12}\cup A_{13}\cup B_{23}$ such that $d_I(v')=2$.

\noindent{\bf Case 2.2.1.1.} There is $v_1\in  A_{12}$ such that $d_I(v_1)=2$.

Firstly, we show that $d_G(u_3)\ge2$.
Otherwise,  $d_G(u_3)=1$, $|K'|=1$ and $A_{13}\cup B_{23}\cup B_{33}=\emptyset$. So $A=A_{11}\cup A_{12}$ and $K\setminus N_G(u_1)=B_{22}\ne\emptyset$.
Since $\rho(G)(x_{u_1}-x_{u_2})=\sum_{v\in A_{11}}x_v-\sum_{v\in B_{22}}x_v\ge0$, we have $A_{11}\ne\emptyset$.
Since $|I|\ge4$ and  $d_G(u_3)=1$, it  is easy to see that $G+u_4v_1\in \mathbb{G}_n$.  By Lemma \ref{addedges}, we have $\rho(G+u_4v_1)>\rho(G)$, a contradiction.

Next, we show that $A_{11}\cup B_{22}=\emptyset$.
Otherwise, $|A_{11}\cup B_{22}|\ge1$, and so $G+u_3v_1=(K,I)$.
As $d_G(u_3)\ge2$, we have by  Lemma \ref{DI1} that $G+u_3v_1$ is not Hamiltonian, so $G+u_3v_1\in \mathbb{G}_n$. By Lemma \ref{addedges}, $\rho(G+u_3v_1)>\rho(G)$, a contradiction.

Now we show that $B_{23}=\emptyset$. Otherwise, suppose that $B_{23}=\{v_2\}$.
As $A_{11}\cup B_{22}=\emptyset$, we have $\rho(G)(x_{u_1}-x_{u_2})=\sum_{v\in A_{13}}x_v-x_{v_2}\ge0$, so $A_{13}\ne\emptyset$.
Then $|A_{12}|=|A_{13}|=|B_{23}|=1$. 
Suppose that $B_{33}\ne\emptyset$. As $d_I(v_1)=2$ and $G$ is $K_{1,4}$-free, $d_I(v)\le2$ for any $v\in B_{33}$.
Then it is easy to see that $G+u_1v_{31}\in \mathbb{G}_n$, and we have by Lemma \ref{addedges} that $\rho(G+u_1v_{31})>\rho(G)$, a contradiction.
So $B_{33}=\emptyset$. 
Thus, $K=K'\cup A_{12}\cup A_{13}\cup B_{23}$. 
Let $A_{13}=\{v_3\}$.
Since $\rho(G)(x_{u_2}-x_{u_3})=x_{v_1}-x_{v_3}$ and
$(\rho(G)+1)(x_{v_1}-x_{v_3})\le x_{u_2}-x_{u_3}$, we have $(\rho^2(G)+\rho(G)-1)(x_{u_2}-x_{u_3})\le0$, i.e., $x_{u_2}=x_{u_3}$ and $d_I(v_3)=2$.
Moreover, we have $d_I(v_2)=2$ as $(1+\rho(G))\rho(G)(x_{u_1}-x_{u_2})=(1+\rho(G))(x_{v_3}-x_{v_2})\le x_{u_1}-x_{u_2}$.
Thus, $d_I(v_i)=2$ for any $i=1,2,3$, implying that $|I|=3$, a contradiction.

As $B_{23}=\emptyset$, we have $K\setminus N_G(u_1)=B_{33}\ne\emptyset$. 
And $d_I(v)\le 2$ for any $v\in B_{33}$ as $G$ is $K_{1,4}$-free and $d_I(v_1)=2$.
If $|A_{13}\cup B_{33}|\ge2$, then $G+u_2v_{31}\in \mathbb{G}_n$, and by Lemma \ref{addedges}, $\rho(G+u_2v_{31})>\rho(G)$, a contradiction. So $B_{33}=\{v_{31}\}$, and $A=A_{12}=\{v_1\}$.
Thus, $G\cong\Gamma_{n,4}$ displayed in Fig. \ref{f3}.

Suppose next that $B_{23}\ne\emptyset$.

\noindent{\bf Case 2.2.1.2.} There is $v_1\in B_{23}$ such that $d_I(v_1)=2$.

Firstly, we show that  $B_{22}\cup B_{33}=\emptyset$.
Otherwise, it is easy to see that $G+u_1v_1\in\mathbb{G}_n$, and we have by Lemma \ref{addedges} that $\rho(G+u_1v_1)>\rho(G)$, a contradiction.

Next, we show that $A_{13}\ne\emptyset$.
Otherwise,  $K\setminus N_G(u_2)=A_{11}\ne\emptyset$. 
And $d_I(v)\le2$ for any $v\in A_{11}$ as $G$ is $K_{1,4}$-free and $d_I(v_1)=2$.
Suppose that $|A|\ge2$. Then $G+u_3v_{11}=(K,I)$, and it is $K_{1,4}$-free as $B_{22}=\emptyset$.
As $d_G(u_3)\ge2$, we have by Lemma \ref{DI1} that $G+u_3v_{11}\in\mathbb{G}_n$. By Lemma \ref{addedges}, $\rho(G+u_3v_{11})>\rho(G)$, a contradiction.
So $|A|=|A_{11}|=1$.
Since $|I|\ge4$, $B_{22}\cup B_{33}=\emptyset$ and $d_I(v_{11})\le2$, we have $d_I(v_{11})=2$ and $|I|=4$. Then $u_4\in N_I(v_{11})$.
Note that $\rho(G)(x_{u_1}-x_{u_2})=x_{v_{11}}-x_{v_1}\ge0$. Then
\[
\rho(G)(\rho(G)+1)(x_{v_{11}}-x_{v_1})=\rho(G)(x_{u_1}+x_{u_4}-x_{u_2}-x_{u_3})<2\left(x_{v_{11}}-x_{v_1}\right),
\]
%Then
%\[
%(\rho^2(G)+\rho(G)-2)(x_{v_{11}}-x_{v_1})<0,
%\]
%i.e., $(\rho^2(G)+\rho(G)-2)(x_{v_{11}}-x_{v_1})<0$,
%so $x_{v_{11}}-x_{v_1}<0$, implying that $x_{u_1}<x_{u_2}$,
a contradiction.

As $x_{u_2}\ge x_{u_3}$ and $B_{22}=\emptyset$, we have $A_{12}\ne\emptyset$,
so $|B_{23}|=|A_{12}|=|A_{13}|=1$.

Now we show that $A_{11}=\emptyset$.
Otherwise, we have $d_I(v)\le2$ for any $v\in A_{11}$ as $d_I(v_1)=2$ and $G$ is $K_{1,4}$-free.
By Lemma \ref{DI1} in view of $|A_{13}|=1$ and  $B_{33}=\emptyset$, it is easy to see that $G+u_2v_{11}\in \mathbb{G}_n$.  By Lemma \ref{addedges}, we have $\rho(G+u_2v_{11})>\rho(G)$, a contradiction.

Thus, $K=K'\cup A_{12}\cup A_{13}\cup B_{23}$.
Let $A_{1i}=\{v_{i}\}$ for $i=2,3$.
As $|I|\ge4$, we have by Lemma \ref{DI1} that $d_I(v_2)=3$. Otherwise, $d_I(v_3)=3$, so $(\rho(G)+1)\rho(G)(x_{u_2}-x_{u_3})=(\rho(G)+1)(x_{v_2}-x_{v_3})<x_{u_2}-x_{u_3}$, a contradiction. Thus $3=d_I(v_2)\ge d_I(v_3)\ge d_I(v_1)=2$, implying that $|I|\le5$.
If $|I|=4$, then by Lemma \ref{DI1}, we have $G\cong\Gamma_{n,4}'$ displayed in Fig. \ref{abc}, but by Lemma \ref{I44}, $\rho(G)<\rho(\Gamma_{n,4})$, a contradiction;
if $|I|=5$, then $G\cong\Gamma_{n,5}'$ displayed in Fig. \ref{f22}, but by Lemma \ref{I5}, $\rho(G)<\rho(\Gamma_{n,5})$, also a contradiction.

\noindent{\bf Case 2.2.1.3.} There is $v_1\in A_{13}$ such that $d_I(v_1)=2$.

Similarly as in Case 2.2.1.2, we have  $A_{11}\cup B_{33}=\emptyset$, $B_{23}\ne \emptyset$.
%Suppose that $A_{11}\cup B_{33}\ne\emptyset$.
%Then it is easy to see that $G+u_2v_1$ is a connected $K_{1,4}$-free split graph of type $(K,I)$, and is not Hamiltonian as $d_G(u_2)\ge2$ and Lemma \ref{DI1}, so by Lemma \ref{addedges}, we have $\rho(G+u_2v_1)>\rho(G)$, a contradiction.
%So $A_{11}\cup B_{33}=\emptyset$.
%
%Suppose that $B_{23}=\emptyset$.
%Then $K\setminus N_G(u_1)=B_{22}\ne\emptyset$.
%As $G$ is $K_{1,4}$-free and $d_I(v_1)=2$, we have $d_I(v)\le2$ for and $v\in B_{22}$.
%Suppose that $|B_{22}\cup A_{12}|\ge2$.
%Then it is easy to see that $G+u_3v_{21}$ is a connected $K_{1,4}$-free split graph of type $(K,I)$, and is not Hamiltonian as $d_G(u_3)\ge2$ and Lemma \ref{DI1},  so by Lemma \ref{addedges}, we have $\rho(G+u_3v_{21})>\rho(G)$, a contradiction.
%So $B_{22}=\{v_{21}\}$ and $A_{12}=\emptyset$.
%As $|I|\ge4$ and $d_I(v_{21})\le2$, we have $d_I(v_{21})=2$. Let $N_I(v_{21})=\{u_2, u_4\}$.
%Note that $\rho(G)(x_{u_2}-x_{u_3})=x_{v_{21}}-x_{v_1}$ and $\rho(G)(\rho(G)+1)(x_{v_{21}}-x_{v_1})=\rho(G)(x_{u_2}+x_{u_4}-x_{u_1}-x_{u_3})<\rho(G)(x_{u_2}-x_{u_3})+x_{v_{21}}-x_{v_1}$.
%Then
%\[
%(\rho^2(G)+\rho(G)-2)(x_{v_{21}}-x_{v_1})<0,
%\]
%%i.e., $(\rho^2(G)+\rho(G)-2)(x_{v_{21}}-x_{v_1})<0$,
%so $x_{v_{21}}<x_{v_1}$, implying that $x_{u_2}<x_{u_3}$, a contradiction.
%It follows that  $B_{23}\ne\emptyset$.
As $x_{u_1}\ge x_{u_3}$ and $A_{11}=\emptyset$, we have $A_{12}\ne\emptyset$, so $|A_{12}|=|B_{23}|=|A_{13}|=1$.
Similarly as above, we have  $B_{22}=\emptyset$, and
%
%Suppose that $B_{22}\ne\emptyset$. As $d_I(v_1)=2$ and $G$ is $K_{1,4}$-free, $d_I(v)\le2$ for any $v\in B_{22}$.
%Then it is easy to see that $G+u_1v_{21}$ is a connected $K_{1,4}$-free split graph of type $(K,I)$ and is not Hamiltonian by Lemma \ref{DI1}  as $d_G(u_1)\ge2$, and by Lemma \ref{addedges}, we have $\rho(G+u_1v_{21})>\rho(G)$, a contradiction.
%So $B_{22}=\emptyset$.
%
$K=K'\cup A_{12}\cup A_{13}\cup B_{23}$.
Let $A_{12}=\{v_2\}$ and $B_{23}=\{v_3\}$. As above, we have  $|I|=5$.
%
%%As $|I|\ge4$, by Lemma \ref{DI1}, we have $d_I(v_2)=3$.  Otherwise, $d_I(v_3)=3$, so $(\rho(G)+1)\rho(G)(x_{u_1}-x_{u_2})=(\rho(G)+1)(x_{v_1}-x_{v_3})<x_{u_1}-x_{u_2}$, a contradiction. Thus $3=d_I(v_2)\ge d_I(v_3)\ge d_I(v_1)=2$, implying that $|I|\le5$.
%
%If $|I|=4$, then by Lemma \ref{DI1}, we have $G\cong\Gamma_{n,4}'$ displayed in Fig. \ref{abc}, but by Lemma \ref{I44}, we have $\rho(G)<\rho(\Gamma_{n,4})$, a contradiction;
But then  $(\rho(G)+1)\rho(G)(x_{u_1}-x_{u_2})=(\rho(G)+1)(x_{v_1}-x_{v_3})<x_{u_1}-x_{u_2}$, a contradiction.

\noindent{\bf Case 2.2.2.} $\forall v\in A_{12}\cup A_{13}\cup B_{23}$, $d_I(v)=3$.

%\begin{Claim}\label{C222}
%If $|A_{1i}|\ge2$ for $i=2$ or $3$, then for any $v', v''\in A_{12}\cup A_{13}\cup B_{23}$, there is $N_I(v')\setminus\{u_1, u_2, u_3\}\ne N_I(v'')\setminus\{u_1, u_2, u_3\}$.
%\end{Claim}
%
%\begin{proof}
%Let $N_I(v')\setminus\{u_1, u_2, u_3\}= N_I(v'')\setminus\{u_1, u_2, u_3\}=\{u'\}$.
%Let $j=\{2,3\}\setminus\{i\}$ if $v'\in A_{1i}$ or $j=1$ if $v'\in B_{23}$.
%By the definition of $u_1,u_2,u_3$, we have $x_{u'}\le x_{u_j}$.
%And we have $d_G(u_j)\ge2$, as otherwise, $j=3$, $|K'|=1$ and $v', v''\in A_{12}$, then we have $(\rho(G)+1)\rho(G)(x_{v'}-x_{v_0})=(\rho(G)+1)(x_{u'}-x_{u_3})\ge 2x_{v'}-x_{v_0}$, i.e., $(\rho^(G)+\rho(G)-1)(x_{v'}-x_{v_0})=x_{v'}>0$, implying that $x_{v_1}>x_{v_0}$, a contradiction.
%Then we can check that $G-u'v'+u_jv'$ is a connected $K_{1,4}$-free split graph of type $(K,I)$, and is not Hamiltonian as $d_G(u_j)\ge2$ and Lemma \ref{DI1}, and by Lemma \ref{perron}, we have $\rho(G-u'v'+u_jv')>\rho(G)$, a contradiction.
%\end{proof}

\noindent{\bf Case 2.2.2.1.} $|A_{12}|=|A_{13}|=|B_{23}|=1$.

Let $A_{12}=\{v_1\}$, $A_{13}=\{v_2\}$ and $B_{23}=\{v_3\}$.
Let $u_4\in N_I(v_1)\setminus \{u_1,u_2\}$. 

We show that $d_G(u_4)=1$. 
Otherwise, $d_G(u_4)\ge2$. 
Then $A_{11}\cup B_{22}=\emptyset$, as otherwise $G-u_4v_1+u_3v_1=(K,I)$, and we can check that  $G-u_4v_1+u_3v_1\in \mathbb{G}_n$, so by Lemma \ref{perron} in view of $x_{u_3}\ge x_{u_4}$, we have $\rho(G-u_4v_1+u_3v_1)>\rho(G)$, a contradiction.
And $B_{33}=\emptyset$, as otherwise $B_{33}\subset N_G(u_4)$ by Claim \ref{adj}, and we can check that $G-u_4v_{31}+u_1v_{31}\in \mathbb{G}_n$, so we have by Lemma \ref{perron} that $\rho(G-u_4v_{31}+u_1v_{31})>\rho(G)$, a contradiction.
By Lemma \ref{DI1} and as $d_G(u_4)\ge2$, we have $|I|=5$ and $d_G(u_5)=1$. 
%Note that $\rho(G)(x_{u_1}-x_{u_2})=x_{v_2}-x_{v_3}\ge0$.  
And $u_4\in N_I(v_2)$, as otherwise $u_4\in N_I(v_3)$ and $u_5\in N_I(v_2)$, then $\rho(G)(1+\rho(G))(x_{v_2}-x_{v_3})=\rho(G)(x_{u_1}+x_{u_5}-x_{u_2}-x_{u_4})
=2(x_{v_2}-x_{v_3})-x_{v_1}<2(x_{v_2}-x_{v_3})$, a contradiction.
So $G\cong\Gamma_{n,5}''$ (displayed in Fig. \ref{f22}).
By Lemma \ref{I5}, $\rho(G)<\rho(\Gamma_{n,5})$, a contradiction.
So $d_G(u_4)=1$.

Let $u_5\in N_I(v_2)\setminus \{u_1,u_3\}$. Similarly as above, we have $d_G(u_5)=1$.

Let $u_6\in N_I(v_3)\setminus \{u_2,u_3\}$.
Suppose that $d_G(u_6)\ge2$. Similarly as above, we have $B_{22}\cup B_{33}=\emptyset$. Suppose that $A_{11}\ne\emptyset$. Then $A_{11}\subset N_G(u_6)$ by Claim \ref{adj}, and $d_I(v)=3$ for any $v\in A_{11}$, as otherwise let $d_I(v_{11})=2$, and we can check that $G+u_2v_{11}\in \mathbb{G}_n$, so we have by Lemma \ref{addedges} that $\rho(G+u_2v_{11})>\rho(G)$, a contradiction.
Let $G'=G-u_6v_{11}+u_2v_{11}$ if $x_{u_2}\ge x_{u_6}$, and $G'=G-u_2v_1+u_6v_1$ otherwise.
Then $G'\in\mathbb{G}_n$, but by Lemma \ref{perron}, $\rho(G)<\rho(G')$, a contradiction.
So $d_G(u_6)=1$.
%Similarly, the degree of the vertex in $N_G(v_i)\setminus\{u_1,u_2,u_3\}$ is equal to $1$ for each $i=2,3$, and

As $G$ is $K_{1,4}$-free, we have $A_{11}=B_{22}=B_{33}=\emptyset$. So $G\cong\Gamma_{n,6}$.

\noindent{\bf Case 2.2.2.2.} At least one of $A_{12}, A_{13}, B_{23}$ is empty.

\begin{Claim} \label{A12}
$A_{12}\ne\emptyset$. 
\end{Claim}

\begin{proof}
Suppose that $A_{12}=\emptyset$.

We first show that $A_{13}=\emptyset$.
Suppose that $A_{13}=\{v'\}$. Let $u'\in N_G(v')\setminus\{u_1,u_3\}$.
Then $K\setminus N_G(u_3)=A_{11}\cup B_{22}\ne\emptyset$.
Suppose that $B_{22}\ne\emptyset$. Then $B_{22}\subset N_G(u')$ by Claim \ref{adj}.
If $|B|=|B_{22}|=1$, then $\rho(x_{u_2}-x_{u_3})=x_{v_{21}}-x_{v'}\ge0$, and $\rho(G)(\rho(G)+1)(x_{v_{21}}-x_{v'})=\rho(G)\left(x_{u_2}+\sum_{u\in N_I(v_{21})\setminus\{u_2,u'\}}x_u-x_{u_1}-x_{u_3}\right)<2(x_{v_{21}}-x_{v'})$, a contradiction.
So $|B|\ge2$.
Let $G'=G-u'v_{21}+u_1v_{21}$ if $B_{33}=\emptyset$, and $G'=G-u'v'+u_2v'$ otherwise.
Then we can check that $G'\in\mathbb{G}_n$.
By Lemma \ref{perron} in view of $x_{u_1}\ge x_{u_2}\ge x_{u'}$, we have $\rho(G')>\rho(G)$, a contradiction. So $B_{22}=\emptyset$.
Then we have $\rho(G)(x_{u_2}-x_{u_3})\le -x_{v'}<0$, a contradiction.
%So $A_{13}=\emptyset$.

As $A_{12}=A_{13}=\emptyset$ and $t=1$, we have $|B_{23}|=1$.
Let $B_{23}=\{v''\}$, and $u''\in N_I(v'')\setminus \{u_2,u_3\}$.
As $x_{u_1}\ge x_{u_2}$, $A_{12}\cup A_{13}=\emptyset$ and $|B_{23}|=1$, we have $A_{11}\ne\emptyset$. Then $A_{11}\subset N_G(u'')$ by Claim \ref{adj}. 

Next we show that $B_{22}\cup B_{33}=\emptyset$ and $|A_{11}|\ge2$.
If $B_{22}\cup B_{33}\ne\emptyset$, then $G-u''v''+u_1v''=(K,I)$, and we can check that $G-u''v''+u_1v''\in\mathbb{G}_n$, so we have by Lemma \ref{perron} in view of $x_{u_1}\ge x_{u''}$ that $\rho(G-u''v''+u_1v'')>\rho(G)$, a contradiction. 
%So $B_{22}\cup B_{33}=\emptyset$. 
If $|A_{11}|= 1$, then $A_{11}=\{v_{11}\}$, and $|K'|=|K|-2\ge|I|-2\ge2$. By Lemma \ref{DI1}, we have $d_I(v_{11})=3$, and let $u_5\in N_I(v_{11})\setminus\{u_1,u''\}$, then $d_G(u_5)=1$.
%Note that $\rho(G)(x_{u_1}-x_{u_2})=x_{v_{11}}-x_{v''}\ge0$. Then we have %$\rho(G)(x_{u_5}-x_{u_3})<x_{v_{11}}-x_{v''}$ and
So $\rho(G)(\rho(G)+1)(x_{v_{11}}-x_{v''})=\rho(G)(x_{u_1}+x_{u_5}-x_{u_2}-x_{u_3})<2(x_{v_{11}}-x_{v''})$,
%\[
%(\rho^2(G)+\rho(G)-2)(x_{v_{11}}-x_{v''})<0,
%\]
%so $x_{v_{11}}<x_{v''}$, implying that $x_{u_1}<x_{u_2}$,
a contradiction.

Since $\rho(G)(\rho(G)+1)(x_{v_0}-x_{v''})=\rho(x_{u_1}-x_{u''})>x_{v_0}-x_{v''}$, we have $|K'|\ge 2$.
Let $G'=G-u_2v_0+u''v_0$ if $x_{u_2}\le x_{u''}$, and $G'=G-u''v_{11}+u_2v_{11}$ otherwise. Then we can check that $G'\in\mathbb{G}_n$ as $|K'|\ge 2$, $K=K'\cup A_{11}\cup B_{23}$, $|A_{11}|\ge2$ and by Lemma \ref{DI1}.
By Lemma \ref{perron}, $\rho(G')>\rho(G)$, a contradiction.
%Thus, $A_{12}\ne\emptyset$.
\end{proof}

From Claim \ref{A12}, let $A_{12}=\{v_1\}$ and $u_4\in N_G(v_1)\setminus\{u_1,u_2\}$.

Suppose first that $A_{13}=\emptyset$. 

%We claim that if $B_{33}\ne\emptyset$, then $G\cong\Gamma_{n,5}$.

We claim that $B_{33}\ne\emptyset$.
Suppose that $B_{33}=\emptyset$. Then $K\setminus N_G(u_1)=B_{22}\cup B_{23}\ne\emptyset$ and $K\setminus N_G(u_2)=A_{11}\ne\emptyset$.
Suppose that $B_{23}\ne\emptyset$. Let $B_{23}=\{v'\}$, and $u'\in N_G(v')\setminus\{u_2,u_3\}$.
Then $A_{11}\subset N_G(u')$ by Claim \ref{adj}.
If $B_{22}\ne \emptyset$, then $G-u'v'+u_1v'\in\mathbb{G}_n$, and we have by Lemma \ref{perron} that  $\rho(G-u'v'+u_1v')>\rho(G)$, a contradiction.
Thus $B_{22}= \emptyset$.
Let $G'=G-u'v_{11}+u_3v_{11}$ if $x_{u_3}\ge x_{u'}$, and $G'=G-u_3v_0+u'v_0$ otherwise.
Then $G'\in \mathbb{G}_n$, and by Lemma \ref{perron}, we have $\rho(G')>\rho(G)$, a contradiction.
Thus, $B_{23}=\emptyset$, so $K=K'\cup A_{11}\cup A_{12}\cup B_{22}$.

Suppose that $|A_{11}|\ge2$. Then $d_I(v)=3$ for any $v\in A_{11}$, as otherwise let $d_I(v_{11})\le2$, then $G+u_2v_{11}\in\mathbb{G}_n$, and by Lemma \ref{addedges}, $\rho(G+u_2v_{11})>\rho(G)$, a contradiction.
Let $u^*\in N_I(v_{11})\setminus\{u_1\}$ such that $x_{u^*}=\max\{x_u: u\in N_I(v_{11})\setminus\{u_1\}\}$. Together with Claim \ref{adj}, we have $B_{22}\subset N_G(u^*)$.
Let $G''=G-u^*v_{11}+u_2v_{11}$ if $x_{u_2}\ge x_{u^*}$, and $G''=G-u_2v_0+u^*v_0$ otherwise.
Then $G''\in \mathbb{G}_n$, and by Lemma \ref{perron}, we have $\rho(G'')>\rho(G)$, a contradiction.

Suppose that $|B_{22}|\ge2$. Then $d_I(v)=3$ for any $v\in B_{22}$, as otherwise let $d_I(v_{21})\le2$, then $G+u_1v_{21}\in \mathbb{G}_n$, and by Lemma \ref{addedges}, $\rho(G+u_1v_{21})>\rho(G)$, a contradiction.
Note that $G-u_2v_{21}+u_1v_{21}\in \mathbb{G}_n$, and by Lemma \ref{perron}, we have $\rho(G-u_2v_{21}+u_1v_{21})>\rho(G)$, a contradiction.

So $A_{12}=\{v_1\}$, $A_{11}=\{v_{11}\}$ and $B_{22}=\{v_{21}\}$, 
%By the above argument, $K=K'\cup A_{11}\cup A_{12}\cup B_{22}$ with $|A_{11}|=|A_{12}|=|B_{22}|=1$,
implying that $4\le|I|\le7$.
We claim that $u_4\notin N_I(v_{11})\cup N_I(v_{21})$. Otherwise, $G-u_4v_1+u_3v_1\in\mathbb{G}_n$, and by Lemma \ref{perron}, $\rho(G-u_4v_1+u_3v_1)>\rho(G)$, a contradiction.
If $|I|=4$, then $d_G(v_{11})=d_G(v_{21})=1$ by Lemma \ref{DI1}. And if $5\le|I|\le7$, then $G\cong\Gamma$ as displayed in Fig. \ref{ff8}.
For $|I|=4,5$, we can check that $G+u_3v_{11}\in\mathbb{G}_n$, and by Lemma \ref{addedges}, $\rho(G+u_3v_{11})>\rho(G)$, a contradiction. 
So $|I|=6,7$. 
Note that $\rho(G)(x_{u_3}-x_{u_5})\ge 2x_{v_0}-x_{v_{11}}-x_{v_{21}}\ge 0$ by the choice of $v_0$. So $x_{u_3}\ge x_{u_5}$. Then we can check that $G-u_5v_{11}+u_3v_{11}\in \mathbb{G}_n$, and by Lemma \ref{perron}, $\rho(G-u_5v_{11}+u_3v_{11})>\rho(G)$, a contradiction.

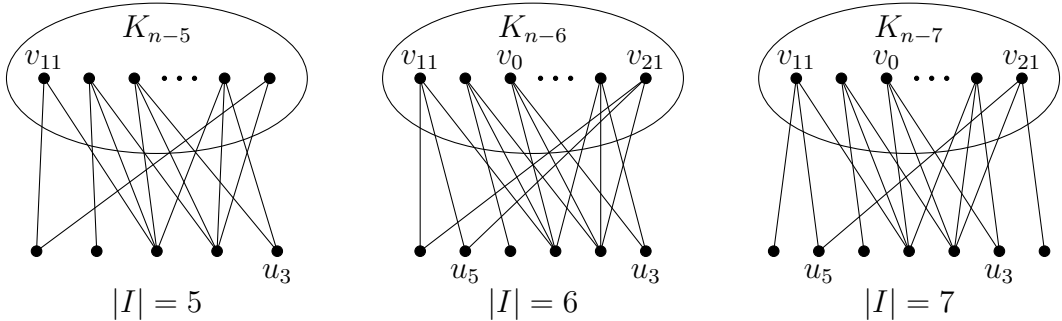
\begin{figure}[ht]
\centering
\begin{tikzpicture}
\draw (2,-0.8) ellipse (2 and 1);
\filldraw [black] (0.5,-0.8) circle (2pt);
\filldraw [black] (1.1,-0.8) circle (2pt);
\filldraw [black] (1.7,-0.8) circle (2pt);
\filldraw [black] (2.1,-0.8) circle (0.8pt);
\filldraw [black] (2.3,-0.8) circle (0.8pt);
\filldraw [black] (2.5,-0.8) circle (0.8pt);
\filldraw [black] (2.9,-0.8) circle (2pt);
\filldraw [black] (3.5,-0.8) circle (2pt);
\filldraw [black] (0.4,-3.1) circle (2pt);
\draw  [black](0.4,-3.1)--(0.5,-0.8);
\draw  [black](0.4,-3.1)--(3.5,-0.8);
\filldraw [black] (1.2,-3.1) circle (2pt);
\draw  [black](1.2,-3.1)--(1.1,-0.8);
\filldraw [black] (2.0,-3.1) circle (2pt);
\draw  [black](2.0,-3.1)--(0.5,-0.8);
\draw  [black](2.0,-3.1)--(1.1,-0.8);
\draw  [black](2.0,-3.1)--(1.7,-0.8);
\draw  [black](2.0,-3.1)--(2.9,-0.8);
\filldraw [black] (2.8,-3.1) circle (2pt);
\draw  [black](2.8,-3.1)--(1.1,-0.8);
\draw  [black](2.8,-3.1)--(1.7,-0.8);
\draw  [black](2.8,-3.1)--(2.9,-0.8);
\draw  [black](2.8,-3.1)--(3.5,-0.8);
\filldraw [black] (3.6,-3.1) circle (2pt);
\draw  [black](3.6,-3.1)--(1.7,-0.8);
\draw  [black](3.6,-3.1)--(2.9,-0.8);
\node at (2, -0.15) {$K_{n-5}$};
%\node at (0.4, -3.4) {$u_5$};
%\node at (1.2, -3.4) {$u_4$};
%\node at (2.0, -3.4) {$u_1$};
%\node at (2.8, -3.4) {$u_2$};
\node at (3.6, -3.4) {$u_3$};
\node at (0.5, -0.55) {$v_{11}$};
%\node at (1.1, -0.55) {$v_1$};
%\node at (1.7, -0.55) {$v_0$};
%\node at (3.5, -0.55) {$v_{21}$};
\node at (2, -3.8) {$|I|=5$};

\draw (7,-0.8) ellipse (2 and 1);
\filldraw [black] (5.5,-0.8) circle (2pt);
\filldraw [black] (6.1,-0.8) circle (2pt);
\filldraw [black] (6.7,-0.8) circle (2pt);
\filldraw [black] (7.1,-0.8) circle (0.8pt);
\filldraw [black] (7.3,-0.8) circle (0.8pt);
\filldraw [black] (7.5,-0.8) circle (0.8pt);
\filldraw [black] (7.9,-0.8) circle (2pt);
\filldraw [black] (8.5,-0.8) circle (2pt);
\filldraw [black] (5.5,-3.1) circle (2pt);
\draw  [black](5.5,-3.1)--(5.5,-0.8);
\draw  [black](5.5,-3.1)--(8.5,-0.8);
\filldraw [black] (6.1,-3.1) circle (2pt);
\draw  [black](6.1,-3.1)--(5.5,-0.8);
\draw  [black](6.1,-3.1)--(8.5,-0.8);
\filldraw [black] (6.7,-3.1) circle (2pt);
\draw  [black](6.7,-3.1)--(6.1,-0.8);
\filldraw [black] (7.3,-3.1) circle (2pt);
\draw  [black](7.3,-3.1)--(5.5,-0.8);
\draw  [black](7.3,-3.1)--(6.1,-0.8);
\draw  [black](7.3,-3.1)--(6.7,-0.8);
\draw  [black](7.3,-3.1)--(7.9,-0.8);
\filldraw [black] (7.9,-3.1) circle (2pt);
\draw  [black](7.9,-3.1)--(6.1,-0.8);
\draw  [black](7.9,-3.1)--(6.7,-0.8);
\draw  [black](7.9,-3.1)--(7.9,-0.8);
\draw  [black](7.9,-3.1)--(8.5,-0.8);
\filldraw [black] (8.5,-3.1) circle (2pt);
\draw  [black](8.5,-3.1)--(6.7,-0.8);
\draw  [black](8.5,-3.1)--(7.9,-0.8);
\node at (7, -0.15) {$K_{n-6}$};
%\node at (5.5, -3.4) {$u_6$};
\node at (6.1, -3.4) {$u_5$};
%\node at (6.7, -3.4) {$u_4$};
%\node at (7.3, -3.4) {$u_1$};
%\node at (7.9, -3.4) {$u_2$};
\node at (8.5, -3.4) {$u_3$};
\node at (5.5, -0.55) {$v_{11}$};
%\node at (6.1, -0.55) {$v_1$};
\node at (6.7, -0.55) {$v_0$};
\node at (8.5, -0.55) {$v_{21}$};
\node at (7, -3.8) {$|I|=6$};

\draw (12,-0.8) ellipse (2 and 1);
\filldraw [black] (10.5,-0.8) circle (2pt);
\filldraw [black] (11.1,-0.8) circle (2pt);
\filldraw [black] (11.7,-0.8) circle (2pt);
\filldraw [black] (12.1,-0.8) circle (0.8pt);
\filldraw [black] (12.3,-0.8) circle (0.8pt);
\filldraw [black] (12.5,-0.8) circle (0.8pt);
\filldraw [black] (12.9,-0.8) circle (2pt);
\filldraw [black] (13.5,-0.8) circle (2pt);
\filldraw [black] (10.2,-3.1) circle (2pt);
\draw  [black](10.2,-3.1)--(10.5,-0.8);
\filldraw [black] (10.8,-3.1) circle (2pt);
\draw  [black](10.8,-3.1)--(10.5,-0.8);
\draw  [black](10.8,-3.1)--(13.5,-0.8);
\filldraw [black] (11.4,-3.1) circle (2pt);
\draw  [black](11.4,-3.1)--(11.1,-0.8);
\filldraw [black] (12.0,-3.1) circle (2pt);
\draw  [black](12,-3.1)--(10.5,-0.8);
\draw  [black](12,-3.1)--(11.1,-0.8);
\draw  [black](12,-3.1)--(11.7,-0.8);
\draw  [black](12,-3.1)--(12.9,-0.8);
\filldraw [black] (12.6,-3.1) circle (2pt);
\draw  [black](12.6,-3.1)--(11.1,-0.8);
\draw  [black](12.6,-3.1)--(11.7,-0.8);
\draw  [black](12.6,-3.1)--(12.9,-0.8);
\draw  [black](12.6,-3.1)--(13.5,-0.8);
\filldraw [black] (13.2,-3.1) circle (2pt);
\draw  [black](13.2,-3.1)--(11.7,-0.8);
\draw  [black](13.2,-3.1)--(12.9,-0.8);
\filldraw [black] (13.8,-3.1) circle (2pt);
\draw  [black](13.8,-3.1)--(13.5,-0.8);
\node at (12, -0.15) {$K_{n-7}$};
\node at (10.8, -3.4) {$u_5$};
\node at (13.2, -3.4) {$u_3$};
\node at (10.5, -0.55) {$v_{11}$};
\node at (11.7, -0.55) {$v_0$};
\node at (13.5, -0.55) {$v_{21}$};
\node at (12, -3.8) {$|I|=7$};
\end{tikzpicture}
\caption{The structure of $\Gamma$.}
\label{ff8}
\end{figure}

So $B_{33}\ne\emptyset$. Then $B_{33}\subset N_G(u_4)$ by Claim \ref{adj}.
Then $A_{11}\cup B_{22}=\emptyset$, as otherwise $G-u_4v_1+u_3v_1\in \mathbb{G}_n$, and by Lemma \ref{perron} in view of $x_{u_3}\ge x_{u_4}$, $\rho(G-u_4v_1+u_3v_1)>\rho(G)$, a contradiction.
Note that  $|B|=1$, as otherwise $|B|\ge 2$, then $G-u_4v_{31}+u_1v_{31}\in \mathbb{G}_n$, and by Lemma \ref{perron}, $\rho(G-u_4v_{31}+u_1v_{31})>\rho(G)$, a contradiction.
By Lemma \ref{DI1}, we have $G\cong\Gamma_{n,5}$.

Suppose next that  $A_{13}\ne\emptyset$.
Let $A_{13}=\{v_2\}$, and $u_j\in N_G(v_2)\setminus\{u_1,u_2\}$, where $j\ge4$.
Then $B_{23}=\emptyset$.

Suppose that $B_{22}\ne\emptyset$.
Then we claim that $j\ge5$. Otherwise $G-u_4v_1+u_3v_1\in \mathbb{G}_n$, and by Lemma \ref{perron}, $\rho(G-u_4v_1+u_3v_1)>\rho(G)$, a contradiction.
Assume that $j=5$.
Then $B_{22}\subset N_G(u_5)$ by Claim \ref{adj}.
If $A_{11}\cup B_{33}\ne\emptyset$, then $G-u_5v_2+u_2v_2\in \mathbb{G}_n$, and by Lemma \ref{perron}, we have $\rho(G-u_5v_2+u_2v_2)>\rho(G)$, a contradiction.
So $A_{11}\cup B_{33}=\emptyset$.
And we claim that $d_I(v)=3$ for any $v\in B_{22}$, as otherwise let $d_I(v_{21})\le2$, then $G+u_3v_{21}\in \mathbb{G}_n$, and we have by Lemma \ref{addedges} that $\rho(G+u_3v_{21})>\rho(G)$, a contradiction.
If $|B_{22}|\ge2$, then $G-u_5v_{21}+u_1v_{21}\in \mathbb{G}_n$, and by Lemma \ref{perron} in view of $x_{u_5}\le x_{u_2}\le x_{u_1}$, we have $\rho(G-u_5v_{21}+u_1v_{21})>\rho(G)$, a contradiction.
So $|B_{22}|=1$.
By Lemma \ref{DI1}, $u_4\notin N_I(v_{21})$.
As $\rho(G)(x_{u_3}-x_{u_5})=(n-9)x_{v_0}-x_{v_{21}}>0$, we have $x_{u_3}>x_{u_5}$.
Then $G-u_5v_{21}+u_3v_{21}\in \mathbb{G}_n$, and by Lemma \ref{perron}, we have $\rho(G-u_5v_{21}+u_3v_{21})>\rho(G)$, a contradiction.
Thus $B_{22}=\emptyset$, and $K\setminus N_G(u_1)=B_{33}\ne\emptyset$.

By Claim \ref{adj}, we have $B_{33}\subset N_G(u_4)$.
Then $A_{11}=\emptyset$. Otherwise, $G-u_4v_1+u_3v_1\in \mathbb{G}_n$, and we have by Lemma \ref{DI1} that $\rho(G-u_4v_1+u_3v_1)>\rho(G)$, a contradiction. And we have $B_{33}=\{v_{31}\}$, as otherwise, $G-u_4v_{31}+u_2v_{31}\in \mathbb{G}_n$, and By Lemma \ref{perron} in view of $x_{u_4}\le x_{u_3}\le x_{u_2}$, we have $\rho(G-u_4v_{31}+u_2v_{31})>\rho(G)$, a contradiction.
Now, we have $(1+\rho(G))\rho(G)(x_{u_2}-x_{u_3})=(1+\rho(G))(x_{v_1}-x_{v_2}-x_{v_{31}})
<x_{u_2}+x_{u_4}-x_{u_3}-x_{v_1}<x_{u_2}-x_{u_3}$, a contradiction.

\noindent{\bf Case 2.3.} $t\ge2$.

\begin{Claim}\label{AABB}
If $|A_{1i}|\ge2$ for $i=2,3$ $($or $|B_{23}|\ge2$$)$, then the following items are true:

(i) $d_I(v)=3$ for any $v\in A_{1i}$ $($or $B_{23}$$)$, 

(ii) $N_I(v')\setminus\{u_1, u_2, u_3\}\ne N_I(v'')\setminus\{u_1, u_2, u_3\}$ for any $v'\in A_{1i}$ $($or $v'\in B_{23}$$)$ and $v''\in A_{12}\cup A_{13}\cup B_{23}$, 

(iii) $B_{jj}=\emptyset$ $($or $A_{11}=\emptyset$$)$, where $j=\{2,3\}\setminus \{i\}$.
\end{Claim}

\begin{proof}
Suppose that $|A_{1i}|\ge2$ (or $|B_{23}|\ge2$). 

Suppose  that (i) is not true. Then  $d_I(v^*)=2$ for some $v^*\in A_{1i}$ (or $B_{23}$). If $i=2$, then let  $G'=G+u_3v^*$ if $d_G(u_3)\ge2$, and $G'=G+u_4v^*$ otherwise. Otherwise, let $G'=G+u_jv^*$ (or $G'=G+u_1v^*$). In either case, we can check that $G'\in \mathbb{G}_n$, and by Lemma \ref{addedges}, we have $\rho(G')>\rho(G)$, a contradiction.

Suppose  that (ii) is not true.
%The second (ii) is trivial if $d_I(v'')=2$.
Then  $d_I(v'')=3$ and $N_I(v')\setminus\{u_1, u_2, u_3\}= N_I(v'')\setminus\{u_1, u_2, u_3\}=\{u'\}$. By the choice of $v_0$, we have $x_{u'}\le x_{u_j}$ (or $x_{u'}\le x_{u_1}$). And we have $d_G(u_j)\ge2$, as otherwise, $j=3$, $|K'|=1$ and $v', v''\in A_{12}$, then we have $(\rho(G)+1)\rho(G)(x_{v'}-x_{v_0})=(\rho(G)+1)(x_{u'}-x_{u_3})>x_{v'}-x_{v_0}$, implying that $x_{v'}>x_{v_0}$, a contradiction. Then we can check that $G-u'v'+u_jv'\in\mathbb{G}_n$ (or $G-u'v'+u_1v'\in\mathbb{G}_n$), and by Lemma \ref{perron}, we have $\rho(G-u'v'+u_jv')>\rho(G)$ (or $\rho(G-u'v'+u_1v')>\rho(G)$), a contradiction.

Now we prove (iii). Suppose first that $|A_{1i}|\ge2$
and $B_{jj}\ne\emptyset$.

By Claim \ref{adj}, we have  $B_{jj}\subset N_G(N_I(v)\setminus\{u_1,u_i\})$, where $v\in A_{1i}$. As $G$ is $K_{1,4}$-free, we have $|A_{1i}|=2$ by (ii). Let $v_1\in A_{1i}$, and $u_4\in N_I(v_1)\setminus\{u_1,u_i\}$. It can be checked that $G-u_4v_1+u_jv_1\in \mathbb{G}_n$. By Lemma \ref{perron} in view of $x_{u_j}\ge x_{u_4}$ that $\rho(G-u_4v_1+u_jv_1)>\rho(G)$, a contradiction.   Similarly,  $|B_{23}|\ge2$ and $A_{11}\ne\emptyset$ lead to a contradiction.
%Suppose that $|A_{1i}|\ge2$. Then for any vertex $v\in A_{1i}$, $d_I(v)=3$, as otherwise $G+vu_j$ is a connected $K_{1,4}$-free split graph of type $(K,I)$, and by Lemma \ref{DI1}, it is not Hamiltonian as $d_G(u_j)\ge2$. By Lemma \ref{addedges}, we have $\rho(G+vu_j)>\rho(G)$, a contradiction.
%Suppose that $B_{jj}\ne\emptyset$, $j=\{2,3\}\setminus \{i\}$ and $i=2,3$.
%As $G$ is $K_{1,4}$-free, $B_{jj}\subset N_G(N_I(v)\setminus\{u_1,u_i\})$, where $v\in A_{1i}$.
%Then by Claim \ref{C222} and $G$ is $K_{1,4}$-free, we have $|A_{1i}|=2$.
%Let $v_1\in A_{1i}$ with $N_I(v_1)=\{u_1,u_i,u_4\}$.
%It can be checked that $G-v_1u_4+v_1u_j$ is a connected $K_{1,4}$-free split graph of type $(K,I)$ and is not Hamiltonian, and by Lemma \ref{perron} and $x_{u_j}\ge x_{u_4}$, we have $\rho(G-v_1u_4+v_1u_j)>\rho(G)$, a contradiction.
%If  $|B_{23}|\ge2$, similar argument leads to a contradiction.
\end{proof}

%Suppose that $\min\{|A_{12}|, |A_{13}|, |B_{23}|\}\ge2$.
%By Claim \ref{AABB}, we can obtain that $A_{11}\cup B_{22}\cup B_{33}=\emptyset$.
%Let $v_1\in A_{13}$.
%Then $G-u_3v_1+u_2v_1$  is not Hamiltonian by Lemma \ref{DI1} and $d_G(u_2)\ge 2$, so it is in $\mathbb{G}_n$.
%By Lemma \ref{perron} in view of $x_{u_2}\ge x_{u_3}$, we have $\rho(G-u_3v_1+u_2v_1)>\rho(G)$, a contradiction.
%So $\min\{|A_{12}|, |A_{13}|, |B_{23}|\}\le1$.

\begin{Claim} \label{C2.3.1.} Two of $\{|A_{12}|, |A_{13}|, |B_{23}|\}$ are at most $1$.
\end{Claim}

\begin{proof} Suppose that this is not true.

Suppose that all of $\{|A_{12}|, |A_{13}|, |B_{23}|\}$ are larger than $1$.
By Claim \ref{AABB}, we can obtain that $A_{11}\cup B_{22}\cup B_{33}=\emptyset$.
Let $v_1\in A_{13}$.
Then $G-u_3v_1+u_2v_1\in \mathbb{G}_n$, and by Lemma \ref{perron} in view of $x_{u_2}\ge x_{u_3}$, we have $\rho(G-u_3v_1+u_2v_1)>\rho(G)$, a contradiction.
So at least one of them is at most $1$.

Exactly one of $|A_{12}|, |A_{13}|, |B_{23}|$ is at most $1$.
Suppose that $|A_{12}|\ge 2$. Then $B_{33}=\emptyset$ by Claim \ref{AABB}. 
Fix $v_1\in A_{13}$ if $|A_{13}|\ge2$, or $v_2\in B_{23}$ if $|B_{23}|\ge2$.
Let $i=1$ if $|A_{13}|\ge2$ and $i=2$ otherwise, and $j=\{1,2\}\setminus\{i\}$.
Then $G-u_3v_i+u_jv_i\in \mathbb{G}_n$ as $B_{33}=\emptyset$, but by Lemma \ref{perron} in view of $x_{u_j}\ge x_{u_3}$, we have $\rho(G-u_3v_i+u_jv_i)>\rho(G)$, a contradiction.
It follows that $|A_{12}|\le1$, so $|A_{13}|\ge2$ and $|B_{23}|\ge2$.
By Claim \ref{AABB}, we have $A_{11}=B_{22}=\emptyset$.
Let $v_3\in B_{23}$.
Then  $G-u_2v_3+u_1v_3\in \mathbb{G}_n$, but by Lemma \ref{perron} in view of $x_{u_1}\ge x_{u_2}$, we have $\rho(G-u_2v_3+u_1v_3)>\rho(G)$, a contradiction.
\end{proof}

%By Claim \ref{C2.3.1.}, two of $\{|A_{12}|, |A_{13}|, |B_{23}|\}$ are less than or equal to $1$.

If $|A_{13}|\ge2$, then $B_{22}=\emptyset$ by Claim \ref{AABB}, and we have  $(1+\rho(G))\rho(G)(x_{u_2}-x_{u_3})=(1+\rho(G))(\sum_{v\in A_{12}}x_v-\sum_{v\in A_{13}\cup B_{33}}x_v)\le (1+\rho(G))(x_{v_0}-\sum_{v\in A_{13}\cup B_{33}}x_v)< x_{u_2}-x_{u_1}-x_{u_3}<0$, a contradiction. If $|B_{23}|\ge2$, then $A_{11}=\emptyset$ by Claim \ref{AABB}, and by  $(\rho(G)+1)\rho(G)(x_{u_1}-x_{u_2})=(\rho(G)+1)(\sum_{v\in A_{13}}x_v-\sum_{v\in B_{23}\cup B_{22}}x_v)<x_{u_1}-x_{u_2}$, also a contradiction. As $t\ge2$, we have $|A_{12}|=t\ge2$.

By Claim \ref{AABB}, we have $B_{33}=\emptyset$ and $d_I(v)=3$ for any $v\in A_{12}$.

\begin{Claim} \label{Thur} $A_{11}=\emptyset$ if $|A_{13}|=1$, and $B_{22}=\emptyset$ if $|B_{23}|=1$.
\end{Claim}

\begin{proof}
Suppose first that $|A_{13}|=1$ and $A_{11}\ne\emptyset$. Then  $d_I(v)=3$ for any $v\in A_{11}$. Otherwise, let $v'\in A_{11}$ such that $d_I(v')\le2$, and we can check that $G+u_2v'\in \mathbb{G}_n$, and by Lemma \ref{addedges}, $\rho(G+u_2v')>\rho(G)$, a contradiction.
Let $N_I(v_{11})\setminus\{u_1\}=\{u', u''\}$ with $x_{u'}\ge x_{u''}$.
As $x_{u'}\ge x_{u''}$ and by Claim \ref{adj}, we have $\emptyset\ne B\subset N_G(u')$. 
By Claim \ref{AABB} and $|A_{12}|\ge2$, $N_G(u')\ne K\setminus\{v_0\}$.
%Clearly, $x_{u_2}\ge x_{u''}$.
%Let $G'=G-u''v_{11}+u_2v_{11}$ if $d_G(u'')\ge2$, 
Let $G'=G-u'v_{11}+u_2v_{11}$ if $x_{u_2}\ge x_{u'}$, and $G'=G-u_2v_0+u'v_0$ otherwise, and we can see that $G'\in \mathbb{G}_n$, and by Lemma \ref{perron}, we have $\rho(G')>\rho(G)$, a contradiction.
So $A_{11}=\emptyset$.

Suppose next that $|B_{23}|=1$ and $B_{22}\ne\emptyset$. Similarly as above,
we have $d_I(v)=3$ for any $v\in B_{22}$.
%
% (Similar???) Then we have $d_I(v)=3$ for any $v\in B_{22}$. Otherwise, let $v'\in B_{22}$ such that $d_I(v')\le2$, and we can check that $G+u_1v'$ is a connected $K_{1,4}$-free split graph of type $(K,I)$ and is not Hamiltonian, and by Lemma \ref{addedges}, $\rho(G+u_1v')>\rho(G)$, a contradiction.
Let $u', u''\in N_I(v_{21})\setminus\{u_2\}$ with $x_{u'}\ge x_{u''}$.
By Claim \ref{adj}, we have $\emptyset\ne A_{11}\cup A_{13}\subset N_G(u')$. %Clearly, $x_{u_1}\ge x_{u''}$. 
Similarly as above, $B_{22}=\emptyset$.
\end{proof}

\noindent{\bf Case 2.3.1.} $|A_{13}|=|B_{23}|=1$.

By Claim \ref{Thur}, $A_{11}=B_{22}=\emptyset$.
By Claim \ref{AABB}, we have $|I|\ge5$, and if $|I|=5$, then $|A_{12}|=2$ and $d_I(v)=2$ for any $v\in A_{13}\cup B_{23}$.
Take $v_1\in A_{12}$ and $v_2\in A_{13}\cup B_{23}$. Let $u_4\in N_I(v_1)\setminus\{u_1,u_2\}$. Then $x_{u_3}\ge x_{u_4}$ by the choice of $v_0$. 
Note that $G-u_4v_1+u_3v_1+u_4v_2\cong\Gamma_{n,5}'$. Then 
\[
\rho(G-u_4v_1+u_3v_1+u_4v_2)-\rho(G)\ge 2(x_{u_3}-x_{u_4})x_{v_1}+2x_{u_4}x_{v_2}>0,
\]
a contradiction.
Similarly, if $|I|=6$, we can check that $\rho(G)<\rho(\Gamma_{n,6})$, a contradiction.
And for $|I|\ge7$, we can check that $G\cong\Gamma_{n,|I|}$ as displayed in Fig. \ref{f3}.

\noindent{\bf Case 2.3.2.} $|A_{13}|=1$ and $|B_{23}|=0$, or $|A_{13}|=0$ and $|B_{23}|=1$.

Suppose that $|A_{13}|=1$ and $|B_{23}|=0$. By Claim \ref{Thur}, $A_{11}=\emptyset$.
so $B_{22}\ne\emptyset$.
Note that $d_I(v)=3$ for any vertex $v\in B_{22}$. Otherwise, let $d_I(v_{21})\le2$, then $G+u_3v_{21}\in \mathbb{G}_n$, and by Lemma \ref{addedges}, we have $\rho(G+u_3v_{21})>\rho(G)$, a contradiction.
Let $A_{13}=\{v_1\}$, and $u_4\in N_I(v_{21})\setminus \{u_2\}$ such that $u_4\in N_I(v_1)$ by Claim \ref{adj}. Then $x_{u_4}\le x_{u_2}\le x_{u_1}$ by the choice of $v_0$. 
If $|B_{22}|=1$, then $\rho(G)(x_{u_4}-x_{u_3})= x_{v_{21}}-x_{v_0}\le0$, so $x_{u_4}\le x_{u_3}$.
Let $G'= G-u_4v_{21}+u_3v_{21}$ if $|B_{22}|=1$, and $G'=G-u_4v_{21}+u_1v_{21}$ otherwise. 
Then we can check that $G'\in\mathbb{G}_n$, and by Lemma \ref{perron}, we have $\rho(G')>\rho(G)$, a contradiction.

%%Let $A_{13}=\{v_1\}$ and $u_4\in N_I(v_1)\setminus \{u_1,u_3\}$.
%Then $B_{22}\subset N_G(u_4)$ by Claim \ref{adj} and $u_4\notin N_I(v)$ for any $v\in A_{12}$ by Claim \ref{AABB}.
%%If $|B_{22}|\ge2$, then $G-u_4v_{21}+u_1v_{21}\in \mathbb{G}_n$, and by Lemma \ref{perron} in view of $x_{u_4}\le x_{u_2}\le x_{u_1}$, we have $\rho(G-u_4v_{21}+u_1v_{21})>\rho(G)$, a contradiction, so $|B_{22}|=1$.
%As $\rho(G)(x_{u_4}-x_{u_3})\le x_{v_{21}}-x_{v_0}\le0$, we have  $x_{u_4}\le x_{u_3}$.
%It is easy to see that $G-u_4v_{21}+u_3v_{21}\in \mathbb{G}_n$. By Lemma \ref{perron}, we have $\rho(G-u_4v_{21}+u_3v_{21})>\rho(G)$, a contradiction.

Suppose next that $|A_{13}|=0$ and $|B_{23}|=1$.
By Claim \ref{Thur}, $B_{22}=\emptyset$, so $A_{11}\ne\emptyset$.
Similarly as above, $d_I(v)=3$ for any vertex $v\in A_{11}$.
Let $v_2\in B_{23}$, and $u'\in N_I(v_{11})\setminus \{u_1\}$ such that $u'\in N_I(v_2)$. 
Then $u'\notin N_I(v)$ for any $v\in A_{12}$ by Claim \ref{AABB}.
Let $v_3\in A_{12}$.
Let $G''=G-u'v_{11}+u_3v_{11}$ if $x_{u_3}\ge x_{u'}$, and $G''=G-u_3v_0+u'v_0$ otherwise.
Then $G''\in \mathbb{G}_n$, and by Lemma \ref{perron}, $\rho(G'')>\rho(G)$, a contradiction.

\noindent{\bf Case 2.3.3.} $|A_{13}|=|B_{23}|=0$.

As $A_{13}=B_{23}=B_{33}=\emptyset$, we have $K\setminus N_G(u_1)=B_{22}\ne\emptyset$ and $K\setminus N_G(u_2)=A_{11}\ne\emptyset$ i.e., $K=K'\cup A_{12}\cup A_{11}\cup B_{22}$.
Let $v_1\in A_{12}$, and $u_4 \in N_I(v_1)\setminus\{u_1,u_2\}$.

\begin{Claim} \label{a11b22}
$|A_{11}|=|B_{22}|=1$.
\end{Claim}

\begin{proof} We only show that $|A_{11}|=1$, as the argument for $|B_{22}|=1$ is similar.

Suppose that $|A_{11}|\ge2$. Then $d_I(v)=3$ for any vertex $v\in A_{11}$, as otherwise it can be check that $G+u_2v'\in\mathbb{G}_n$ for some $v'\in A_{11}$ with $d_I(v')\le2$, and by Lemma \ref{addedges}, we have $\rho(G+u_2v')>\rho(G)$, a contradiction.
Let $N_I(v_{11})\setminus \{u_1\}=\{u', u''\}$ with $x_{u'}\ge x_{u''}$.
By Claim \ref{adj}, $B_{22}\subset N_G(u')$. 
And $N_G(u')\ne K\setminus\{v_0\}$ by Claim \ref{AABB} and $|A_{12}|\ge2$.
%Clearly, $x_{u_2}\ge x_{u''}$. Let $G'=G-u''v_{11}+u_2v_{11}$ if $d_G(u'')\ge2$, 
Let $G'=G-u'v_{11}+u_2v_{11}$ if $x_{u_2}\ge x_{u'}$ and $G'=G-u_2v_0+u'v_0$ otherwise. Then $G'\in\mathbb{G}_n$, and by Lemma \ref{perron}, we have $\rho(G')>\rho(G)$, a contradiction.
So $|A_{11}|=1$.
\end{proof}

It is evident that $d_I(v_{i1})\le 3$ for $i=1, 2$.

\begin{Claim} \label{FINA}
$d_I(v_{i1})=2,3$ for $i=1, 2$.
\end{Claim}

\begin{proof} Suppose that this is not true. Then  $d_I(v_{i1})=1$ for some $i=1, 2$.
Note that if $d_G(u_3)=1$, then  $0\le \rho(G)(\rho(G)+1)(x_{v_0}-x_{v_1})=\rho(G)(x_{u_3}-x_{u_4})\le x_{v_0}-x_{v_1}$, so $x_{v_0}=x_{v_1}$, and $d_G(u_4)=1$. So $G+u_3v_{i1}$ is not Hamiltonian.
Thus $G+u_3v_{i1}\in\mathbb{G}_n$, and by Lemma \ref{addedges}, $\rho(G+u_3v_{i1})>\rho(G)$, a contradiction.
\end{proof}

By Claim \ref{FINA}, $d_I(v_{i1})=2,3$ for $i=1, 2$. 
We claim that $u_4\notin N_I(v_{i1})$. Otherwise, $d_G(u_3)\ge2$ as $x_{v_0}\ge x_{v_1}$, and $G-u_4v_1+u_3v_1\in\mathbb{G}_n$, so by Lemma \ref{perron}, we have $\rho(G-u_4v_1+u_3v_1)>\rho(G)$, a contradiction. Similarly, for any $v\in A_{12}$ and $u\in N_I(v)\setminus\{u_1,u_2\}$, we have $u\notin N_I(v_{i1})$.

Suppose first that $d_I(v_{i1})=3$ for some $i=1$ or $2$. Let $N_I(v_{i1})\setminus\{u_i\}=\{u',u''\}$ with $x_{u'}\ge x_{u''}$, and $j=\{1,2\}\setminus\{i\}$. Then $u'\in N_I(v_{j1})$.
If $d_I(v_{j1})=2$, then $G+u_3v_{j1}\in\mathbb{G}_n$, and by Lemma \ref{addedges}, $\rho(G+u_3v_{j1})>\rho(G)$, a contradiction.
So $d_I(v_{j1})=3$. 
%If $N_I(v_{i1})\setminus\{u_i\}=N_I(v_{j1})\setminus\{u_j\}$, then
%$x_{v_{11}}=x_{v_{21}}$ and $|K'|\ge 2$.
%%If $d_G(u_3)=1$, then $d_G(u)=1$ for any $u\in N_I(v)\setminus\{u_1,u_2\}$ and $v\in A_{12}$.
%%Note that $\rho(G)(\rho(G)+1)(x_{v_1}-x_{v_{11}})=\rho(G)(x_{u_2}+x_{u_4}-x_{u'}-x_{u''})> 3(x_{v_1}-x_{v_{11}})$ and
%Note that $\rho(G)(x_{u'}-x_{u_3})\le 2(x_{v_{11}}-x_{v_0})\le0$. So $x_{u'}\le x_{u_3}$.
%Since $G-u'v_{11}+u_3v_{11}\in\mathbb{G}_n$, we have by Lemma \ref{perron} that $\rho(G-u'v_{11}+u_3v_{11})>\rho(G)$, a contradiction.
%%$(\rho(G)+1)\rho(G)(x_{u'}-x_{u_3})=(\rho(G)+1)(2x_{v_{11}}-x_{v_0})>(\rho(G)+1)x_{v_{11}}+x_{u'}-x_{u_2}-x_{u_3}>x_{u'}-x_{u_2}$.
%%So $x_{v_1}-x_{v_{11}}>0$ and $x_{u'}-x_{u_3}>0$.
%%Let $G'=G-u'v_{11}+u_3v_{11}$ if $d_G(u_3)\ge2$
%%, and $G'=G-u'v_{11}-u'v_{21}-u_3v_0-u_4v_1+u'v_0+u'v_1+u_3v_{11}+u_4v_{21}$ otherwise.
%%Then $G'\in\mathbb{G}_n$, and $\rho(G')>\rho(G)$ by Lemma \ref{perron} in the former case, while in the latter case, we have
%%\[
%%\rho(G')-\rho(G)\ge2(x_{u'}-x_{u_3})(x_{v_0}-x_{v_11}+x_{v_1}-x_{v_{21}})>0,
%%\]
%%also a contradiction.
By Claim \ref{NOTE}, $N_I(v_{i1})\setminus\{u_i\}\ne N_I(v_{j1})\setminus\{u_j\}$.
%Note that $u_4\notin N_I(v_{i1})$ or $N_I(v_{j1})$.
%Otherwise, let $u_4\in N_I(v_{i1})$, then $d_G(u_3)\ge2$ as $x_{v_0}\ge x_{v_1}$, and  it is easy to see that $G-u_4v_{i1}+u_3v_{i1}\in\mathbb{G}_n$, and by Lemma \ref{perron} and $x_{u_3}\ge x_{u_4}$, we have $\rho(G-u_4v_{i1}+u_3v_{i1})>\rho(G)$, a contradiction.
Now by Claim \ref{AABB}, we have $G\cong\Gamma_{n,|I|}^{**}$ with $|I|\ge8$ as displayed in Fig. \ref{flem11}.
By Lemma \ref{I4}, $\rho(G)<\rho(\Gamma_{n,|I|})$, a contradiction.

Suppose next that $d_I(v_{11})=d_I(v_{21})=2$. If $N_I(v_{11})\cap N_I(v_{21})\ne\emptyset$, then $G+u_3v_{11}\in\mathbb{G}_n$, and by Lemma \ref{addedges}, we have $\rho(G+u_3v_{11})>\rho(G)$, a contradiction.
So $N_I(v_{11})\cap N_I(v_{21})=\emptyset$.
Let $u^*\in N_I(v_{11})\setminus\{u_1\}$.
%If $d_G(u^*)=1$, then there is at least one vertex in $N_I(v)\setminus\{u_1,u_2\}$ for some $v\in A_{12}$ such that its degree is equal to $1$. So
Then $G+u^*v_{21}\in\mathbb{G}_n$, and by Lemma \ref{addedges}, we have $\rho(G+u^*v_{21})>\rho(G)$, a contradiction.
\end{proof}

%The authors thank the referees for their constructive comments and suggestions.

%\noindent {\bf Declarations}
\bigskip

\noindent {\bf Funding}

\noindent
This work was supported by the
National Natural Science Foundation of China (No.~12571364).

\bigskip

\noindent {\bf Declaration of competing interest}

\noindent
There is no competing interest.
% The authors do not have any conflict of interest or competing interests.

\bigskip

\noindent {\bf Data Availability}

\noindent
Data is available within the article.

\begin{appendices}

\section{Proof of Lemma \ref{bbb}}
\begin{proof}
Let $I=\{u_1,u_2\}$ and $V(K_{n-2})=\{v_1,\dots,v_{n-2}\}$ in $\Gamma_{n,2}'$ and $\Gamma_{n,2}''$.

We first show that $\rho(\Gamma_{n,2})>\rho(\Gamma_{n,2}'')$.
Let $N_{\Gamma_{n,2}''}(u_2)=\{v_{n-2}\}$. Then $v_i\in N_{\Gamma_{n,2}''}(u_1)$ for $i=1,\dots,n-3$. 
Denote by $\mathbf{x}$ the Perron vector of $\Gamma_{n,2}''$. 
By symmetry, $x_v=x_{v_1}$ for any $v\in N_{\Gamma_{n,2}''}(u_1)$.
Note that $\rho(\Gamma_{n,2}'')\ge\rho(K_{n-2})=n-3$. 
Since $(\rho(\Gamma_{n,2}'')+1)(x_{v_1}-x_{v_{n-2}})=x_{u_1}-x_{u_2}$ and $\rho(\Gamma_{n,2}'')(x_{u_1}-x_{u_2})=(n-3)x_{v_1}-x_{v_{n-2}}$, we have
\[
\left(\rho^2(\Gamma_{n,2}'')+\rho(\Gamma_{n,2}'')-1\right)\left(x_{v_1}-x_{v_{n-2}}\right)=(n-4)x_{v_1},
\]
so $x_{v_1}> x_{v_{n-2}}$.
Note that $\Gamma_{n,2}''-u_2v_{n-2}+u_2v_1\cong\Gamma_{n,2}$.
Thus, by Lemma \ref{perron}, we have $\rho(\Gamma_{n,2})>\rho(\Gamma_{n,2}'')$.

Next we show that $\rho(\Gamma_{n,2})>\rho(\Gamma_{n,2}')$. 
By a direct calculation, $\rho(\Gamma_{6,2})=3.7105>\rho(\Gamma_{6,2}')=3.6262$ as desired.
Suppose that $n\ge7$. 
Let $N_{\Gamma_{n,2}'}(u_1)=N_{\Gamma_{n,2}'}(u_2)=\{v_1,v_2\}$. 
Denote by $\mathbf{y}$ the Perron vector of $\Gamma_{n,2}'$. 
By symmetry, $y_1:=y_{u_1}=y_{u_2}$, $y_2:=y_{v_1}=y_{v_2}$, and denote by $y_3$ the entry of $\mathbf{y}$ at each vertex  from $\{v_3,\dots,v_{n-2}\}$. 
Note that $\rho(\Gamma_{n,2}')\ge\rho(K_{n-2})=n-3$. 
Since $(\rho(\Gamma_{n,2}')+1)(y_3-y_2)=-2y_1$ and $\rho(\Gamma_{n,2}')y_1=2y_2$, we have $\rho(\Gamma_{n,2}')(\rho(\Gamma_{n,2}')+1)(y_3-y_2)=-4y_2$,
so
\[
y_3=\left(1-\frac{4}{\rho^2(\Gamma_{n,2}')+\rho(\Gamma_{n,2}')}\right)y_2
\ge\left(1-\frac{4}{(n-2)(n-3)}\right)y_2.
\]
Note that $\Gamma_{n,2}'-u_2v_1+\{u_1v_j:j=3,\dots,n-3\}\cong\Gamma_{n,2}$. 
Thus
\begin{align*}
\rho(\Gamma_{n,2})-\rho(\Gamma_{n,2}')& \ge \mathbf{y}^\top(A(\Gamma_{n,2})-A(\Gamma_{n,2}'))\mathbf{y}=2y_1\left((n-5)y_3-y_2\right)\\
& \ge2\left(n-6-\frac{4(n-5)}{(n-2)(n-3)}\right)y_1y_2\\
& >0,
\end{align*}
as desired.
\end{proof}

\section{Proof of Lemma \ref{I3}}
\begin{proof}
By a direct calculation, we have $\rho(\Gamma_{6,3}')=3.2814>\rho(\Gamma_{6,3})=3.0478>\rho(\Gamma_{6,3}'')=3.0437$ and 
$\rho(\Gamma_{7,3})=4.1908>\rho(\Gamma_{7,3}'')=4.1747>\rho(\Gamma_{7,3}')=4.0764$.
Suppose in the following that $n\ge8$.

%Let $\pi_1:=K'\cup \{v_2\} \cup \{v_1,v_3\} \cup \{u_1,u_2\}\cup \{u_3\}$, which is an equitable for   $A(\Gamma_{n,3}'')$. Let $Q_3$ be the corresponding   quotient matrix.

First, we prove  that  $\rho(\Gamma_{n,3})>\rho(\Gamma_{n,3}')$. 
Let $Q_1$ and $Q_2$ be the quotient matrices corresponding to  proper equitable partitions  of $A(\Gamma_{n,3})$ and $A(\Gamma_{n,3}')$, respectively, where 
\[
Q_1=\begin{pmatrix}
n-5 & 1 & 2 & 0\\
n-4 & 0 & 0 & 1\\
n-4 & 0 & 0 & 0\\
0 & 1 & 0 & 0\\
\end{pmatrix}, \qquad
Q_2=\begin{pmatrix}
n-6 & 1 & 1 & 0 & 1\\
n-5 & 0 & 1 & 2 & 1\\
n-5 & 1 & 0 & 2 & 0\\
 0 & 1 & 1 & 0 & 0\\
n-5 & 1 & 0 & 0 & 0\\
\end{pmatrix}.%\qquad
%Q_3=\begin{pmatrix}
%n-7 & 1 & 2 & 2 & 0\\
%n-6 & 0 & 2 & 2 & 1\\
%n-6 & 1 & 1 & 1 & 0\\
%n-6 & 1 & 1 & 0 & 0\\
%0 & 1 & 0 & 0 & 0\\
%\end{pmatrix}.
\]
By a direct calculation, the characteristic polynomials of $Q_1$ and $Q_2$ are
\[
f(x)=x^4-(n-5)x^3-(3n-11)x^2+(n-5)x+2n-8
\]
and
\[
g(x)=x^5-(n-6)x^4-(3n-9)x^3+(2n-20)x^2+(8n-42)x+2n-8,
\]
respectively.
Note that $\rho(\Gamma_{n,3})$, $\rho(\Gamma_{n,3}')\ge \rho(K_{n-3})=n-4$.
%$f(n-3)=-n^3+12n^2-45n+52<0$ and $g(n-4)=-n^4+15n^3-88n^2+232n-224<0$.
%Together with Lemma \ref{QM}, we have $\rho(\Gamma_{n,3})=\rho(Q_1)>n-3$ and  $\rho(\Gamma_{n,3}')=\rho(Q_2)>n-4$.
For $x\ge n-4$, let
\[
\tau(x)=g(x)-xf(x)=x^4-2x^3+(n-15)x^2+(6n-34)x+2n-8.
\]
Then we can check that $\tau''(x)=12x^2-12x+2n-30\ge\tau''(n-4)=12n^2-106n+210>0$, implying that $\tau'(x)\ge\tau'(n-4)=4n^3-52n^2+208n-266>0$, so $\tau(x)\ge\tau(n-4)=n^4-17n^3+103n^2-272n+272>0$.
Thus, $g(x)>xf(x)$ for $x\ge n-4$. Together with Lemma \ref{QM}, we have $\rho(\Gamma_{n,3})>\rho(\Gamma_{n,3}')$. % for $n\ge7$.

Next we prove that $\rho(\Gamma_{n,3})>\rho(\Gamma_{n,3}'')$. 
Let $K'=N_{\Gamma_{n,3}''}(u_1)\cap N_{\Gamma_{n,3}''}(u_2)\setminus\{v_2\}$. 
Let $\mathbf{x}$ be the Perron vector of $\Gamma_{n,3}''$.
By symmetry, the entry of $\mathbf{x}$ at each vertex from $K'$ is the same, which we denote $x_1$, and let $x_2:=x_{v_1}=x_{v_3}$ and $x_3:=x_{u_1}=x_{u_2}$.
Let $G'=\Gamma_{n,3}''-u_3v_2-u_2v_3+u_2v_1+u_3v_3$.
Let $\mathbf{y}$ be the Perron vector of $G'$.
By symmetry, the entry of $\mathbf{y}$ at each vertex from $K'\cup\{v_2,v_1\}$ is the same, which we denote $y_1$, and let $y_2:=y_{u_1}=y_{u_2}$.
Note that $\rho(G')(\rho(G')+1)(y_1-y_{v_3})=\rho(G')(2y_2-y_{u_3})=2(n-4)y_1-y_{v_3}$ 
and $(\rho(\Gamma''_{n,3})+1)(x_2-x_{v_2})=-(x_3+x_{u_3})$. So 
\[
y_1-y_{v_3}=\frac{2n-9}{\rho(G')^2+\rho(G')-1}y_1 \ \ \ \ \text{and}\ \ \ \  x_2-x_{v_2}=-\frac{1}{\rho(\Gamma''_{n,3})+1}(x_3+x_{u_3}).
\]
As $\rho^2(G')y_{u_3}=\rho(G')y_{v_3}=(n-4)y_1+y_{u_3}$, $\rho(\Gamma''_{n,3})x_2=(n-6)x_1+x_2+x_{v_2}+x_3$ and $\rho(\Gamma''_{n,3})x_3=(n-6)x_1+x_2+x_{v_2}=\rho(\Gamma''_{n,3})x_2-x_3$, we have 
\[
y_{u_3}=\frac{n-4}{\rho^2(G')-1}y_1 \ \ \ \ \text{and}\ \ \ \ x_3=\frac{\rho(\Gamma''_{n,3})}{\rho(\Gamma''_{n,3})+1}x_2.
\]
Together with $\rho(\Gamma''_{n,3})x_{u_3}=x_{v_2}$, we have 
\begin{align*}
\rho(\Gamma''_{n,3})x_{v_2}& =(n-6)x_1+2x_2+2x_3+x_{u_3}=(\rho(\Gamma''_{n,3})+1)x_2-x_{v_2}+x_3+x_{u_3}\\
& =(\rho(\Gamma''_{n,3})+1)x_2-x_{v_2}+\frac{\rho(\Gamma''_{n,3})}{\rho(\Gamma''_{n,3})+1}x_2+\frac{1}{\rho(\Gamma''_{n,3})}x_{v_2},
\end{align*}
so \[
x_{v_2}=\frac{\rho(\Gamma''_{n,3})(\rho^2(\Gamma''_{n,3})+3\rho(\Gamma''_{n,3})+1)}{(\rho(\Gamma''_{n,3})+1)(\rho^2(\Gamma''_{n,3})+\rho(\Gamma''_{n,3})-1)}x_2, 
\]
i.e.,
\[
x_{u_3}=\frac{1}{\rho(\Gamma''_{n,3})}x_{v_2}=\frac{\rho^2(\Gamma''_{n,3})+3\rho(\Gamma''_{n,3})+1}{(\rho(\Gamma''_{n,3})+1)(\rho^2(\Gamma''_{n,3})+\rho(\Gamma''_{n,3})-1)}x_2. 
\]
Let $h(x)=(2n-9)x^3-2(n-4)x^2-(10n-44)x-4n+17$ with $x\ge n-4$. By a direct calculation, we have $h(x)\ge h(n-4)>0$, so  
\[
(2n-9)\rho^3(\Gamma''_{n,3})-(n-4)\rho(\Gamma''_{n,3})^2-(9n-40)\rho(\Gamma''_{n,3})-3n+13>(n-4)\left(\rho^2(\Gamma''_{n,3})+\rho(\Gamma''_{n,3})+1\right).
\]
Thus
\begin{align*}
& \quad\mathbf{y}^\top(\rho(G')-\rho(\Gamma_{n,3}''))\mathbf{x}
=\mathbf{y}^\top(A(G')-A(\Gamma_{n,3}''))\mathbf{x}\\
& =x_{u_2}y_{v_1}+y_{u_2}x_{v_1}+x_{u_3}y_{v_3}+y_{u_3}x_{v_3}-x_{u_2}y_{v_3}-y_{u_2}x_{v_3}-x_{u_3}y_{v_2}-y_{u_3}x_{v_2}\\
%& =x_{u_2}y_1+y_{u_2}x_2+x_{u_3}y_{v_3}+y_{u_3}x_2-x_{u_2}y_{v_3}-y_{u_2}x_2-x_{u_3}y_1-y_{u_3}x_{v_2}\\
& =x_3y_1+y_2x_2+x_{u_3}y_{v_3}+y_{u_3}x_2-x_3y_{v_3}-y_2x_2-x_{u_3}y_1-y_{u_3}x_{v_2}\\
& =(x_3-x_{u_3})(y_1-y_{v_3})+(x_2-x_{v_2})y_{u_3}\\
& =\frac{\rho^3(\Gamma''_{n,3})-4\rho(\Gamma''_{n,3})-1}{(\rho(\Gamma''_{n,3})+1)(\rho^2(\Gamma''_{n,3})+\rho(\Gamma''_{n,3})-1)}\cdot \frac{2n-9}{\rho^2(G')+\rho(G')-1}x_2y_1\\
& \quad -
\frac{\rho^2(\Gamma''_{n,3})+\rho(\Gamma''_{n,3})+1}{(\rho(\Gamma''_{n,3})+1)(\rho^2(\Gamma''_{n,3})+\rho(\Gamma''_{n,3})-1)}\cdot \frac{n-4}{\rho^2(G')-1}x_2y_1\\
& =\frac{\left((2n-9)\rho^3(\Gamma''_{n,3})-(n-4)\rho(\Gamma''_{n,3})^2-(9n-40)\rho(\Gamma''_{n,3})-3n+13\right)\left(\rho^2(G')-1\right)}{(\rho(\Gamma''_{n,3})+1)(\rho^2(\Gamma''_{n,3})+\rho(\Gamma''_{n,3})-1)(\rho^2(G')+\rho(G')-1)(\rho^2(G')-1)}\\
& \quad - \frac{(n-4)\left(\rho^2(\Gamma''_{n,3})+\rho(\Gamma''_{n,3})+1\right)\rho(G')}{(\rho(\Gamma''_{n,3})+1)(\rho^2(\Gamma''_{n,3})+\rho(\Gamma''_{n,3})-1)(\rho^2(G')+\rho(G')-1)(\rho^2(G')-1)}\\
& >0.
\end{align*}
Now as $G'\cong\Gamma_{n,3}$., we have $\rho(\Gamma_{n,3})>\rho(\Gamma_{n,3}'')$.
\end{proof}

\section{Proof of Lemma \ref{I44}}
\begin{proof}
By a direct calculation, $\rho(\Gamma_{8,4})=4.5722>\rho(\Gamma_{8,4}')=4.5513$.
Suppose that  $n\ge9$.
Note that $\Gamma_{n,4}'$ and $\Gamma_{n,4}$ have equal average degree $n-3-\frac{8}{n}$. So $\rho(\Gamma_{n,4}')\ge n-3-\frac{8}{n}$ and $\rho(\Gamma_{n,4})\ge n-3-\frac{8}{n}$.

Note that
\[
Q=\begin{pmatrix}
0 & n-6 & 1 & 2 & 0 & 0\\
1 & n-7 & 1 & 2 & 1 & 0\\
1 & n-6 & 0 & 0 & 1 & 1\\
1 & n-6 & 0 & 0 & 0 & 0\\
0 & n-6 & 1 & 0 & 0 & 0\\
0 & 0 & 1 & 0 & 0 & 0\\
\end{pmatrix}, \qquad
Q'=\begin{pmatrix}
0 & n-7 & 2 & 1 & 2 & 0\\
1 & n-8 & 2 & 0 & 2 & 1\\
1 & n-7 & 1 & 0 & 1 & 1\\
1 & 0 & 0 & 0 & 0 & 0 \\
1 & n-7 & 1 & 0 & 0 & 0\\
0 & n-7 & 2 & 0 & 0 & 0\\
\end{pmatrix}
\]
are  the quotient matrices of $A(\Gamma_{n,4})$ and $A(\Gamma_{n,4}')$ corresponding to proper equitable partitions, respectively.
By a direct calculation, the  characteristic polynomial of $Q$ and $Q'$ are
\[
f(x)=x^6-(n-7)x^5-5(n-5)x^4-(3n-13)x^3+4(2n-11)x^2+4(n-5)x-2n+12,
\]
and
\[
g(x)=x^6-(n-7)x^5-(5n-26)x^4-(4n-19)x^3+(6n-36)x^2+(4n-21)x-n+7
\]
respectively.
%Since $\rho(\Gamma_{n,4}')\ge n-3-\frac{8}{n}$, we have by Lemma \ref{QM} that $\rho(\Gamma_{n,4})$ is the largest root of $f(x)=0$ and $\rho(\Gamma_{n,4}')$ is the largest root of $g(x)=0$. 
For $x\ge n-3-\frac{8}{n}$, let
\[
h(x):=g(x)-f(x)=x^4-(n-6)x^3-2(n-4)x^2-x+n-5.
\]
As $h''(x)=12x^2-6(n-6)x-4(n-4)\ge h''(n-4)=2(n-4)(3n-8)>0$, we have $h'(x)=4x^3-3(n-6)x^2-4(n-4)x-1\ge h'(n-4)=(n-3)(n^2-7n+11)>0$.
So $h(x)\ge h(n-3-\frac{8}{n})=\frac{1}{n^4}(n^7-15n^6+47n^5+173n^4-744n^3-1024n^2+3072n+4096)>0$, i.e., $g(x)>f(x)$. Now by Lemma \ref{QM}, $\rho(\Gamma_{n,4})>\rho(\Gamma_{n,4}')$.
\end{proof}

\section{Proof of Lemma  \ref{I5}}

\begin{proof}
Note that $\Gamma_{n,5}'+u'v'\cong\Gamma_{n,5}''$. So by Lemma \ref{addedges}, $\rho(\Gamma_{n,5}'')>\rho(\Gamma_{n,5}')$.
By a direct calculation, we have $5.8346=\rho(\Gamma_{10,5})>\rho(\Gamma_{10,5}'')=5.8259$ and $6.9403=\rho(\Gamma_{11,5})>\rho(\Gamma_{11,5}'')=6.9314$.
Suppose that $n\ge12$.
By Lemma \ref{QM} and a direct calculation,  $\rho(\Gamma_{n,5})$ and $\rho(\Gamma_{n,5}'')$ are the largest root of $xf(x)=0$ and $xg(x)=0$, respectively, where
\[
f(x)=x^6-(n-8)x^5-(5n-28)x^4-(n+2)x^3+(16n-106)x^2+(6n-34)x-13n+91
\]
and
\[
g(x)=x^6-(n-8)x^5-(5n-29)x^4-(2n-5)x^3+(14n-100)x^2+(6n-35)x-11n+88.
\]
Since $f(n-4)<0$ and $g(n-4)<0$, we have $\rho(\Gamma_{n,5})>n-4$ and  $\rho(\Gamma_{n,5}'')>n-4$.
Assume that $x>n-4$.
Let $h(x):=g(x)-f(x)=x^4-(n-7)x^3-(2n-6)x^2-x+2n-3$.
Then we can check $h''(x)=12x^2-6(n-7)x-2(2n-6)>h''(n-4)=6n^2-34n+36>0$, implying that $h'(x)>h'(n-4)=n^3-7n^2+4n+31>0$, so $h(x)>h(n-4)=n^3-14n^2+65n-95>0$.
Hence, $g(x)>f(x)$, so $\rho(\Gamma_{n,5})>\rho(\Gamma_{n,5}'')$.
\end{proof}

\section{Proof of Lemma \ref{I4}}
\begin{proof}

For convenience, let $a=|V_4|=|V_1|$. Then $|V_4'|=|V_1'|=a-1$ and $|V_2|=|V_2'|=n-2a-6$. Let $\mathbf{x}$ and $\mathbf{y}$ be the unit Perron vector of $\Gamma_{n, |I|}^{*}$ and $\Gamma_{n, |I|}^{**}$, respectively.
By symmetry, denote $x_i$ the entry of $\mathbf{x}$ at each vertex from $V_i$ for $i=1,\dots,9$, and $y_j$ the entry of $\mathbf{y}$ at each vertex from $V_j'$ for $j=1,\dots,8$.

We first prove that $\rho(\Gamma_{n,|I|})>\rho(\Gamma_{n, |I|}^{**})$.
Fix $u_1\in V_2'$ in $\Gamma_{n, |I|}^{**}$.
Let $V_3'=\{u_2, u_3\}$, $V_6'=\{v_1\}$ and $V_7'=\{v_2\}$.
Note that $\rho(\Gamma_{n, |I|}^{**})(y_7-y_5)=y_3-(a-1)y_1-(n-2a-6)y_2$,  $\rho(\Gamma_{n, |I|}^{**})(y_8-y_6)=y_3-(n-2a-6)y_2$ and $(\rho(\Gamma_{n, |I|}^{**})+1)(y_3-y_2)=y_7-y_5+y_8-y_6$. Then 
\[
\rho(\Gamma_{n, |I|}^{**})(\rho(\Gamma_{n, |I|}^{**})+1)(y_3-y_2)=\rho(\Gamma_{n, |I|}^{**})(y_7-y_5+y_8-y_6)=2y_3-(a-1)y_1-2(n-2a-6)y_2,
\]
so
\[
(\rho(\Gamma_{n, |I|}^{**})^2+\rho(\Gamma_{n, |I|}^{**})-2)(y_3-y_2)=-(a-1)y_1-2(n-2a-7)y_2<0,
\]
where the last inequality follows as $n\ge2|I|=2(a+5)$.
Hence, $y_3-y_2<0$, so
\[
\rho(\Gamma_{n, |I|}^{**})(y_7-y_6)=2y_3-(n-2a-6)y_2<2(y_3-y_2)<0,
\]
i.e., $y_7-y_6<0$.
As $N_G(v_1)\subset N_G(u_3)$, we have $y_6=y_{v_1}<y_{u_3}=y_3$.
And by $\rho(\Gamma_{n, |I|}^{**})y_3=(a-1)y_1+(n-2a-6)y_2+y_3+y_5+y_7+y_8$, we have
\[
(\rho(\Gamma_{n, |I|}^{**})+1)(y_2-2y_3)=-(a-1)y_1-(n-2a-6)y_2-2y_3-2y_7-2y_8+y_6<0,
\]
i.e., $y_2-2y_3<0$.
Let $\Gamma'=\Gamma_{n, |I|}^{**}-u_1v_1-u_2v_2-u_3v_2+u_1v_2+u_2v_1+u_3v_1$.
Note that $U_1\cup\dots U_7$ is a equitable partition of $V(\Gamma')$, where $U_1=V_1'\cup\{u_1\}$, $U_2=V_2'\setminus\{u_1\}$, $U_4=V_4'\cup V_7'$, $U_7=V_8'$ and $U_i=V_i'$ for $i=3,5,6$. Then $|U_1|=|U_4|=a$ and $|U_2|=n-2a-7$.
Let $\mathbf{z}$ be the unit Perron vector of $\Gamma'$.
By symmetry, denote $z_i$ the entry of $\mathbf{z}$ at each vertex from $U_i$ for $i=1,\dots,7$.
Note that $\rho(\Gamma')(z_4-z_6)=z_1-(n-2a-7)z_2-2z_3$ and $(\rho(\Gamma')+1)(z_1-z_2)=z_4-z_6$. Then
\[
\rho(\Gamma')(\rho(\Gamma')+1)(z_1-z_2)=z_1-(n-2a-7)z_2-2z_3,
\]
so
\[
(\rho^2(\Gamma')+\rho(\Gamma')-1)(z_1-z_2)=-(n-2a-8)z_2-2z_3<0,
\]
which implies that $z_1-z_2<0$, and then $z_4-z_6<0$.
Since $(\rho(\Gamma')+1)(z_1-z_3)=z_5+z_4-z_6-z_7$ and $\rho(\Gamma')z_3=az_1+(n-2a-7)z_2+z_3+z_5+z_6+z_7$, we have
\[
(\rho(\Gamma')+1)(z_1-2z_3)=z_4-az_1-(n-2a-7)z_2-2z_3-2z_6-2z_7<0,
\]
implying that $z_1-2z_3<0$.
Then
\begin{align*}
\mathbf{z}^\top(\rho(\Gamma')-\rho(\Gamma_{n, |I|}^{**}))\mathbf{y}
& =z_{u_1}y_{v_2}+y_{u_1}z_{v_2}+z_{u_2}y_{v_1}+y_{u_2}z_{v_1}+z_{u_3}y_{v_1}+y_{u_3}z_{v_1}\\
& \quad-z_{u_1}y_{v_1}-y_{u_1}z_{v_1}-z_{u_2}y_{v_2}-y_{u_2}z_{v_2}-z_{u_3}y_{v_2}-y_{u_3}z_{v_2}\\
& =z_1y_7+y_2z_4+2z_3y_6+2y_3z_6-z_1y_6-y_2z_6-2z_3y_7-2y_3z_4\\
& =(y_2-2y_3)(z_4-z_6)+(y_7-y_6)(z_1-2z_3)\\
& >0,
\end{align*}
i.e., $\rho(\Gamma')>\rho(\Gamma_{n, |I|}^{**})$.
Thus $\rho(\Gamma_{n,|I|})>\rho(\Gamma_{n, |I|}^{**})$ as $\Gamma'\cong\Gamma_{n,|I|}$.

Next we prove that $\rho(\Gamma_{n,|I|})>\rho(\Gamma_{n, |I|}^*)$. 
Note that $\Gamma_{n, |I|}^*\cong \Gamma_{n, |I|}^{**}$ if $a=1$. 
Suppose in the following that $a\ge2$. 
Suppose first that $x_5<x_7$.
%We first prove that $\rho(\Gamma_{n, |I|}^*)\le\rho(\Gamma_{n, |I|}^{**})$.
%If $a=1$, then $\Gamma_{n, |I|}^*\cong \Gamma_{n, |I|}^{**}$, i.e., $\rho(\Gamma_{n, |I|}^*)=\rho(\Gamma_{n, |I|}^{**})$.
%Suppose that $a\ge2$.
%Note that for $i=4,5$,
%\[
%n-a-3=\Delta(\Gamma_i)\ge\rho(\Gamma_i)\ge\frac{1}{n}\sum_{v\in V(\Gamma_i)}d_G(v)>n-2a-5.
%\]
%
%
%
%%Suppose first that $n\ge3a+6$.
%Let $\mathbf{x}$ and $\mathbf{y}$ be the unit Perron vector of $\Gamma_{n, |I|}^{*}$ and $\Gamma_{n, |I|}^{**}$, respectively.
%By symmetry, denote $x_i$ the entry of $\mathbf{x}$ at each vertex from $V_i$ for $i=1,\dots,9$, and $y_j$ the entry of $\mathbf{y}$ at each vertex from $V_j'$ for $j=1,\dots,8$.
%As $\rho(\Gamma_{n, |I|}^*)(x_4-x_8)=x_1-(n-2a-6)x_2$ and $(\rho(\Gamma_{n, |I|}^*)+1)(x_1-x_2)=x_4+x_5-x_7-x_8$, we have
%\[
%\rho(\Gamma_{n, |I|}^*)(\rho(\Gamma_{n, |I|}^*)+1)(x_1-x_2)=(a+1)x_1-2(n-2a-6)x_2,
%\]
%so
%\[
%(\rho^2(\Gamma_{n, |I|}^*)+\rho(\Gamma_{n, |I|}^*)-a-1)(x_1-x_2)=(-2n+5a+13)x_2\le0, %a=2 取等号
%\]
%where the last inequality follows as $n\ge\frac{5}{2}|I|-6=\frac{5a+13}{2}$.
%On the other hand, as $\rho(\Gamma_{n, |I|}^*)>n-2a-5>a+1$, we have
%\[
%\rho(\Gamma_{n, |I|}^*)(x_5-x_7)=ax_1-(n-2a-6)x_2=(\rho^2(\Gamma_{n, |I|}^*)+\rho(\Gamma_{n, |I|}^*)-1)(x_1-x_2)-x_2<0,
%\]
%implying that $x_5-x_7<0$.
%%Since $\rho(\Gamma_{n, |I|}^*)>n-2a-5>a+1$, we have $x_1-x_2\le0$, and so $x_5-x_7\le0$.
Note that $\rho(\Gamma_{n, |I|}^{**})(y_7-y_5)=y_3-(a-1)y_1-(n-2a-6)y_2$, $\rho(\Gamma_{n, |I|}^{**})(y_8-y_4)=y_3-y_1$ and $(\rho(\Gamma_{n, |I|}^{**})+1)(y_3-y_1)=y_7-y_5+y_8-y_4$. We have
\[
\rho(\Gamma_{n, |I|}^{**})(\rho(\Gamma_{n, |I|}^{**})+1)(y_3-y_1)=2y_3-ay_1-(n-2a-6)y_2,
\]
so
\[
(\rho^2(\Gamma_{n, |I|}^{**})+\rho(\Gamma_{n, |I|}^{**})-2)(y_3-y_1)=(-a+2)y_1-(n-2a-6)y_2<0.
\]
Then $y_3-y_1<0$, and so $\rho(\Gamma_{n, |I|}^{**})(y_7-y_5)<y_3-y_1<0$, implying that
 $y_7-y_5<0$.
Fix $v_0\in V_1\subset V(\Gamma_{n, |I|}^*)$.
Let $V_5=\{u\}$ and $V_7=\{w\}$ in $\Gamma_{n, |I|}^*$.
As
\[
\Gamma_{n, |I|}^{**}\cong\Gamma_{n, |I|}^*-\{vu: v\in V_1\setminus\{v_0\}\}+\{vw: v\in V_1\setminus\{v_0\}\},
\]
we have
\begin{align*}
\mathbf{x}^\top(\rho(\Gamma_{n, |I|}^*)-\rho(\Gamma_{n, |I|}^{**}))\mathbf{y}&=\sum_{v\in V_1\setminus\{v_0\}}(x_vy_u+y_vx_u-x_vy_w-y_vx_w)\\
& =\sum_{v\in V_1\setminus\{v_0\}}(x_1y_7+y_1x_5-x_1y_5-y_1x_7)\\
& =\sum_{v\in V_1\setminus\{v_0\}}(x_1(y_7-y_5)+y_1(x_5-x_7))\\
& <0,
\end{align*}
i.e., $\rho(\Gamma_{n, |I|}^*)<\rho(\Gamma_{n, |I|}^{**})$.
So $\rho(\Gamma_{n, |I|})>\rho(\Gamma_{n, |I|}^{**})>\rho(\Gamma_{n, |I|}^*)$.

Suppose next that $x_5\ge x_7$.
Denote by $V_i=\{v_i\}$ for $i=3,5,\dots, 9$. Take $v_2\in V_2$.
Let $G'=\Gamma_{n, |I|}^*-v_3v_7-\{uv_7:u\in V_2\setminus\{v_2\}\}+v_3v_8+\{uv_5:u\in V_2\setminus\{v_2\}\}$.
Then $G'\cong \Gamma_{n, |I|}$.
Note that $W_1\cup\dots W_7$ is a equitable partition of $V(G')$, where  $W_2=V_2\setminus\{v_2\}$, $W_3=\{v_2,v_3\}$, $W_5=\{v_5,v_6\}$, $W_6=\{v_7,v_9\}$, $W_7=V_8$ and $W_i=V_i$ for $i=1,4$. Then $|W_1|=|W_4|=a$ and $|W_2|=n-2a-7$.
Let $\mathbf{z'}$ be the unit Perron vector of $G'$.
By symmetry, denote $z_i'$ the entry of $\mathbf{z'}$ at each vertex from $W_i$ for $i=1,\dots,7$.
Since $\rho(G')(z_5'-z_6')=az_1'+(n-2a-7)z_2'>0$ and $\rho(G')(z_7'-z_6')=(n-2a-7)z_2'+z_3'>0$, we have $z_5'-z_6'>0$ and $z_7'-z_6'>0$.
Note that $\rho(\Gamma_{n, |I|}^*)(x_7-x_8)=x_3$.
Then
\begin{align*}
& \quad \mathbf{z'}^\top\left(\rho(G')-\rho(\Gamma_{n, |I|}^{*})\right)\mathbf{x}\\
& = \sum_{v\in V_2\setminus\{v_2\}}(x_vz_{v_5}'+x_5z_v-x_vz_{v_7}'-x_7z_v)+x_3z_{v_8}'+x_8z_{v_3}'-x_3z_{v_7}'-x_7z_{v_3}'\\
& =(n-2a-7)(x_2z_5'+x_5z_2'-x_2z_6'-x_7z_2')+x_3z_7'+x_8z_3'-x_3z_6'-x_7z_3'\\
& =(n-2a-7)(x_2(z_5'-z_6')+(x_5-x_7)z_2')+x_3(z_7'-z_6')+(x_8-x_7)z_3'\\
& =(n-2a-7)(x_2(z_5'-z_6')+(x_5-x_7)z_2')+\frac{n-2a-7}{\rho(G')}x_3z_2'+\left(\frac{1}{\rho(G')}-\frac{1}{\rho(\Gamma_{n, |I|}^*)}\right)x_3z_3',
\end{align*}
i.e.,
\[
\left(\rho(G')-\rho(\Gamma_{n, |I|}^{*})\right)\left(\mathbf{z'}^\top\mathbf{x}+\frac{x_3z_3'}{\rho(G')\rho(\Gamma_{n, |I|}^*)}\right)>0,
\]
so $\rho(G')>\rho(\Gamma_{n, |I|}^*)$.
Thus $\rho(\Gamma_{n,|I|})>\rho(\Gamma_{n, |I|}^*)$ as $G'\cong\Gamma_{n,|I|}$.
\end{proof}

\end{appendices}
\end{document}